\documentclass[11pt]{amsart}
\usepackage[top=1.5in, bottom=1.5in, left=1.5in, right=1.5in]{geometry}
\expandafter\let\csname ver@amsthm.sty\endcsname\relax

\usepackage{graphicx}
\usepackage{tikz}
\usetikzlibrary{arrows}
\usetikzlibrary{decorations.markings}
\usetikzlibrary{decorations.pathmorphing}

\usepackage{amsmath}
\usepackage{amssymb}
\usepackage{enumerate}
\usepackage{mathdots}
\usepackage{mathtools,euscript}
\usepackage{youngtab}
\usepackage{hyperref}
\usepackage{amsthm}
\usepackage{dsfont,rotating}
\usepackage[capitalize,noabbrev]{cleveref}

\allowdisplaybreaks

\numberwithin{equation}{section}

\newtheorem{thm}{Theorem}[section]
\newtheorem{lemma}[thm]{Lemma}
\newtheorem{cor}[thm]{Corollary}
\newtheorem{prop}[thm]{Proposition}
\newtheorem{conj}[thm]{Conjecture}

\newtheorem{Definition}[thm]{Definition}
\newenvironment{definition}
  {\begin{Definition}\rm}{\end{Definition}}

\newtheorem{Example}[thm]{Example}
\newenvironment{example}
  {\begin{Example}\rm}{\end{Example}}

\newtheorem{Remark}[thm]{Remark}
\newenvironment{remark}
  {\begin{Remark}\rm}{\end{Remark}}

\crefname{thm}{Theorem}{Theorems}
\crefname{lemma}{Lemma}{Lemmas}
\crefname{cor}{Corollary}{Corollaries}
\crefname{prop}{Proposition}{Propositions}
\crefname{conj}{Conjecture}{Conjectures}
\crefname{question}{Question}{Questions}

\crefname{definition}{Definition}{Definitions}
\crefname{example}{Example}{Examples}
\crefname{remark}{Remark}{Remarks}

\newcommand{\emailhref}[1]{\email{\href{#1}{#1}}}

\newcommand{\dfn}[1]{\textcolor{blue}{\emph{#1}}}

\renewcommand{\binom}{\genfrac{(}{)}{0pt}{}}
\DeclareRobustCommand{\qbinom}{\genfrac{\lbrack}{\rbrack}{0pt}{}}

\definecolor{green}{HTML}{006600}

\newcommand{\rr}{\mathbb{R}}
\newcommand{\zz}{\mathbb{Z}}

\newcommand{\kk}{\mathbb{K}}

\renewcommand{\ss}{\mathfrak{S}}

\newcommand{\cala}{\mathcal{A}}

\newcommand{\calc}{\mathcal{C}}

\newcommand{\calf}{\mathcal{F}}

\newcommand{\calj}{\mathcal{J}}

\newcommand{\calo}{\mathcal{O}}
\newcommand{\calp}{\mathcal{P}}

\newcommand{\calr}{\mathcal{R}}

\newcommand{\sfa}{\mathsf{A}}
\newcommand{\sfb}{\mathsf{B}}

\newcommand{\sfp}{\mathsf{P}}
\newcommand{\sfpl}{\mathsf{P}_\mathsf{L}}
\newcommand{\sfpr}{\mathsf{P}_\mathsf{R}}	
\newcommand{\sfppl}{\mathsf{P}'_\mathsf{L}}
\newcommand{\sfppr}{\mathsf{P}'_\mathsf{R}}
\newcommand\rect[2]{\mathsf{[#1]}\times\mathsf{[#2]}}
\newcommand{\sfphat}{\widehat{\sfp}}
\newcommand{\hatz}{\widehat{0}}
\newcommand{\hato}{\widehat{1}}

\newcommand\row{\operatorname{Row}}
\newcommand\rowA{\row_{\cala}}
\newcommand\rowJ{\row_{\calj}}
\newcommand\rowF{\row_{\calf}}

\newcommand\rvac{\operatorname{Rvac}}
\newcommand\drvac{\operatorname{DRvac}}
\newcommand\rvacA{\rvac_{\cala}}
\newcommand\rvacF{\rvac_{\calf}}
\newcommand\drvacA{\drvac_{\cala}}
\newcommand\drvacF{\drvac_{\calf}}

\newcommand\rowFpl{\rowF^{\mathrm{PL}}}
\newcommand\rowApl{\rowA^{\mathrm{PL}}}
\newcommand\rvacFpl{\rvacF^{\mathrm{PL}}}
\newcommand\rvacApl{\rvacA^{\mathrm{PL}}}
\newcommand\drvacFpl{\drvacF^{\mathrm{PL}}}
\newcommand\drvacApl{\drvacA^{\mathrm{PL}}}

\newcommand\rowFb{\rowF^{\mathrm{B}}}
\newcommand\rowAb{\rowA^{\mathrm{B}}}
\newcommand\rvacFb{\rvacF^{\mathrm{B}}}
\newcommand\rvacAb{\rvacA^{\mathrm{B}}}
\newcommand\drvacFb{\drvacF^{\mathrm{B}}}
\newcommand\drvacAb{\drvacA^{\mathrm{B}}}

\newcommand\rowFbj[1]{\row_{\calf,\geq #1}^{\mathrm{B}}}
\newcommand\rowAbj[1]{\row_{\cala,\geq #1}^{\mathrm{B}}}
\newcommand\rvacFbj[1]{\rvac_{\calf,\geq #1}^{\mathrm{B}}}
\newcommand\rvacAbj[1]{\rvac_{\cala,\geq #1}^{\mathrm{B}}}

\newcommand{\ibar}{\overline{\imath}}
\newcommand{\jbar}{\overline{\jmath}}

\newcommand{\iab}{\iota^{\mathrm{B}}_{\sfa}}
\newcommand{\iapl}{\iota^{\mathrm{PL}}_{\sfa}}
\newcommand{\ibb}{\iota^{\mathrm{B}}_{\sfb}}
\newcommand{\ibpl}{\iota^{\mathrm{PL}}_{\sfb}}

\newcommand{\aut}{\operatorname{Aut}}
\newcommand{\dyck}{\operatorname{Dyck}}
\newcommand{\cat}{\operatorname{Cat}}
\newcommand{\rk}{\operatorname{rk}}
\newcommand{\mc}{\operatorname{MC}}
\newcommand{\flip}{\operatorname{Flip}}

\newcommand{\st}{\operatorname{ST}}
\newcommand{\oy}{\EuScript{Y}}
\newcommand{\oyb}{\oy^{\mathrm{B}}}
\newcommand{\down}{\nabla}
\newcommand{\up}{\Delta}

\newcommand{\h}{h}
\newcommand{\hpl}{h^{\mathrm{PL}}}
\newcommand{\hb}{h^{\mathrm{B}}}
\newcommand{\maj}{\operatorname{maj}}

\newcommand{\lk}{\operatorname{LK}}
\newcommand{\LK}{\operatorname{LK}}
\newcommand{\lkpl}{\LK^{\mathrm{PL}}}

\newcommand{\lkb}{\LK^{\mathrm{B}}}

\newcommand{\calfpl}{\calf^{\mathrm{PL}}}
\newcommand{\caljpl}{\calj^{\mathrm{PL}}}
\newcommand{\calapl}{\cala^{\mathrm{PL}}}
\newcommand{\calfb}{\calf^{\mathrm{B}}}
\newcommand{\caljb}{\calj^{\mathrm{B}}}
\newcommand{\calab}{\cala^{\mathrm{B}}}

\newcommand{\bft}{\mathbf{t}}
\newcommand{\bfT}{\mathbf{T}}
\newcommand{\bftau}{\boldsymbol{\tau}}

\newcommand\longmapsfrom{\mathrel{\reflectbox{\ensuremath{\longmapsto}}}}
\newcommand{\longmapsdown}{\text{\begin{rotate}{-90}$\longmapsto$\end{rotate}}}
\newcommand{\longmapsup}{\text{\begin{rotate}{-90}$\longmapsfrom$\end{rotate}}}

\title{The birational Lalanne--Kreweras involution}

\author{Sam Hopkins}
\address{Department of Mathematics, Howard University, Washington, DC}
\emailhref{samuelfhopkins@gmail.com}
\thanks{Sam Hopkins was supported by NSF grant \#1802920}

\author{Michael Joseph}
\address{Department of Technology and Mathematics, Dalton State College, Dalton, GA}
\emailhref{mjosephmath@gmail.com}

\keywords{Lalanne--Kreweras involution, rowvacuation, rowmotion, toggles, piecewise-linear and birational lifts, homomesy}

\begin{document}

\begin{abstract}
The Lalanne--Kreweras involution is an involution on the set of Dyck paths which combinatorially exhibits the symmetry of the number of valleys and major index statistics. We define piecewise-linear and birational extensions of the Lalanne--Kreweras involution. Actually, we show that the Lalanne--Kreweras involution is a special case of a more general operator, called rowvacuation, which acts on the antichains of any graded poset. Rowvacuation, like the closely related and more studied rowmotion operator, is a composition of toggles. We obtain the piecewise-linear and birational lifts of the Lalanne--Kreweras involution by using the piecewise-linear and birational toggles of Einstein and Propp. We show that the symmetry properties of the Lalanne--Kreweras involution extend to these piecewise-linear and birational lifts.
\end{abstract}

\maketitle

\section{Introduction} \label{sec:intro}

The starting point of our work is a certain well-known involution on the set of Dyck paths. A \dfn{Dyck path of semilength $n$} is a lattice path in $\zz^2$ with steps of the form $(1,1)$ (\dfn{up steps}) and $(1,-1)$ (\dfn{down steps}) from $(0,0)$ to $(2n,0)$ which never goes below the $x$-axis. Let $\dyck_{n}$ denote the set of Dyck paths of semilength~$n$. The number of such Dyck paths is the famous Catalan number $\cat(n) \coloneqq \frac{1}{n+1}\binom{2n}{n}$. 

A \dfn{valley} in a Dyck path is a down step which is immediately followed by an up step. Although not obvious, it is true that the number of Dyck paths in $\dyck_{n}$ with~$k$ valleys is the same as the number with $(n-1)-k$ valleys, for all $k$. The \dfn{Lalanne--Kreweras involution} is an involution on $\dyck_{n}$ which combinatorially exhibits this symmetry: it sends a Dyck path with $k$ valleys to one with $(n-1)-k$ valleys.

A related statistic to number of valleys is major index. The \dfn{major index} of a Dyck path is the sum of the positions of its valleys. Major index is an important statistic because the $q$-Catalan number $\cat(n;q) \coloneqq \frac{1}{[n+1]_q}\qbinom{2n}{n}_q$ is the generating function for Dyck paths in $\dyck_{n}$ according to their major indices. Again, although not obvious, major index is symmetrically distributed: there are as many Dyck paths in $\dyck_{n}$ with major index $k$ as with major index~$n(n-1)-k$. And again, the Lalanne--Kreweras involution combinatorially exhibits this symmetry: it sends a Dyck path with major index $k$ to one with major index~$n(n-1)-k$.

The Lalanne--Kreweras involution is described on Dyck paths in the following way.  Consider a Dyck path $D$. Draw southeast lines starting at the junctions between pairs of consecutive up steps, and draw southwest lines starting at the junctions between pairs of consecutive down steps.  There will be the same number of southeast lines as southwest lines.  Mark the intersection between the $k$th (from left-to-right) southeast line and the $k$th southwest line. The Dyck path $\LK(D)$ is the unique path (drawn upside-down) with valleys (drawn upside-down) at the marked points.  See \cref{fig:DyckLP} for an example. This involution was first considered by Kreweras~\cite{kreweras1970sur} and was later studied by Lalanne~\cite{lalanne1992involution}; in referring to it as the Lalanne--Kreweras involution we follow Callan~\cite{callan2007bijections}.

\begin{figure}
\begin{tikzpicture}[scale=2/3]
\draw[blue,semithick] (0,0) -- (1,1) -- (2,2) -- (3,1) -- (4,2) -- (5,3) -- (6,2) -- (7,1) -- (8,0) -- (9,1) -- (10,2) -- (11,1) -- (12,2) -- (13,3) -- (14,4) -- (15,3) -- (16,4) -- (17,3) -- (18,2) -- (19,1) -- (20,0);
\draw[red,semithick] (0,0) -- (1,-1) -- (2,-2) -- (3,-1) -- (4,-2) -- (5,-1) -- (6,0) -- (7,-1) -- (8,-2) -- (9,-3) -- (10,-4) -- (11,-3) -- (12,-2) -- (13,-3) -- (14,-2) -- (15,-1) -- (16,-2) -- (17,-1) -- (18,-2) -- (19,-1) -- (20,0);
\draw[dashed] (1,1) -- (3,-1) -- (6,2);
\draw[dashed] (4,2) -- (6,0) -- (7,1);
\draw[dashed] (9,1) -- (12,-2) -- (17,3);
\draw[dashed] (12,2) -- (15,-1) -- (18,2);
\draw[dashed] (13,3) -- (17,-1) -- (19,1);
\draw[red,fill] (1,-1) circle [radius=.1];
\draw[red,fill] (2,-2) circle [radius=.1];
\draw[red,fill] (3,-1) circle [radius=.1];
\draw[red,fill] (4,-2) circle [radius=.1];
\draw[red,fill] (5,-1) circle [radius=.1];
\draw[red,fill] (6,0) circle [radius=.1];
\draw[red,fill] (7,-1) circle [radius=.1];
\draw[red,fill] (8,-2) circle [radius=.1];
\draw[red,fill] (9,-3) circle [radius=.1];
\draw[red,fill] (10,-4) circle [radius=.1];
\draw[red,fill] (11,-3) circle [radius=.1];
\draw[red,fill] (12,-2) circle [radius=.1];
\draw[red,fill] (13,-3) circle [radius=.1];
\draw[red,fill] (14,-2) circle [radius=.1];
\draw[red,fill] (15,-1) circle [radius=.1];
\draw[red,fill] (16,-2) circle [radius=.1];
\draw[red,fill] (17,-1) circle [radius=.1];
\draw[red,fill] (18,-2) circle [radius=.1];
\draw[red,fill] (19,-1) circle [radius=.1];
\draw[purple,fill] (0,0) circle [radius=.1];
\draw[blue,fill] (1,1) circle [radius=.1];
\draw[blue,fill] (2,2) circle [radius=.1];
\draw[blue,fill] (3,1) circle [radius=.1];
\draw[blue,fill] (4,2) circle [radius=.1];
\draw[blue,fill] (5,3) circle [radius=.1];
\draw[blue,fill] (6,2) circle [radius=.1];
\draw[blue,fill] (7,1) circle [radius=.1];
\draw[blue,fill] (8,0) circle [radius=.1];
\draw[blue,fill] (9,1) circle [radius=.1];
\draw[blue,fill] (10,2) circle [radius=.1];
\draw[blue,fill] (11,1) circle [radius=.1];
\draw[blue,fill] (12,2) circle [radius=.1];
\draw[blue,fill] (13,3) circle [radius=.1];
\draw[blue,fill] (14,4) circle [radius=.1];
\draw[blue,fill] (15,3) circle [radius=.1];
\draw[blue,fill] (16,4) circle [radius=.1];
\draw[blue,fill] (17,3) circle [radius=.1];
\draw[blue,fill] (18,2) circle [radius=.1];
\draw[blue,fill] (19,1) circle [radius=.1];
\draw[purple,fill] (20,0) circle [radius=.1];
\draw[dashed] (-0.3,0) -- (20.3,0);
\draw (0,-0.2) -- (0,0.2);
\draw (1,-0.2) -- (1,0.2);
\draw (2,-0.2) -- (2,0.2);
\draw (3,-0.2) -- (3,0.2);
\draw (4,-0.2) -- (4,0.2);
\draw (5,-0.2) -- (5,0.2);
\draw (6,-0.2) -- (6,0.2);
\draw (7,-0.2) -- (7,0.2);
\draw (8,-0.2) -- (8,0.2);
\draw (9,-0.2) -- (9,0.2);
\draw (10,-0.2) -- (10,0.2);
\draw (11,-0.2) -- (11,0.2);
\draw (12,-0.2) -- (12,0.2);
\draw (13,-0.2) -- (13,0.2);
\draw (14,-0.2) -- (14,0.2);
\draw (15,-0.2) -- (15,0.2);
\draw (16,-0.2) -- (16,0.2);
\draw (17,-0.2) -- (17,0.2);
\draw (18,-0.2) -- (18,0.2);
\draw (19,-0.2) -- (19,0.2);
\draw (20,-0.2) -- (20,0.2);
\end{tikzpicture}
\caption{A Dyck path $D$ of semilength 10 in {\color{blue} blue} together with $\LK(D)$ drawn upside-down in {\color{red} red}.} \label{fig:DyckLP}
\end{figure}
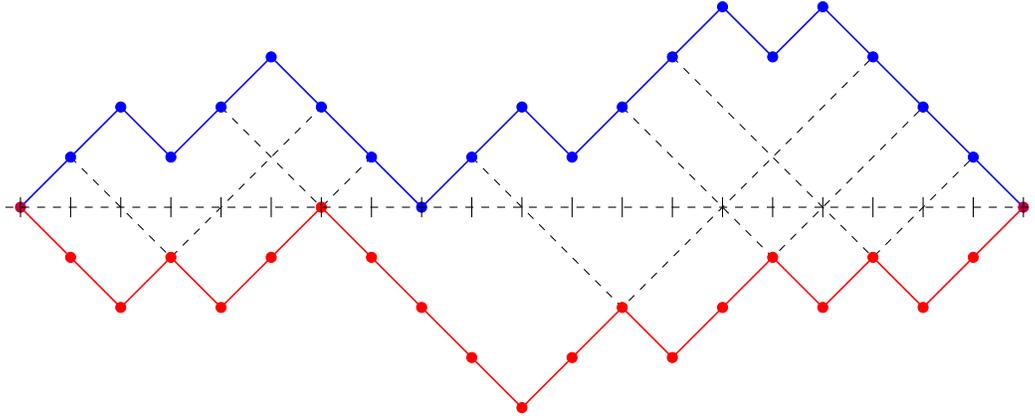

\begin{figure}
\begin{tikzpicture}[scale=2/3]
\draw[thick] (-0.15, 1.85) -- (-0.85, 1.15);
\draw[thick] (0.15, 1.85) -- (0.85, 1.15);
\draw[thick] (-1.15, 0.85) -- (-1.85, 0.15);
\draw[thick] (-0.85, 0.85) -- (-0.15, 0.15);
\draw[thick] (0.85, 0.85) -- (0.15, 0.15);
\draw[thick] (1.15, 0.85) -- (1.85, 0.15);
\draw[thick] (0.85, -0.85) -- (0.15, -0.15);
\draw[thick] (1.15, -0.85) -- (1.85, -0.15);
\draw[thick] (2.85, -0.85) -- (2.15, -0.15);
\draw[thick] (-0.85, -0.85) -- (-0.15, -0.15);
\draw[thick] (-1.15, -0.85) -- (-1.85, -0.15);
\draw[thick] (-2.85, -0.85) -- (-2.15, -0.15);
\draw [thick] (0,2) circle [radius=0.2];
\node at (0,2+0.5) {\tiny $[1,4]$};
\draw [thick] (-1,1) circle [radius=0.2];
\node at (-1,1+0.5) {\tiny $[1,3]$};
\draw [thick] (1,1) circle [radius=0.2];
\node at (1,1+0.5) {\tiny $[2,4]$};
\draw [thick] (-2,0) circle [radius=0.2];
\node at (-2,0+0.5) {\tiny $[1,2]$};
\draw [thick,fill] (0,0) circle [radius=0.2];
\node at (0,0+0.5) {\tiny $[2,3]$};
\draw [thick,fill] (2,0) circle [radius=0.2];
\node at (2,0+0.5) {\tiny $[3,4]$};
\draw [thick] (-3,-1) circle [radius=0.2];
\node at (-3,-1+0.5) {\tiny $[1,1]$};
\draw [thick] (-1,-1) circle [radius=0.2];
\node at (-1,-1+0.5) {\tiny $[2,2]$};
\draw [thick] (1,-1) circle [radius=0.2];
\node at (1,-1+0.5) {\tiny $[3,3]$};
\draw [thick] (3,-1) circle [radius=0.2];
\node at (3,-1+0.5) {\tiny $[4,4]$};
\draw[dashed] (-6,-2) -- (6,-2);
\draw (-5,-2.2) -- (-5,-1.8);
\draw (-4,-2.2) -- (-4,-1.8);
\draw (-3,-2.2) -- (-3,-1.8);
\draw (-2,-2.2) -- (-2,-1.8);
\draw (-1,-2.2) -- (-1,-1.8);
\draw (0,-2.2) -- (0,-1.8);
\draw (1,-2.2) -- (1,-1.8);
\draw (2,-2.2) -- (2,-1.8);
\draw (3,-2.2) -- (3,-1.8);
\draw (4,-2.2) -- (4,-1.8);
\draw (5,-2.2) -- (5,-1.8);
\draw[blue] (-5,-2) -- (-4,-1) -- (-3,0) -- (-2,1) -- (-1,0) -- (0,-1) -- (1,0) -- (2,-1) -- (3,0) -- (4,-1) -- (5,-2);
\draw[blue,fill] (-5,-2) circle [radius=0.09];
\draw[blue,fill] (-4,-1) circle [radius=0.09];
\draw[blue,fill] (-3,0) circle [radius=0.09];
\draw[blue,fill] (-2,1) circle [radius=0.09];
\draw[blue,fill] (-1,0) circle [radius=0.09];
\draw[blue,fill] (0,-1) circle [radius=0.09];
\draw[blue,fill] (1,0) circle [radius=0.09];
\draw[blue,fill] (2,-1) circle [radius=0.09];
\draw[blue,fill] (3,0) circle [radius=0.09];
\draw[blue,fill] (4,-1) circle [radius=0.09];
\draw[blue,fill] (5,-2) circle [radius=0.09];
\end{tikzpicture}
\caption{The bijection between $\dyck_{n+1}$ and $\cala(\sfa^n)$.} \label{fig:dyck_bij}
\end{figure}

We prefer to describe the Lalanne--Kreweras involution using the language of partially ordered sets (posets). For $\sfp$ a poset, we use $\cala(\sfp)$ to denote the set of antichains of $\sfp$. We use the standard notations $[a,b] \coloneqq \{a,a+1,\ldots,b\}$ for intervals, and $[n] \coloneqq [1,n]$. Let $\sfa^n$ denote the poset whose elements are the non-empty intervals~$[i,j] \subseteq [n]$ for $1\leq i \leq j \leq n$, ordered by containment. (This poset is isomorphic to the \dfn{root poset of the Type A root system}, hence the name.) There is a standard bijection between $\dyck_{n+1}$ and $\cala(\sfa^n)$ which is depicted in \cref{fig:dyck_bij}. Under this bijection, the number of valleys of the Dyck path becomes the cardinality of the antichain, and major index becomes $\maj(A) \coloneqq \sum_{[i,j]\in A} (i+j)$.

A set of intervals $A=\{[i_1,j_1],\ldots,[i_k,j_k]\}$ is an antichain  of $\sfa^n$ if and only if:
\begin{itemize}
    \item $1\leq i_1 < \cdots < i_k \leq n$,
    \item $1\leq j_1 < \cdots < j_k \leq n$,
    \item and $i_\ell \leq j_\ell$ for all $1\leq \ell \leq k$.
\end{itemize} 
The \dfn{Lalanne--Kreweras involution} $\lk\colon \cala(\sfa^n)\to\cala(\sfa^n)$, thought of as an involution on antichains via the bijection depicted in \cref{fig:dyck_bij}, sends such an antichain to $\lk(A) \coloneqq \{[i'_1,j'_1],\ldots,[i'_m,j'_m]\}$, where
\begin{align*}
\{i'_1 < \cdots < i'_m\} &\coloneqq [n] \setminus \{j_1,\ldots,j_k\}, \\
\{j'_1< \cdots <j'_m\} &\coloneqq [n] \setminus \{i_1,\ldots,i_k\}.
\end{align*}
As an example, the {\color{blue}blue} Dyck path in \cref{fig:DyckLP} corresponds to the four element antichain $A=\{[1,2], [4,4], [5,6], [6,9]\}$ in $\cala(\sfa^9)$. Here $\{i'_1, \ldots, i'_5\}=\{1,3,5,7,8\}$ and $\{j'_1, \ldots, j'_5\}=\{2,3,7,8,9\}$. So we have $\LK(A) = \{[1,2], [3,3], [5,7], [7,8], [8,9]\}$, which indeed corresponds to the {\color{red}red} Dyck path in the figure.

It is straightforward to verify that $\lk(A)$ is an antichain of $\sfa^n$, so that this does define an involution $\lk\colon \cala(\sfa^n)\to\cala(\sfa^n)$. And it is clear that $\#A+\#\lk(A) = n$ and $\#\maj(A) + \#\maj(\lk(A))=n(n+1)$ for all $A\in \cala(\sfa^n)$.  This antichain definition appeared in work of Panyushev~\cite{panyushev2004adnilpotent, panyushev2009orbits}, who was apparently unaware that this involution had previously been considered.\footnote{The first author learned that Panyushev's involution was the same as the Lalanne--Kreweras involution at a talk (\url{http://fpsac2019.fmf.uni-lj.si/resources/Slides/205slides.pdf})  about the FindStat project (\url{http://www.findstat.org/}) given by Martin Rubey at the FPSAC 2019 conference.}

To sketch the proof that this antichain definition agrees with the usual Dyck path definition of the Lalanne--Kreweras involution, consider a Dyck path $D$ with corresponding antichain $A$.  Label the up steps from left to right.  Then $j,j+1$ is a pair of consecutive up steps if and only if there is no element of the form $[i,j]$ in $A$.  Likewise label the down steps from left to right.  Then $i,i+1$ is a pair of consecutive down steps if and only if there is no element of the form $[i,j]$ in $A$.  Thus, there is an intersection of the southeast line at the junction of up steps $j,j+1$ and the southwest line at the junction of down steps $i,i+1$ (in other words, there is a valley at this intersection in~$\lk(D)$) precisely when~$[j,i]\in \lk(A)$.

\medskip

In this paper we define \dfn{piecewise-linear} and \dfn{birational extensions of the Lalanne--Kreweras involution}. Let us briefly explain what this means. 

For $\sfp$ a poset we use~$\rr^\sfp$ to denote the vector space of real-valued functions on~$\sfp$. The \dfn{chain polytope}~$\calc(\sfp)$ of $\sfp$ is the polytope of points $\pi \in \rr^\sfp$ satisfying the inequalities
\[0 \leq \sum_{x \in C} \pi(x) \leq 1 \qquad \textrm{for any chain $C=\{x_1 < \cdots < x_k\}\subseteq \sfp$}.\]
Stanley~\cite{stanley1986twoposet} showed that the vertices of $\calc(\sfp)$ are the indicator functions of the antichains $A\in \cala(\sfp)$, and so we may identify these vertices with antichains.

The \dfn{tropicalization} of a subtraction-free rational expression is the result of replacing $+$'s by $\max$'s and $\times$'s by $+$'s everywhere in this expression; it defines a continuous and piecewise-linear map. If the rational expression is defined on, e.g., $\rr_{>0}^{\sfp}$ (the set of positive real-valued functions on $\sfp$), then its tropicalization will be defined on, e.g., $\rr^{\sfp}$.

Our piecewise-linear and birational extensions of the Lalanne--Kreweras involution are the maps $\lkpl$ and $\lkb$ described in the following theorem.

\begin{thm} \label{thm:main_intro}
Let $\kappa \in \rr_{>0}$ be a parameter. There exists a map $\lkb\colon \rr_{>0}^{\sfa^n} \to \rr_{>0}^{\sfa^n}$ defined by a subtraction-free rational expression for which:
\begin{itemize}
    \item $\lkb$ is an involution;
    \item for any $\pi \in \rr_{>0}^{\sfa^n}$,
\begin{align*}
\prod_{[i,j] \in \sfa^n} \pi([i,j]) \cdot \prod_{[i,j] \in \sfa^n} \lkpl(\pi) ([i,j]) &= \kappa^{n}, \\
\prod_{[i,j] \in \sfa^n} \pi([i,j])^{i+j} \cdot \prod_{[i,j] \in \sfa^n} \lkpl(\pi)([i,j])^{i+j} &= \kappa^{n(n+1)}.
\end{align*}
\end{itemize}
Its tropicalization is a piecewise-linear map $\lkpl \colon \rr^{\sfa^n} \to\rr^{\sfa^n}$. In turn, $\lkpl$ restricts to a map on the chain polytope $\calc(\sfa^{n})$ and recovers the combinatorial Lalanne--Kreweras involution $\lk\colon \cala(\sfa^{n})\to \cala(\sfa^{n})$ when restricted to the vertices of $\calc(\sfa^{n})$.
\end{thm}

\Cref{fig:lk_example} depicts these maps in the case $n=3$. The reader is encouraged to verify that the map $\lkb\colon \rr_{>0}^{\sfa^3} \to\rr_{>0}^{\sfa^3}$ depicted there satisfies the conditions of \cref{thm:main_intro}. Also observe how $\lkpl\colon \rr^{\sfa^3} \to\rr^{\sfa^3}$ is the tropicalization of $\lkb$ (note that the parameter $\kappa \in \rr_{>0}$ which appears in the definition of $\lkb$ becomes the constant~$1$ when we tropicalize). Finally, the reader can check that $\lkpl$ restricts to the appropriate map on $\calc(\sfa^{n})$.

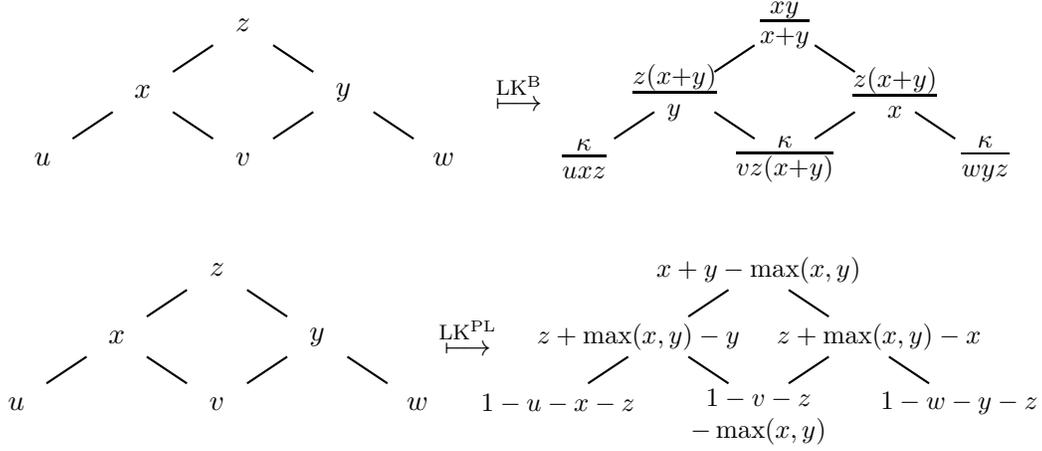
\begin{figure}
\begin{tikzpicture}[xscale=4/3,yscale=8/9]
\begin{scope}
\draw[thick] (-0.3, 1.7) -- (-0.7, 1.3);
\draw[thick] (0.3, 1.7) -- (0.7, 1.3);
\draw[thick] (-1.3, 0.7) -- (-1.7, 0.3);
\draw[thick] (-0.7, 0.7) -- (-0.3, 0.3);
\draw[thick] (0.7, 0.7) -- (0.3, 0.3);
\draw[thick] (1.3, 0.7) -- (1.7, 0.3);
\node at (0,2) {$  z  $};
\node at (-1,1) {$  x  $};
\node at (1,1) {$  y  $};
\node at (-2,0) {$  u  $};
\node at (0,0) {$  v  $};
\node at (2,0) {$  w  $};
\node at (2.75,1) {$\stackrel{\lkb}{\longmapsto}$};
\end{scope}
\begin{scope}[shift={(5.4,0)}]
\draw[thick] (-0.3, 1.7) -- (-0.7, 1.3);
\draw[thick] (0.3, 1.7) -- (0.7, 1.3);
\draw[thick] (-1.3, 0.7) -- (-1.7, 0.3);
\draw[thick] (-0.7, 0.7) -- (-0.3, 0.3);
\draw[thick] (0.7, 0.7) -- (0.3, 0.3);
\draw[thick] (1.3, 0.7) -- (1.7, 0.3);
\node at (0,2) {\Large $  \frac{xy}{x+y} $};
\node at (-1.1,1) {\Large $  \frac{z(x+y)}{y}  $};
\node at (1.1,1) {\Large $  \frac{z(x+y)}{x}  $};
\node at (-2,0) {\Large $ \frac{\kappa}{uxz}  $};
\node at (0,0) {\Large $  \frac{\kappa}{vz(x+y)}  $};
\node at (2,0) {\Large $  \frac{\kappa}{wyz} $};
\end{scope}
\end{tikzpicture}

\vspace{20pt}

\begin{tikzpicture}[xscale=4/3,yscale=8/9]
\begin{scope}
\draw[thick] (-0.3, 1.7) -- (-0.7, 1.3);
\draw[thick] (0.3, 1.7) -- (0.7, 1.3);
\draw[thick] (-1.3, 0.7) -- (-1.7, 0.3);
\draw[thick] (-0.7, 0.7) -- (-0.3, 0.3);
\draw[thick] (0.7, 0.7) -- (0.3, 0.3);
\draw[thick] (1.3, 0.7) -- (1.7, 0.3);
\node at (0,2) {$  z  $};
\node at (-1,1) {$  x  $};
\node at (1,1) {$  y  $};
\node at (-2,0) {$  u  $};
\node at (0,0) {$  v  $};
\node at (2,0) {$  w  $};
\node at (2.5,1) {$\stackrel{\lkpl}{\longmapsto}$};
\end{scope}
\begin{scope}[shift={(5.4,0)}]
\draw[thick] (-0.3, 1.7) -- (-0.7, 1.3);
\draw[thick] (0.3, 1.7) -- (0.7, 1.3);
\draw[thick] (-1.3, 0.7) -- (-1.7, 0.3);
\draw[thick] (-0.7, 0.7) -- (-0.3, 0.3);
\draw[thick] (0.7, 0.7) -- (0.3, 0.3);
\draw[thick] (1.3, 0.7) -- (1.7, 0.3);
\node at (0,2) {\small $  x+y-\max(x,y)  $};
\node at (-1.2,1) {\small $  z+\max(x,y)-y  $};
\node at (1.2,1) {\small $  z+\max(x,y)-x  $};
\node at (-2,0) {\small $  1-u-x-z  $};
\node at (0,-0.2) {\small \parbox{0.8in}{\begin{center}$  1-v-z$ \\$-\max(x,y)  $\end{center}}};
\node at (2,0) {\small $  1-w-y-z $};
\end{scope}
\end{tikzpicture}
\caption{The piecewise-linear and birational lifts of the Lalanne--Kreweras involution for $\sfa^3$.} \label{fig:lk_example}
\end{figure}

Observe how the second bulleted item in \cref{thm:main_intro} is the birational analog of the fact that the Lalanne--Kreweras involution exhibits the symmetry of the antichain cardinality and major index statistics. Thus, our piecewise-linear and birational extensions retain the key features of $\lk$, namely: being an involution, and exhibiting these symmetries.

\medskip

Recently, there has been a great deal of interest in studying piecewise-linear and birational extensions of constructions from algebraic combinatorics. Indeed, these piecewise-linear and birational maps are at the core of the growing subfield of dynamical algebraic combinatorics~\cite{roby2016dynamical}. Our work fits squarely into this research program. 

The combinatorial operator whose piecewise-linear and birational lifts have received the most attention is rowmotion. \dfn{Rowmotion}, $\rowA\colon \cala(\sfp)\to\cala(\sfp)$, is the invertible operator on the set of antichains of a poset $\sfp$ defined by
\[\rowA(A) \coloneqq \down(\{x\in \sfp\colon \textrm{$x \not\leq y$ for any $y\in A$}\})\] 
for all $A\in\cala(\sfp)$, where $\down(X)$ denotes the set of minimal elements of $X$. 

It is known \cite{cameron1995orbits, joseph2019antichain} that rowmotion can alternatively be defined as a composition of toggles. \dfn{Toggles} are certain simple, local involutions which ``toggle'' the status of an element in a set when possible. This toggle perspective turns out to be very useful for analyzing the behavior of rowmotion. Moreover, in 2013 Einstein and Propp~\cite{einstein2018combinatorial} introduced piecewise-linear and birational extensions of the toggles, and, with these, piecewise-linear and birational extensions of rowmotion.

Our first step towards defining the piecewise-linear and birational extensions of the Lalanne--Kreweras involution is to show that $\lk\colon \cala(\sfa^n)\to\cala(\sfa^n)$ can be written as a composition of toggles. Actually, we show that $\lk$ is a special case of a more general construction. 

For any graded poset $\sfp$, \dfn{rowvacuation}, $\rvacA\colon\cala(\sfp)\to\cala(\sfp)$, is another map on antichains defined as a certain composition of toggles. Rowmotion and rowvacuation are ``partner'' operators in exactly the same way that promotion and evacuation are ``partner'' operators. We recall that promotion and evacuation are two operators on the set of linear extensions of a poset which were first defined and studied by Sch\"{u}tzenberger~\cite{schutzenberger1972promotion}. The same basic facts about promotion and evacuation hold for rowmotion and rowvacuation:  rowvacuation is always an involution, just like evacuation is; rowvacuation conjugates rowmotion to its inverse, just like evacuation does for promotion; etc. This connection explains the name ``rowvacuation.''

We show that, in the case $\sfp=\sfa^n$, rowvacuation is precisely the Lalanne--Kreweras involution. This gives us natural candidates for $\lkpl$ and $\lkb$, where we simply replace the toggles in the definition of rowvacuation with their piecewise-linear and birational extensions. General properties of the toggles imply that these $\lkpl$ and $\lkb$ remain involutions.

Then the final thing is to establish the piecewise-linear and birational analogs of the fact that the Lalanne--Kreweras involution exhibits the symmetry of the antichain cardinality and major index statistics. Results of this kind have also been a focus of recent research in dynamical algebraic combinatorics. More precisely, if~$\varphi$ is an invertible operator acting on a combinatorial set $X$, and $f\colon X\to \rr$ is some statistic on $X$, then we say that $f$ is \dfn{homomesic} with respect to the action of $\varphi$ on $X$ if the average of $f$ along every $\varphi$-orbit is equal to the same constant. 

With this terminology, we can say that the antichain cardinality and major index statistics are homomesic under the Lalanne--Kreweras involution. In fact, there is a broader collection of homomesies for $\lk$. For $1\leq i \leq n$, define $\h_i\colon \cala(\sfa^n)\to \zz$ by
\[\h_i(A) \coloneqq \#\{j\colon [i,j] \in A\} + \#\{j\colon [j,i]\in A\}.\]
It is easily seen that $\h_i(A)+\h_i(\lk(A))=2$ for all $A\in\cala(\sfa^{n})$, i.e., that the average of $\h_i$ along any $\lk$-orbit is $1$. Furthermore, we have
\begin{align*}
\#A &= \frac{1}{2}(\h_1(A) + \h_2(A) + \cdots + \h_n(A)), \\
\maj(A) &= \h_1(A) + 2 \cdot \h_2(A) + \cdots + n \cdot \h_n(A).
\end{align*}
Any linear combination of homomesies is again a homomesy. Thus, the  antichain cardinality and major index homomesies for $\lk$ follow from the $\h_i$ homomesies. 

We show that the (piecewise-linear and birational analogs of) the $\h_i$ homomesies extend to $\lkpl$ and $\lkb$. We do this via a careful analysis of a certain embedding of the triangle-shaped poset $\sfa^{n}$ into the rectangle poset $\rect{n+1}{n+1}$. From now on we will not separately emphasize the antichain cardinality and major index statistics, and instead focus on the more general $\h_i$ statistics.

\medskip

Here is the outline of the rest of the paper. In \cref{sec:basics} we review rowmotion, toggling, and rowvacuation for arbitrary graded posets. Rowvacuation was first defined, briefly, in~\cite{hopkins2020order}. We spend more time explaining its basic properties here. Also, rowvacuation was previously defined in its order filter variant $\rvacF$; we need the antichain variant of rowvacuation, so we review in depth the translation between these two. In \cref{sec:lk_is_rvac}, we prove that rowvacuation for the poset $\sfa^n$ is the Lalanne--Kreweras involution. We do this by showing that they both satisfy the same recurrence. In \cref{sec:homomesies} we establish the homomesies for piecewise-linear and birational rowvacuation of $\sfa^n$. There are two main ingredients to our proof: a rowmotion-equivariant embedding of $\rr_{>0}^{\sfa^n}$ into $\rr_{>0}^{\rect{n+1}{n+1}}$ due to Grinberg and Roby~\cite{grinberg2015birational2}; and a result of Roby and the second author~\cite{joseph2021birational} which says that under rowmotion of the rectangle, a certain associated vector, called the ``Stanley--Thomas word,'' rotates. In \cref{sec:bn} we consider the poset $\sfb^n$, the \dfn{root poset of the Type~B root system}, which is obtained from $\sfa^{2n-1}$ by ``folding'' it along its vertical axis of symmetry. In \cref{sec:enumeration} we discuss some related enumeration: counting fixed points of the various operators we consider here. Finally, in \cref{sec:future} we briefly discuss some directions for future research.

\medskip

\noindent {\bf Acknowledgments}:
The authors are grateful for useful conversations with Cole Cash, David Einstein, Sergi Elizalde, Darij Grinberg, Chandler Keith, Matthew Plante, James Propp, Vic Reiner, Tom Roby, Jessica Striker, and Nathan Williams. Computations in Sage~\cite{sage} have also been invaluable. We are also grateful to the organizers of the 2020 BIRS online workshop on Dynamical Algebraic Combinatorics for giving us a chance to present this work while it was in preparation. Finally, we thank the anonymous referees, whose comments improved the exposition of this paper.

\section{Rowmotion, toggling, and rowvacuation: definitions and basics} \label{sec:basics}

In this section we review the basics concerning rowmotion, toggling, and rowvacuation, including their piecewise-linear and birational extensions. 

We assume familiarity with the standard terms and notations associated with posets, as discussed for instance in \cite[Ch.~3]{stanley2011ec1}. All the results in this section will hold for any finite\footnote{Throughout all posets will be finite, and we will drop this adjective from now on.} graded poset, not just the poset $\sfa^n$ relevant to the discussion in \cref{sec:intro}. So throughout this section, $\sfp$ will denote a \dfn{graded} poset of rank $r$, that is, a poset $\sfp$ with a \dfn{rank function} $\rk: \sfp\to \zz_{\geq 0}$ satisfying
\begin{itemize}
\item $\rk(x)=0$ for any minimal element $x$;
\item $\rk(y)=\rk(x)+1$ if $x \lessdot y$;
\item every maximal element $x$ has $\rk(x)=r$.
\end{itemize}
For example, $\sfa^n$ is a graded poset of rank $n-1$, with $\rk([i,j]) = j-i$ for $[i,j] \in \sfa^n$. For $0\leq i \leq r$, we use $\sfp_i \coloneqq \{p \in \sfp: \rk(p)=i\}$ to denote the $i$th \dfn{rank} of $\sfp$.

We will constantly work with the following three families of subsets of posets.

\begin{itemize}
\item An \dfn{order filter} (resp.\ \dfn{order ideal}) of $\sfp$ is a subset $F\subseteq \sfp$ such that if $x\in F$ and $y\geq x$ (resp.\ $y\leq x$) in $\sfp$, then $y\in F$.  We use $\calf(\sfp)$ and $\calj(\sfp)$ to denote the sets of order filters and order ideals of $\sfp$, respectively.
\item An \dfn{antichain} of $\sfp$ is a subset $A \subseteq \sfp$ in which any two elements are incomparable.  We denote the set of antichains of $\sfp$ by $\cala(\sfp)$.
\end{itemize}

We proceed to define the various operators on these sets. Because we are interested in both rowmotion and rowvacuation, in both their order filter and antichain incarnations, and at the combinatorial, the piecewise-linear, and birational levels, we have a total of $2\times 2 \times 3=12$ maps to discuss. In order to avoid duplication when explaining the basic properties of these maps, we will give proofs only at the birational level (which is the most general).

\subsection{Rowmotion}

\dfn{Rowmotion} is an invertible operator that is defined on~$\calf(\sfp)$, or equivalently on $\cala(\sfp)$.  Each rowmotion map can be described in two ways. The first is as a composition of the following three bijections:
\begin{itemize}
\item \dfn{complementation} $\Theta\colon 2^\sfp \to 2^\sfp$, where $\Theta(S) \coloneqq \sfp\setminus  S$ (so $\Theta$ sends order ideals to order filters and vice versa);
\item \dfn{up-transfer} $\up\colon \calj(\sfp) \to \cala(\sfp)$, where $\up(I)$ denotes the set of maximal elements of $I$;
\item \dfn{down-transfer} $\down\colon \calf(\sfp) \to \cala(\sfp)$, where $\down(F)$ denotes the set of minimal elements of $F$.
\end{itemize}
Evidently, $\Theta^{-1}=\Theta$. Also note that, for an antichain $A\in\cala(\sfp)$,
\[\up^{-1}(A)=\{x\in \sfp\colon x\leq y \text{ for some } y\in A\}\]
and similarly
\[\down^{-1}(A)=\{x\in \sfp\colon x\geq y \text{ for some } y\in A\}.\]

\begin{definition} 
\dfn{Order filter rowmotion}, denoted $\rowF: \calf(\sfp) \to \calf(\sfp)$, is given by $\rowF \coloneqq \Theta \circ \up^{-1} \circ \down$. 
\end{definition}

\begin{definition}
\dfn{Antichain rowmotion}, denoted $\rowA: \cala(\sfp) \to \cala(\sfp)$, is given by $\rowA \coloneqq \down \circ \Theta \circ \up^{-1}$.
\end{definition}

Rowmotion was first considered by Brouwer and Schrijver~\cite{brouwer1974period} and has had several names in the literature; however, the name ``rowmotion,'' due to Striker and Williams~\cite{striker2012promotion}, seems to have stuck. For more on the history of rowmotion see~\cite{striker2012promotion} and~\cite[\S 7]{thomas2019rowmotion}.

\begin{example}\label{ex:a3-row}
Below we demonstrate one application of $\rowF$ and $\rowA$ for the poset $\sfa^3$:
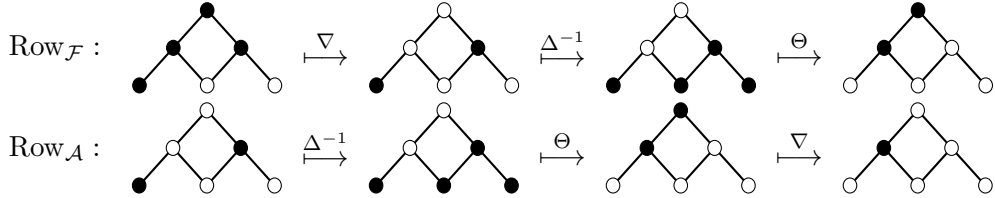
\begin{center}
\begin{tikzpicture}[xscale=0.45,yscale=0.5]
\node at (-4.5,-3) {$\rowF:$};
\begin{scope}[shift={(0,-4)}]
\draw[thick] (-0.1, 1.9) -- (-0.9, 1.1);
\draw[thick] (0.1, 1.9) -- (0.9, 1.1);
\draw[thick] (-1.1, 0.9) -- (-1.9, 0.1);
\draw[thick] (-0.9, 0.9) -- (-0.1, 0.1);
\draw[thick] (0.9, 0.9) -- (0.1, 0.1);
\draw[thick] (1.1, 0.9) -- (1.9, 0.1);
\draw[fill=black] (0,2) circle [radius=0.2];
\draw[fill=black] (-1,1) circle [radius=0.2];
\draw[fill=black] (1,1) circle [radius=0.2];
\draw[fill=black] (-2,0) circle [radius=0.2];
\draw[fill=white] (0,0) circle [radius=0.2];
\draw[fill=white] (2,0) circle [radius=0.2];
\end{scope}
\node at (3.5,-3) {$\stackrel{\down}{\longmapsto}$};
\begin{scope}[shift={(7,-4)}]
\draw[thick] (-0.1, 1.9) -- (-0.9, 1.1);
\draw[thick] (0.1, 1.9) -- (0.9, 1.1);
\draw[thick] (-1.1, 0.9) -- (-1.9, 0.1);
\draw[thick] (-0.9, 0.9) -- (-0.1, 0.1);
\draw[thick] (0.9, 0.9) -- (0.1, 0.1);
\draw[thick] (1.1, 0.9) -- (1.9, 0.1);
\draw[fill=white] (0,2) circle [radius=0.2];
\draw[fill=white] (-1,1) circle [radius=0.2];
\draw[fill=black] (1,1) circle [radius=0.2];
\draw[fill=black] (-2,0) circle [radius=0.2];
\draw[fill=white] (0,0) circle [radius=0.2];
\draw[fill=white] (2,0) circle [radius=0.2];
\end{scope}
\node at (10.5,-3) {$\stackrel{\up^{-1}}{\longmapsto}$};
\begin{scope}[shift={(14,-4)}]
\draw[thick] (-0.1, 1.9) -- (-0.9, 1.1);
\draw[thick] (0.1, 1.9) -- (0.9, 1.1);
\draw[thick] (-1.1, 0.9) -- (-1.9, 0.1);
\draw[thick] (-0.9, 0.9) -- (-0.1, 0.1);
\draw[thick] (0.9, 0.9) -- (0.1, 0.1);
\draw[thick] (1.1, 0.9) -- (1.9, 0.1);
\draw[fill=white] (0,2) circle [radius=0.2];
\draw[fill=white] (-1,1) circle [radius=0.2];
\draw[fill=black] (1,1) circle [radius=0.2];
\draw[fill=black] (-2,0) circle [radius=0.2];
\draw[fill=black] (0,0) circle [radius=0.2];
\draw[fill=black] (2,0) circle [radius=0.2];
\end{scope}
\node at (17.5,-3) {$\stackrel{\Theta}{\longmapsto}$};
\begin{scope}[shift={(21,-4)}]
\draw[thick] (-0.1, 1.9) -- (-0.9, 1.1);
\draw[thick] (0.1, 1.9) -- (0.9, 1.1);
\draw[thick] (-1.1, 0.9) -- (-1.9, 0.1);
\draw[thick] (-0.9, 0.9) -- (-0.1, 0.1);
\draw[thick] (0.9, 0.9) -- (0.1, 0.1);
\draw[thick] (1.1, 0.9) -- (1.9, 0.1);
\draw[fill=black] (0,2) circle [radius=0.2];
\draw[fill=black] (-1,1) circle [radius=0.2];
\draw[fill=white] (1,1) circle [radius=0.2];
\draw[fill=white] (-2,0) circle [radius=0.2];
\draw[fill=white] (0,0) circle [radius=0.2];
\draw[fill=white] (2,0) circle [radius=0.2];
\end{scope}
\end{tikzpicture}
\end{center}

\begin{center}
\begin{tikzpicture}[xscale=0.45,yscale=0.5]
\node at (-4.5,1) {$\rowA:$};
\begin{scope}
\draw[thick] (-0.1, 1.9) -- (-0.9, 1.1);
\draw[thick] (0.1, 1.9) -- (0.9, 1.1);
\draw[thick] (-1.1, 0.9) -- (-1.9, 0.1);
\draw[thick] (-0.9, 0.9) -- (-0.1, 0.1);
\draw[thick] (0.9, 0.9) -- (0.1, 0.1);
\draw[thick] (1.1, 0.9) -- (1.9, 0.1);
\draw[fill=white] (0,2) circle [radius=0.2];
\draw[fill=white] (-1,1) circle [radius=0.2];
\draw[fill=black] (1,1) circle [radius=0.2];
\draw[fill=black] (-2,0) circle [radius=0.2];
\draw[fill=white] (0,0) circle [radius=0.2];
\draw[fill=white] (2,0) circle [radius=0.2];
\end{scope}
\node at (3.5,1) {$\stackrel{\up^{-1}}{\longmapsto}$};
\begin{scope}[shift={(7,0)}]
\draw[thick] (-0.1, 1.9) -- (-0.9, 1.1);
\draw[thick] (0.1, 1.9) -- (0.9, 1.1);
\draw[thick] (-1.1, 0.9) -- (-1.9, 0.1);
\draw[thick] (-0.9, 0.9) -- (-0.1, 0.1);
\draw[thick] (0.9, 0.9) -- (0.1, 0.1);
\draw[thick] (1.1, 0.9) -- (1.9, 0.1);
\draw[fill=white] (0,2) circle [radius=0.2];
\draw[fill=white] (-1,1) circle [radius=0.2];
\draw[fill=black] (1,1) circle [radius=0.2];
\draw[fill=black] (-2,0) circle [radius=0.2];
\draw[fill=black] (0,0) circle [radius=0.2];
\draw[fill=black] (2,0) circle [radius=0.2];
\end{scope}
\node at (10.5,1) {$\stackrel{\Theta}{\longmapsto}$};
\begin{scope}[shift={(14,0)}]
\draw[thick] (-0.1, 1.9) -- (-0.9, 1.1);
\draw[thick] (0.1, 1.9) -- (0.9, 1.1);
\draw[thick] (-1.1, 0.9) -- (-1.9, 0.1);
\draw[thick] (-0.9, 0.9) -- (-0.1, 0.1);
\draw[thick] (0.9, 0.9) -- (0.1, 0.1);
\draw[thick] (1.1, 0.9) -- (1.9, 0.1);
\draw[fill=black] (0,2) circle [radius=0.2];
\draw[fill=black] (-1,1) circle [radius=0.2];
\draw[fill=white] (1,1) circle [radius=0.2];
\draw[fill=white] (-2,0) circle [radius=0.2];
\draw[fill=white] (0,0) circle [radius=0.2];
\draw[fill=white] (2,0) circle [radius=0.2];
\end{scope}
\node at (17.5,1) {$\stackrel{\down}{\longmapsto}$};
\begin{scope}[shift={(21,0)}]
\draw[thick] (-0.1, 1.9) -- (-0.9, 1.1);
\draw[thick] (0.1, 1.9) -- (0.9, 1.1);
\draw[thick] (-1.1, 0.9) -- (-1.9, 0.1);
\draw[thick] (-0.9, 0.9) -- (-0.1, 0.1);
\draw[thick] (0.9, 0.9) -- (0.1, 0.1);
\draw[thick] (1.1, 0.9) -- (1.9, 0.1);
\draw[fill=white] (0,2) circle [radius=0.2];
\draw[fill=black] (-1,1) circle [radius=0.2];
\draw[fill=white] (1,1) circle [radius=0.2];
\draw[fill=white] (-2,0) circle [radius=0.2];
\draw[fill=white] (0,0) circle [radius=0.2];
\draw[fill=white] (2,0) circle [radius=0.2];
\end{scope}
\end{tikzpicture}
\end{center}
\end{example}

\begin{remark}
It is more common to consider \emph{order ideal} rowmotion, defined as $\rowJ := \up^{-1} \circ \down \circ \Theta$, instead of order filter rowmotion $\rowF=\Theta \circ \up^{-1} \circ \down$. These two forms of rowmotion are of course conjugated to one another by $\Theta$. We find it more convenient to use the order filter perspective here to align with the conventions in the piecewise-linear and birational realms.
\end{remark}

\subsection{Order filter toggling}
As first discovered by Cameron and Fon-Der-Flaass~\cite{cameron1995orbits}, an equivalent way to describe order filter rowmotion $\rowF$ is in terms of simple involutions called \dfn{toggles}. (They actually defined toggles on order ideals not order filters, but again this is simply a choice of convention.)

\begin{definition} 
Let $p \in \sfp$.  Then the \dfn{order filter toggle} at $p$, $t_p: \calf(\sfp)\to \calf(\sfp)$, is defined by
\[t_p(F)\coloneqq \begin{cases} F\cup\{p\} &\text{if $p\not\in F$ and $F\cup\{p\}\in \calf(\sfp)$,}\\
F\setminus \{p\} &\text{if $p \in F$ and $F\setminus \{p \}\in \calf(\sfp)$,}\\
F &\text{otherwise.} \end{cases}\]
\end{definition}

The \dfn{order filter toggle group} of $\sfp$ is the group generated by $\{t_p : p \in \sfp\}$. Some basic properties of toggles are that each toggle $t_p$ is an involution, and for $p,q\in \sfp$, we have $t_p t_q = t_q t_p$ if and only if neither $p$ nor $q$ covers the other.

Recall that a \dfn{linear extension} of $\sfp$ is a listing $(x_1,x_2,\dots,x_n)$ containing every element of $\sfp$ exactly once, and for which $x_i<x_j$ implies that~$i<j$.

\begin{prop}[{\cite[Lem.~1]{cameron1995orbits}}] \label{prop:row-toggles}
Let $(x_1,x_2,\dots,x_n)$ be any linear extension of $\sfp$.  Then $\rowF=t_{x_1} t_{x_2} \cdots t_{x_n}$.
\end{prop}

\begin{example}
Let us demonstrate \cref{prop:row-toggles} on the poset $\sfa^3$. With the labels below, $(a,b,c,d,e,f)$ is a linear extension. We can see $\rowF = t_at_bt_ct_dt_et_f$ when applied to the same order filter considered in \cref{ex:a3-row}:
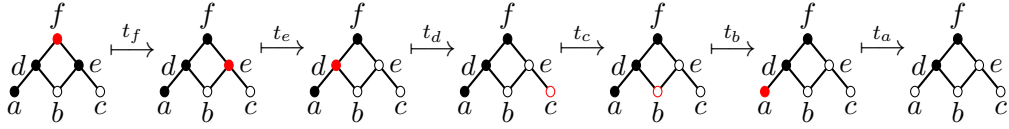
\begin{center}
\begin{tikzpicture}[yscale=0.35, xscale=0.285]
\begin{scope}[shift={(0,-4)}]
\draw[thick] (-0.1, 1.9) -- (-0.9, 1.1);
\draw[thick] (0.1, 1.9) -- (0.9, 1.1);
\draw[thick] (-1.1, 0.9) -- (-1.9, 0.1);
\draw[thick] (-0.9, 0.9) -- (-0.1, 0.1);
\draw[thick] (0.9, 0.9) -- (0.1, 0.1);
\draw[thick] (1.1, 0.9) -- (1.9, 0.1);
\draw[color=red,fill=red] (0,2) circle [radius=0.2];
\draw[fill=black] (-1,1) circle [radius=0.2];
\draw[fill=black] (1,1) circle [radius=0.2];
\draw[fill=black] (-2,0) circle [radius=0.2];
\draw[fill=white] (0,0) circle [radius=0.2];
\draw[fill=white] (2,0) circle [radius=0.2];
\node[above] at (0,2) {$f$};
\node[left] at (-1,1) {$d$};
\node[right] at (1,1) {$e$};
\node[below] at (-2,0) {$a$};
\node[below] at (0,0) {$b$};
\node[below] at (2,0) {$c$};
\end{scope}
\node at (3.5,-2) {$\stackrel{t_f}{\longmapsto}$};
\begin{scope}[shift={(7,-4)}]
\draw[thick] (-0.1, 1.9) -- (-0.9, 1.1);
\draw[thick] (0.1, 1.9) -- (0.9, 1.1);
\draw[thick] (-1.1, 0.9) -- (-1.9, 0.1);
\draw[thick] (-0.9, 0.9) -- (-0.1, 0.1);
\draw[thick] (0.9, 0.9) -- (0.1, 0.1);
\draw[thick] (1.1, 0.9) -- (1.9, 0.1);
\draw[fill=black] (0,2) circle [radius=0.2];
\draw[fill=black] (-1,1) circle [radius=0.2];
\draw[color=red,fill=red] (1,1) circle [radius=0.2];
\draw[fill=black] (-2,0) circle [radius=0.2];
\draw[fill=white] (0,0) circle [radius=0.2];
\draw[fill=white] (2,0) circle [radius=0.2];
\node[above] at (0,2) {$f$};
\node[left] at (-1,1) {$d$};
\node[right] at (1,1) {$e$};
\node[below] at (-2,0) {$a$};
\node[below] at (0,0) {$b$};
\node[below] at (2,0) {$c$};
\end{scope}
\node at (10.5,-2) {$\stackrel{t_e}{\longmapsto}$};
\begin{scope}[shift={(14,-4)}]
\draw[thick] (-0.1, 1.9) -- (-0.9, 1.1);
\draw[thick] (0.1, 1.9) -- (0.9, 1.1);
\draw[thick] (-1.1, 0.9) -- (-1.9, 0.1);
\draw[thick] (-0.9, 0.9) -- (-0.1, 0.1);
\draw[thick] (0.9, 0.9) -- (0.1, 0.1);
\draw[thick] (1.1, 0.9) -- (1.9, 0.1);
\draw[fill=black] (0,2) circle [radius=0.2];
\draw[color=red,fill=red] (-1,1) circle [radius=0.2];
\draw[fill=white] (1,1) circle [radius=0.2];
\draw[fill=black] (-2,0) circle [radius=0.2];
\draw[fill=white] (0,0) circle [radius=0.2];
\draw[fill=white] (2,0) circle [radius=0.2];
\node[above] at (0,2) {$f$};
\node[left] at (-1,1) {$d$};
\node[right] at (1,1) {$e$};
\node[below] at (-2,0) {$a$};
\node[below] at (0,0) {$b$};
\node[below] at (2,0) {$c$};
\end{scope}
\node at (17.5,-2) {$\stackrel{t_d}{\longmapsto}$};
\begin{scope}[shift={(21,-4)}]
\draw[thick] (-0.1, 1.9) -- (-0.9, 1.1);
\draw[thick] (0.1, 1.9) -- (0.9, 1.1);
\draw[thick] (-1.1, 0.9) -- (-1.9, 0.1);
\draw[thick] (-0.9, 0.9) -- (-0.1, 0.1);
\draw[thick] (0.9, 0.9) -- (0.1, 0.1);
\draw[thick] (1.1, 0.9) -- (1.9, 0.1);
\draw[fill=black] (0,2) circle [radius=0.2];
\draw[fill=black] (-1,1) circle [radius=0.2];
\draw[fill=white] (1,1) circle [radius=0.2];
\draw[fill=black] (-2,0) circle [radius=0.2];
\draw[fill=white] (0,0) circle [radius=0.2];
\draw[color=red,fill=white] (2,0) circle [radius=0.2];
\node[above] at (0,2) {$f$};
\node[left] at (-1,1) {$d$};
\node[right] at (1,1) {$e$};
\node[below] at (-2,0) {$a$};
\node[below] at (0,0) {$b$};
\node[below] at (2,0) {$c$};
\end{scope}
\node at (24.5,-2) {$\stackrel{t_c}{\longmapsto}$};
\begin{scope}[shift={(28,-4)}]
\draw[thick] (-0.1, 1.9) -- (-0.9, 1.1);
\draw[thick] (0.1, 1.9) -- (0.9, 1.1);
\draw[thick] (-1.1, 0.9) -- (-1.9, 0.1);
\draw[thick] (-0.9, 0.9) -- (-0.1, 0.1);
\draw[thick] (0.9, 0.9) -- (0.1, 0.1);
\draw[thick] (1.1, 0.9) -- (1.9, 0.1);
\draw[fill=black] (0,2) circle [radius=0.2];
\draw[fill=black] (-1,1) circle [radius=0.2];
\draw[fill=white] (1,1) circle [radius=0.2];
\draw[fill=black] (-2,0) circle [radius=0.2];
\draw[color=red,fill=white] (0,0) circle [radius=0.2];
\draw[fill=white] (2,0) circle [radius=0.2];
\node[above] at (0,2) {$f$};
\node[left] at (-1,1) {$d$};
\node[right] at (1,1) {$e$};
\node[below] at (-2,0) {$a$};
\node[below] at (0,0) {$b$};
\node[below] at (2,0) {$c$};
\end{scope}
\node at (31.5,-2) {$\stackrel{t_b}{\longmapsto}$};
\begin{scope}[shift={(35,-4)}]
\draw[thick] (-0.1, 1.9) -- (-0.9, 1.1);
\draw[thick] (0.1, 1.9) -- (0.9, 1.1);
\draw[thick] (-1.1, 0.9) -- (-1.9, 0.1);
\draw[thick] (-0.9, 0.9) -- (-0.1, 0.1);
\draw[thick] (0.9, 0.9) -- (0.1, 0.1);
\draw[thick] (1.1, 0.9) -- (1.9, 0.1);
\draw[fill=black] (0,2) circle [radius=0.2];
\draw[fill=black] (-1,1) circle [radius=0.2];
\draw[fill=white] (1,1) circle [radius=0.2];
\draw[color=red,fill=red] (-2,0) circle [radius=0.2];
\draw[fill=white] (0,0) circle [radius=0.2];
\draw[fill=white] (2,0) circle [radius=0.2];
\node[above] at (0,2) {$f$};
\node[left] at (-1,1) {$d$};
\node[right] at (1,1) {$e$};
\node[below] at (-2,0) {$a$};
\node[below] at (0,0) {$b$};
\node[below] at (2,0) {$c$};
\end{scope}
\node at (38.5,-2) {$\stackrel{t_a}{\longmapsto}$};
\begin{scope}[shift={(42,-4)}]
\draw[thick] (-0.1, 1.9) -- (-0.9, 1.1);
\draw[thick] (0.1, 1.9) -- (0.9, 1.1);
\draw[thick] (-1.1, 0.9) -- (-1.9, 0.1);
\draw[thick] (-0.9, 0.9) -- (-0.1, 0.1);
\draw[thick] (0.9, 0.9) -- (0.1, 0.1);
\draw[thick] (1.1, 0.9) -- (1.9, 0.1);
\draw[fill=black] (0,2) circle [radius=0.2];
\draw[fill=black] (-1,1) circle [radius=0.2];
\draw[fill=white] (1,1) circle [radius=0.2];
\draw[fill=white] (-2,0) circle [radius=0.2];
\draw[fill=white] (0,0) circle [radius=0.2];
\draw[fill=white] (2,0) circle [radius=0.2];
\node[above] at (0,2) {$f$};
\node[left] at (-1,1) {$d$};
\node[right] at (1,1) {$e$};
\node[below] at (-2,0) {$a$};
\node[below] at (0,0) {$b$};
\node[below] at (2,0) {$c$};
\end{scope}
\end{tikzpicture}
\end{center}
\end{example}

\subsection{Rowvacuation} 
While rowmotion and toggling can be defined on any poset, our next action, rowvacuation, is defined only on graded posets.

For $0\leq i \leq r$, set
\[\bft_i \coloneqq \prod\limits_{p \in \sfp_i} t_p.\]
The \dfn{order filter rank toggle} $\bft_i$ is well-defined because toggles of elements of the same rank commute. Some immediate properties of these rank toggles are recorded in the next proposition.

\begin{prop} \label{prop:basic-properties-OI-rank-toggles}
For $0\leq i,j\leq r$,
\begin{itemize}
\item $\bft_i^2 = 1$;
\item $\bft_i \bft_j = \bft_j \bft_i$ if $|i-j| > 1$.
\end{itemize}
\end{prop}

Clearly, $\rowF = \bft_0 \bft_1 \cdots \bft_{r-1} \bft_r$. This ``row-by-row'' (``rank-by-rank'') description of rowmotion is why it is called ``rowmotion.'' Rowvacuation is also built out of these rank toggles.

\begin{definition}
\dfn{Order filter rowvacuation} is the map $\rvacF\colon \calf(\sfp) \to \calf(\sfp)$ defined as the following composition of rank toggles
\[\rvacF \coloneqq (\bft_r) (\bft_{r-1} \bft_r) \cdots (\bft_1 \bft_2 \cdots \bft_{r-1} \bft_r) (\bft_0 \bft_1 \bft_2 \cdots \bft_{r-1} \bft_r).\]
\dfn{Order filter dual rowvacuation}, $\drvacF\colon \calf(\sfp) \to \calf(\sfp)$, is
\[\drvacF \coloneqq (\bft_0) (\bft_1 \bft_0) \cdots (\bft_{r-1} \cdots  \bft_2 \bft_1 \bft_0) (\bft_r \bft_{r-1} \cdots  \bft_2 \bft_1 \bft_0).\]
\end{definition}

We use $\sfp^*$ to denote the \dfn{dual poset} to a poset $\sfp$. There is an obvious duality $^* \colon \calf(\sfp)\to \calf(\sfp^*)$ between the order filters of~$\sfp$ and of~$\sfp^*$; namely, $F^* \coloneqq \Theta(F)$ for all $F \in \calf(\sfp)$. We could alternatively define dual rowvacuation by setting $\drvacF(F) \coloneqq \rvacF(F^*)^*$ for all $F \in \calf(\sfp)$. This explains the ``dual'' in the name ``dual rowvacuation.''

The following are the basic properties relating rowmotion and rowvacuation which hold for all graded posets:

\begin{prop}[{c.f.~\cite{hopkins2020order}}] \label{prop:row-rvac-dihedral}
For any graded poset $\sfp$ of rank $r$,
\begin{itemize}
    \item $\rvacF$ and $\drvacF$ are involutions;
    \item $\rvacF \circ \rowF = \rowF^{-1} \circ \rvacF$;
    \item $\drvacF \circ \rowF = \rowF^{-1} \circ \drvacF$;
    \item $\rowF^{r+2} = \drvacF \circ \rvacF$.
\end{itemize}
\end{prop}

So the cyclic group action of $\rowF$ extends to a dihedral group action generated by $\rowF$ and $\rvacF$. \Cref{prop:row-rvac-dihedral} says that rowmotion and rowvacuation together satisfy the same basic properties as Sch\"{u}tzenberger's promotion and evacuation operators acting on the linear extensions of a poset~\cite{schutzenberger1972promotion,stanley2009promotion} (hence the name ``rowvacuation''). Regarding the appearance of $\rowF^{r+2}$, note that there is always a $\rowF$ orbit of size $r+2$: $\{\{p\in\sfp\colon \rk(p)\leq i\}\colon i=-1,0,1,\ldots,\rk(P)\}$.

In the next proposition we show that knowledge of the whole rowmotion orbit of an order filter lets us read off its rowvacuation.

\begin{prop} \label{prop:rvac_sew_row}
Let $F \in \calf(\sfp)$ and $p\in \sfp_i$. Then $p\in\rvacF(F)$ if and only if $p\in\rowF^{i+1}(F)$.
\end{prop}

\begin{example}
Consider the following order filter $F$ in $\calf(\sfa^3)$ (the same one considered in \cref{ex:a3-row}):
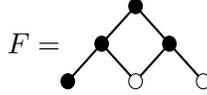
\begin{center}
\begin{tikzpicture}[xscale=0.45,yscale=0.5]
\node at (-3,1){$F=$};
\draw[thick] (-0.1, 1.9) -- (-0.9, 1.1);
\draw[thick] (0.1, 1.9) -- (0.9, 1.1);
\draw[thick] (-1.1, 0.9) -- (-1.9, 0.1);
\draw[thick] (-0.9, 0.9) -- (-0.1, 0.1);
\draw[thick] (0.9, 0.9) -- (0.1, 0.1);
\draw[thick] (1.1, 0.9) -- (1.9, 0.1);
\draw[fill=black] (0,2) circle [radius=0.2];
\draw[fill=black] (-1,1) circle [radius=0.2];
\draw[fill=black] (1,1) circle [radius=0.2];
\draw[fill=black] (-2,0) circle [radius=0.2];
\draw[fill=white] (0,0) circle [radius=0.2];
\draw[fill=white] (2,0) circle [radius=0.2];
\end{tikzpicture}
\end{center}
We compute its first three rowmotion iterates:
\begin{center}
\begin{tikzpicture}[xscale=0.45,yscale=0.5]
\begin{scope}[shift={(0,-1)}]
\draw[thick] (-0.1, 1.9) -- (-0.9, 1.1);
\draw[thick] (0.1, 1.9) -- (0.9, 1.1);
\draw[thick] (-1.1, 0.9) -- (-1.9, 0.1);
\draw[thick] (-0.9, 0.9) -- (-0.1, 0.1);
\draw[thick] (0.9, 0.9) -- (0.1, 0.1);
\draw[thick] (1.1, 0.9) -- (1.9, 0.1);
\draw[fill=black] (0,2) circle [radius=0.2];
\draw[fill=black] (-1,1) circle [radius=0.2];
\draw[fill=black] (1,1) circle [radius=0.2];
\draw[fill=black] (-2,0) circle [radius=0.2];
\draw[fill=white] (0,0) circle [radius=0.2];
\draw[fill=white] (2,0) circle [radius=0.2];
\end{scope}
\node at (3.5,0) {$\stackrel{\rowF}{\longmapsto}$};
\begin{scope}[shift={(7,-1)}]
\draw[thick] (-0.1, 1.9) -- (-0.9, 1.1);
\draw[thick] (0.1, 1.9) -- (0.9, 1.1);
\draw[thick] (-1.1, 0.9) -- (-1.9, 0.1);
\draw[thick] (-0.9, 0.9) -- (-0.1, 0.1);
\draw[thick] (0.9, 0.9) -- (0.1, 0.1);
\draw[thick] (1.1, 0.9) -- (1.9, 0.1);
\draw[fill=black] (0,2) circle [radius=0.2];
\draw[fill=black] (-1,1) circle [radius=0.2];
\draw[fill=white] (1,1) circle [radius=0.2];
\draw[thick, color=blue, fill=white] (-2,0) circle [radius=0.2];
\draw[thick, color=blue, fill=white] (0,0) circle [radius=0.2];
\draw[thick, color=blue, fill=white] (2,0) circle [radius=0.2];
\end{scope}
\node at (10.5,0) {$\stackrel{\rowF}{\longmapsto}$};
\begin{scope}[shift={(14,-1)}]
\draw[thick] (-0.1, 1.9) -- (-0.9, 1.1);
\draw[thick] (0.1, 1.9) -- (0.9, 1.1);
\draw[thick] (-1.1, 0.9) -- (-1.9, 0.1);
\draw[thick] (-0.9, 0.9) -- (-0.1, 0.1);
\draw[thick] (0.9, 0.9) -- (0.1, 0.1);
\draw[thick] (1.1, 0.9) -- (1.9, 0.1);
\draw[fill=black] (0,2) circle [radius=0.2];
\draw[thick, color=green, fill=white] (-1,1) circle [radius=0.2];
\draw[thick, color=green, fill=green] (1,1) circle [radius=0.2];
\draw[fill=white] (-2,0) circle [radius=0.2];
\draw[fill=white] (0,0) circle [radius=0.2];
\draw[fill=black] (2,0) circle [radius=0.2];
\end{scope}
\node at (17.5,0) {$\stackrel{\rowF}{\longmapsto}$};
\begin{scope}[shift={(21,-1)}]
\draw[thick] (-0.1, 1.9) -- (-0.9, 1.1);
\draw[thick] (0.1, 1.9) -- (0.9, 1.1);
\draw[thick] (-1.1, 0.9) -- (-1.9, 0.1);
\draw[thick] (-0.9, 0.9) -- (-0.1, 0.1);
\draw[thick] (0.9, 0.9) -- (0.1, 0.1);
\draw[thick] (1.1, 0.9) -- (1.9, 0.1);
\draw[thick, color=red, fill=red] (0,2) circle [radius=0.2];
\draw[fill=black] (-1,1) circle [radius=0.2];
\draw[fill=black] (1,1) circle [radius=0.2];
\draw[fill=black] (-2,0) circle [radius=0.2];
\draw[fill=black] (0,0) circle [radius=0.2];
\draw[fill=white] (2,0) circle [radius=0.2];
\end{scope}
\end{tikzpicture}
\end{center}
Then \cref{prop:rvac_sew_row} says that we can compute $\rvacF(F)$ by ``sewing together'' the ranks from these iterates:
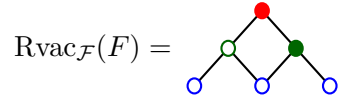
\begin{center}
\begin{tikzpicture}[xscale=0.45,yscale=0.5]
\node at (-5,1){$\rvacF(F)=$};
\draw[thick] (-0.1, 1.9) -- (-0.9, 1.1);
\draw[thick] (0.1, 1.9) -- (0.9, 1.1);
\draw[thick] (-1.1, 0.9) -- (-1.9, 0.1);
\draw[thick] (-0.9, 0.9) -- (-0.1, 0.1);
\draw[thick] (0.9, 0.9) -- (0.1, 0.1);
\draw[thick] (1.1, 0.9) -- (1.9, 0.1);
\draw[thick, color=red, fill=red] (0,2) circle [radius=0.2];
\draw[thick, color=green, fill=white] (-1,1) circle [radius=0.2];
\draw[thick, color=green, fill=green] (1,1) circle [radius=0.2];
\draw[thick, color=blue, fill=white] (-2,0) circle [radius=0.2];
\draw[thick, color=blue, fill=white] (0,0) circle [radius=0.2];
\draw[thick, color=blue, fill=white] (2,0) circle [radius=0.2];
\end{tikzpicture}
\end{center}
\end{example}

\Cref{prop:rvac_sew_row} is useful for translating information about rowmotion to rowvacuation, and vice versa (e.g., see \cref{subsec:homomesies} below).

\subsection{Antichain toggling}
There is nothing special about order filters in the definition of toggles.  Striker~\cite{striker2018generalized} suggested the study of toggles for other families of subsets, including antichains. Antichain toggling is examined in detail in~\cite{joseph2019antichain}. The definition of the antichain toggle is analogous to that of the order filter toggle; though note that removing an element from an antichain always yields an antichain.

\begin{definition} 
Let $p \in \sfp$.  Then the \dfn{antichain toggle} at $p$, $\tau_p: \cala(\sfp)\to \cala(\sfp)$, is defined by
\[\tau_p(A)\coloneqq \begin{cases} A\cup\{p\} &\text{if $p \not\in A$ and $A\cup\{p\}\in \cala(\sfp)$,}\\
A\setminus \{p\} &\text{if $p \in A$,}\\
A &\text{otherwise.} \end{cases}\]
\end{definition}

It is straightforward to see that each antichain toggle $\tau_p$ is an involution, as with the order filter toggles. However, $\tau_p \tau_q = \tau_p\tau_q$ if and only if $p$ and $q$ are incomparable or equal, which is different from the commutativity conditions for the order filter toggles. The \dfn{antichain toggle group} of $\sfp$ is the group generated by $\{\tau_p : p \in \sfp\}$.
 
Antichain rowmotion can also be expressed as a product of toggles, according to a linear extension, but in the \emph{opposite} order as order filter rowmotion.

\begin{prop}[{\cite[Prop.~2.24]{joseph2019antichain}}] \label{joseph2019antichain}
Let $(x_1,x_2,\dots,x_n)$ be any linear extension of~$\sfp$.  Then $\rowA=\tau_{x_n}\cdots \tau_{x_2} \tau_{x_1}$.
\end{prop}

\begin{example}
Let us demonstrate \cref{joseph2019antichain} on the poset $\sfa^3$.  With the labels below, $(a,b,c,d,e,f)$ is a linear extension. We can see that $\rowA = \tau_a \tau_b \tau_c \tau_d \tau_e \tau_f$ when applied to the same antichain considered in \cref{ex:a3-row}:
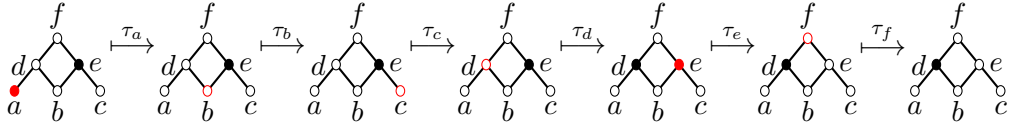
\begin{center}
\begin{tikzpicture}[yscale=0.35, xscale=0.285]
\begin{scope}[shift={(0,-4)}]
\draw[thick] (-0.1, 1.9) -- (-0.9, 1.1);
\draw[thick] (0.1, 1.9) -- (0.9, 1.1);
\draw[thick] (-1.1, 0.9) -- (-1.9, 0.1);
\draw[thick] (-0.9, 0.9) -- (-0.1, 0.1);
\draw[thick] (0.9, 0.9) -- (0.1, 0.1);
\draw[thick] (1.1, 0.9) -- (1.9, 0.1);
\draw[fill=white] (0,2) circle [radius=0.2];
\draw[fill=white] (-1,1) circle [radius=0.2];
\draw[fill=black] (1,1) circle [radius=0.2];
\draw[color=red,fill=red] (-2,0) circle [radius=0.2];
\draw[fill=white] (0,0) circle [radius=0.2];
\draw[fill=white] (2,0) circle [radius=0.2];
\node[above] at (0,2) {$f$};
\node[left] at (-1,1) {$d$};
\node[right] at (1,1) {$e$};
\node[below] at (-2,0) {$a$};
\node[below] at (0,0) {$b$};
\node[below] at (2,0) {$c$};
\end{scope}
\node at (3.5,-2) {$\stackrel{\tau_a}{\longmapsto}$};
\begin{scope}[shift={(7,-4)}]
\draw[thick] (-0.1, 1.9) -- (-0.9, 1.1);
\draw[thick] (0.1, 1.9) -- (0.9, 1.1);
\draw[thick] (-1.1, 0.9) -- (-1.9, 0.1);
\draw[thick] (-0.9, 0.9) -- (-0.1, 0.1);
\draw[thick] (0.9, 0.9) -- (0.1, 0.1);
\draw[thick] (1.1, 0.9) -- (1.9, 0.1);
\draw[fill=white] (0,2) circle [radius=0.2];
\draw[fill=white] (-1,1) circle [radius=0.2];
\draw[fill=black] (1,1) circle [radius=0.2];
\draw[fill=white] (-2,0) circle [radius=0.2];
\draw[color=red,fill=white] (0,0) circle [radius=0.2];
\draw[fill=white] (2,0) circle [radius=0.2];
\node[above] at (0,2) {$f$};
\node[left] at (-1,1) {$d$};
\node[right] at (1,1) {$e$};
\node[below] at (-2,0) {$a$};
\node[below] at (0,0) {$b$};
\node[below] at (2,0) {$c$};
\end{scope}
\node at (10.5,-2) {$\stackrel{\tau_b}{\longmapsto}$};
\begin{scope}[shift={(14,-4)}]
\draw[thick] (-0.1, 1.9) -- (-0.9, 1.1);
\draw[thick] (0.1, 1.9) -- (0.9, 1.1);
\draw[thick] (-1.1, 0.9) -- (-1.9, 0.1);
\draw[thick] (-0.9, 0.9) -- (-0.1, 0.1);
\draw[thick] (0.9, 0.9) -- (0.1, 0.1);
\draw[thick] (1.1, 0.9) -- (1.9, 0.1);
\draw[fill=white] (0,2) circle [radius=0.2];
\draw[fill=white] (-1,1) circle [radius=0.2];
\draw[fill=black] (1,1) circle [radius=0.2];
\draw[fill=white] (-2,0) circle [radius=0.2];
\draw[fill=white] (0,0) circle [radius=0.2];
\draw[color=red,fill=white] (2,0) circle [radius=0.2];
\node[above] at (0,2) {$f$};
\node[left] at (-1,1) {$d$};
\node[right] at (1,1) {$e$};
\node[below] at (-2,0) {$a$};
\node[below] at (0,0) {$b$};
\node[below] at (2,0) {$c$};
\end{scope}
\node at (17.5,-2) {$\stackrel{\tau_c}{\longmapsto}$};
\begin{scope}[shift={(21,-4)}]
\draw[thick] (-0.1, 1.9) -- (-0.9, 1.1);
\draw[thick] (0.1, 1.9) -- (0.9, 1.1);
\draw[thick] (-1.1, 0.9) -- (-1.9, 0.1);
\draw[thick] (-0.9, 0.9) -- (-0.1, 0.1);
\draw[thick] (0.9, 0.9) -- (0.1, 0.1);
\draw[thick] (1.1, 0.9) -- (1.9, 0.1);
\draw[fill=white] (0,2) circle [radius=0.2];
\draw[color=red,fill=white] (-1,1) circle [radius=0.2];
\draw[fill=black] (1,1) circle [radius=0.2];
\draw[fill=white] (-2,0) circle [radius=0.2];
\draw[fill=white] (0,0) circle [radius=0.2];
\draw[fill=white] (2,0) circle [radius=0.2];
\node[above] at (0,2) {$f$};
\node[left] at (-1,1) {$d$};
\node[right] at (1,1) {$e$};
\node[below] at (-2,0) {$a$};
\node[below] at (0,0) {$b$};
\node[below] at (2,0) {$c$};
\end{scope}
\node at (24.5,-2) {$\stackrel{\tau_d}{\longmapsto}$};
\begin{scope}[shift={(28,-4)}]
\draw[thick] (-0.1, 1.9) -- (-0.9, 1.1);
\draw[thick] (0.1, 1.9) -- (0.9, 1.1);
\draw[thick] (-1.1, 0.9) -- (-1.9, 0.1);
\draw[thick] (-0.9, 0.9) -- (-0.1, 0.1);
\draw[thick] (0.9, 0.9) -- (0.1, 0.1);
\draw[thick] (1.1, 0.9) -- (1.9, 0.1);
\draw[fill=white] (0,2) circle [radius=0.2];
\draw[fill=black] (-1,1) circle [radius=0.2];
\draw[color=red,fill=red] (1,1) circle [radius=0.2];
\draw[fill=white] (-2,0) circle [radius=0.2];
\draw[fill=white] (0,0) circle [radius=0.2];
\draw[fill=white] (2,0) circle [radius=0.2];
\node[above] at (0,2) {$f$};
\node[left] at (-1,1) {$d$};
\node[right] at (1,1) {$e$};
\node[below] at (-2,0) {$a$};
\node[below] at (0,0) {$b$};
\node[below] at (2,0) {$c$};
\end{scope}
\node at (31.5,-2) {$\stackrel{\tau_e}{\longmapsto}$};
\begin{scope}[shift={(35,-4)}]
\draw[thick] (-0.1, 1.9) -- (-0.9, 1.1);
\draw[thick] (0.1, 1.9) -- (0.9, 1.1);
\draw[thick] (-1.1, 0.9) -- (-1.9, 0.1);
\draw[thick] (-0.9, 0.9) -- (-0.1, 0.1);
\draw[thick] (0.9, 0.9) -- (0.1, 0.1);
\draw[thick] (1.1, 0.9) -- (1.9, 0.1);
\draw[color=red,fill=white] (0,2) circle [radius=0.2];
\draw[fill=black] (-1,1) circle [radius=0.2];
\draw[fill=white] (1,1) circle [radius=0.2];
\draw[fill=white] (-2,0) circle [radius=0.2];
\draw[fill=white] (0,0) circle [radius=0.2];
\draw[fill=white] (2,0) circle [radius=0.2];
\node[above] at (0,2) {$f$};
\node[left] at (-1,1) {$d$};
\node[right] at (1,1) {$e$};
\node[below] at (-2,0) {$a$};
\node[below] at (0,0) {$b$};
\node[below] at (2,0) {$c$};
\end{scope}
\node at (38.5,-2) {$\stackrel{\tau_f}{\longmapsto}$};
\begin{scope}[shift={(42,-4)}]
\draw[thick] (-0.1, 1.9) -- (-0.9, 1.1);
\draw[thick] (0.1, 1.9) -- (0.9, 1.1);
\draw[thick] (-1.1, 0.9) -- (-1.9, 0.1);
\draw[thick] (-0.9, 0.9) -- (-0.1, 0.1);
\draw[thick] (0.9, 0.9) -- (0.1, 0.1);
\draw[thick] (1.1, 0.9) -- (1.9, 0.1);
\draw[fill=white] (0,2) circle [radius=0.2];
\draw[fill=black] (-1,1) circle [radius=0.2];
\draw[fill=white] (1,1) circle [radius=0.2];
\draw[fill=white] (-2,0) circle [radius=0.2];
\draw[fill=white] (0,0) circle [radius=0.2];
\draw[fill=white] (2,0) circle [radius=0.2];
\node[above] at (0,2) {$f$};
\node[left] at (-1,1) {$d$};
\node[right] at (1,1) {$e$};
\node[below] at (-2,0) {$a$};
\node[below] at (0,0) {$b$};
\node[below] at (2,0) {$c$};
\end{scope}
\end{tikzpicture}
\end{center}
\end{example}
 
Since elements of the same rank are incomparable, the \dfn{antichain rank toggle}
\[\bftau_i \coloneqq \prod\limits_{p \in \sfp_i} \tau_p,\]
for $0\leq i \leq r$, is well-defined. Clearly, $\rowA = \bftau_r \bftau_{r-1} \cdots \bftau_1 \bftau_0$.

\subsection{Antichain rowvacuation}
Instead of considering rowvacuation as an action on an order filter $F\in\calf(P)$, we can consider it to be an action on the antichain $\down(F) \in \cala(\sfp)$ associated to $F$.

\begin{definition}
\dfn{Antichain rowvacuation} is the map $\rvacA\colon \cala(\sfp) \to \cala(\sfp)$ defined as the following composition of antichain rank toggles
\[\rvacA \coloneqq (\bftau_r)(\bftau_r \bftau_{r-1})\cdots (\bftau_r \bftau_{r-1} \cdots \bftau_2 \bftau_1) (\bftau_r \bftau_{r-1} \cdots \bftau_2 \bftau_1 \bftau_0).\]
\dfn{Antichain dual rowvacuation}, $\drvacA \colon \cala(\sfp) \to \cala(\sfp)$, is
\[\drvacA \coloneqq (\bftau_0) (\bftau_0 \bftau_1) \cdots (\bftau_0 \bftau_1 \bftau_2 \cdots \bftau_{r-1}) (\bftau_0 \bftau_1 \bftau_2 \cdots \bftau_{r-1} \bftau_r).\]
\end{definition}

Again, there is an obvious duality $^*\colon \cala(\sfp)\to \cala(\sfp^*)$ given by $A^* \coloneqq A$ for all $A \in \cala(\sfp)$, and again we have $\drvacA(A) = \rvacA(A^*)^*$ for all $A\in \cala(A)$.

Of course, we need to show that the antichain version of rowvacuation is equivalent to its order filter version, which we do in the next proposition. In fact, this proposition explains the conjugacy between all of the order filter and antichain operators. It also asserts that $\rowA$ and $\rvacA$ generate a dihedral action as well.

\begin{prop} \label{prop:rvac-OI-ant}
The following diagrams commute:
\begin{center}
\phantom{1}\hfill
\begin{tikzpicture}[scale=2/3]
\node at (0,1.8) {$\calf(\sfp)$};
\node at (0,0) {$\cala(\sfp)$};
\node at (3.25,1.8) {$\calf(\sfp)$};
\node at (3.25,0) {$\cala(\sfp)$};
\draw[semithick, ->] (0,1.3) -- (0,0.5);
\node[left] at (0,0.9) {$\down$};
\draw[semithick, ->] (0.8,0) -- (2.4,0);
\node[below] at (1.5,0) {$\rowA$};
\draw[semithick, ->] (0.8,1.8) -- (2.4,1.8);
\node[above] at (1.5,1.8) {$\rowF$};
\draw[semithick, ->] (3.25,1.3) -- (3.25,0.5);
\node[right] at (3.25,0.9) {$\down$};
\end{tikzpicture}
\hfill
\begin{tikzpicture}[yscale=2/3]
\node at (0.25,1.8) {$\calf(\sfp)$};
\node at (0.25,0) {$\cala(\sfp)$};
\node at (3,1.8) {$\calf(\sfp)$};
\node at (3,0) {$\cala(\sfp)$};
\draw[semithick, ->] (0.25,1.3) -- (0.25,0.5);
\node[left] at (0.25,0.9) {$\down$};
\draw[semithick, ->] (0.8,0) -- (2.4,0);
\node[below] at (1.5,0) {$\rvacA$};
\draw[semithick, ->] (0.8,1.8) -- (2.4,1.8);
\node[above] at (1.5,1.8) {$\rvacF$};
\draw[semithick, ->] (3,1.3) -- (3,0.5);
\node[right] at (3,0.9) {$\down$};
\end{tikzpicture}\hfill
\begin{tikzpicture}[yscale=2/3]
\node at (0.25,1.8) {$\calf(\sfp)$};
\node at (0.25,0) {$\cala(\sfp)$};
\node at (3,1.8) {$\calf(\sfp)$};
\node at (3,0) {$\cala(\sfp)$};
\draw[semithick, ->] (0.25,1.3) -- (0.25,0.5);
\node[left] at (0.25,0.9) {$\up\circ\Theta$};
\draw[semithick, ->] (0.8,0) -- (2.4,0);
\node[below] at (1.5,0) {$\drvacA$};
\draw[semithick, ->] (0.8,1.8) -- (2.4,1.8);
\node[above] at (1.5,1.8) {$\drvacF$};
\draw[semithick, ->] (3,1.3) -- (3,0.5);
\node[right] at (3,0.9) {$\up\circ\Theta$};
\end{tikzpicture}\hfill\phantom{1}
\end{center}
(Note that $\up\circ\Theta = \rowA^{-1} \circ \down$.) Hence, the first three bulleted items of \cref{prop:row-rvac-dihedral} hold with $\calf$ replaced by $\cala$.
\end{prop}

Finally, we conclude our discussion of rowvacuation at the combinatorial level with another way to compute rowvacuation. For $0\leq i \leq r$, set
\[\sfp_{\geq i} \coloneqq \bigcup_{j=i}^{r} \sfp_j.\]
We now give an inductive description of antichain rowvacuation, where, roughly speaking, we can compute $\rvacA(A)$ by restricting $A$ to $\sfp_{\geq 1}$ and computing rowvacuations there. More precisely, we have the following:

\begin{lemma} \label{lem:rvac_ind}
Let $A \in \cala(\sfp)$ and $p\in \sfp$.
\begin{itemize}
\item If $p \in \sfp_0$, then $p\in \rvacA(A)$ if and only if $p \in \bftau_0(A)$ (i.e., if and only if $A$ does not contain any element $q \geq p$).
\item If $p \in \sfp_{\geq 1}$, then $p\in \rvacA(A)$ if and only if $p \in \rowA^{-1}\circ \rvacA(\overline{A})$, where $\overline{A} \coloneqq A\cap \sfp_{\geq 1} \in \cala(\sfp_{\geq 1})$.
\end{itemize}
\end{lemma}

\begin{example}
Consider the following antichain $A \in \cala(\sfa^3)$ (the same one considered in \cref{ex:a3-row}):
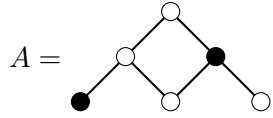
\begin{center}
\begin{tikzpicture}[xscale=0.6,yscale=0.6]
\node at (-3,1){$A=$};
\draw[thick] (-0.1, 1.9) -- (-0.9, 1.1);
\draw[thick] (0.1, 1.9) -- (0.9, 1.1);
\draw[thick] (-1.1, 0.9) -- (-1.9, 0.1);
\draw[thick] (-0.9, 0.9) -- (-0.1, 0.1);
\draw[thick] (0.9, 0.9) -- (0.1, 0.1);
\draw[thick] (1.1, 0.9) -- (1.9, 0.1);
\draw[fill=white] (0,2) circle [radius=0.2];
\draw[fill=white] (-1,1) circle [radius=0.2];
\draw[fill=black] (1,1) circle [radius=0.2];
\draw[fill=black] (-2,0) circle [radius=0.2];
\draw[fill=white] (0,0) circle [radius=0.2];
\draw[fill=white] (2,0) circle [radius=0.2];
\end{tikzpicture}
\end{center}
Observe that $\sfa^{3}_{\geq 1}\simeq\sfa^2$ and $\sfa^{3}_{\geq 2}\simeq\sfa^1$. We show how to use \cref{lem:rvac_ind} to compute $\rvacA(A)$ below:
\begin{center}
\begin{tikzpicture}[scale=0.6]
\begin{scope}[shift={(0,2)}]
\node at (-3,1.5) {$\sfa^1$:};
\draw[red,thick,dashed] (-0.3,1.2) -- (-0.3,1.8) -- (0.3,1.8) -- (0.3,1.2) -- (-0.3,1.2);
\draw [fill=white] (0,1.5) circle [radius=0.2];
\node at (2.5,1.6) {\Large $\stackrel{\rvacA}{\longmapsto}$};
\end{scope}
\begin{scope}[shift={(5,2)}]
\draw [fill=black] (0,1.5) circle [radius=0.2];
\node at (2.5,1.6) {\Large $\stackrel{\rowA^{-1}}{\longmapsto}$};
\end{scope}
\begin{scope}[shift={(10,2)}]
\draw[red,thick] (-0.3,1.2) -- (-0.3,1.8) -- (0.3,1.8) -- (0.3,1.2) -- (-0.3,1.2);
\draw [fill=white] (0,1.5) circle [radius=0.2];
\end{scope}
\begin{scope}
\node at (-3,1.5) {$\sfa^2$:};
\draw[thick] (-0.15, 1.85) -- (-0.85, 1.15);
\draw[thick] (0.15, 1.85) -- (0.85, 1.15);
\draw [fill=white] (0,2) circle [radius=0.2];
\draw [fill=white] (-1,1) circle [radius=0.2];
\draw [fill=black] (1,1) circle [radius=0.2];
\draw[green,thick,dashed] (-1.6,0.7) -- (0,2.4) -- (1.6, 0.7) -- (-1.6,0.7);
\draw[red,thick,dashed] (-0.3,1.7) -- (-0.3,2.3) -- (0.3,2.3) -- (0.3,1.7) -- (-0.3,1.7);
\node at (2.5,1.6) {\Large $\stackrel{\rvacA}{\longmapsto}$};
\end{scope}
\begin{scope}[shift={(5,0)}]
\draw[thick] (-0.15, 1.85) -- (-0.85, 1.15);
\draw[thick] (0.15, 1.85) -- (0.85, 1.15);
\draw [fill=white] (0,2) circle [radius=0.2];
\draw [fill=black] (-1,1) circle [radius=0.2];
\draw [fill=white] (1,1) circle [radius=0.2];
\draw[red,thick] (-0.3,1.7) -- (-0.3,2.3) -- (0.3,2.3) -- (0.3,1.7) -- (-0.3,1.7);
\node at (2.5,1.6) {\Large $\stackrel{\rowA^{-1}}{\longmapsto}$};
\end{scope}
\begin{scope}[shift={(10,0)}]
\draw[thick] (-0.15, 1.85) -- (-0.85, 1.15);
\draw[thick] (0.15, 1.85) -- (0.85, 1.15);
\draw [fill=white] (0,2) circle [radius=0.2];
\draw [fill=white] (-1,1) circle [radius=0.2];
\draw [fill=black] (1,1) circle [radius=0.2];
\draw[green,thick] (-1.6,0.7) -- (0,2.4) -- (1.6, 0.7) -- (-1.6,0.7);
\end{scope}
\node at (-3,-2) {$\sfa^3$:};
\begin{scope}[shift={(0,-3)}]
\draw[thick] (-0.15, 1.85) -- (-0.85, 1.15);
\draw[thick] (0.15, 1.85) -- (0.85, 1.15);
\draw[thick] (-1.15, 0.85) -- (-1.85, 0.15);
\draw[thick] (-0.85, 0.85) -- (-0.15, 0.15);
\draw[thick] (0.85, 0.85) -- (0.15, 0.15);
\draw[thick] (1.15, 0.85) -- (1.85, 0.15);
\draw [fill=white] (0,2) circle [radius=0.2];
\draw [fill=white] (-1,1) circle [radius=0.2];
\draw [fill=black] (1,1) circle [radius=0.2];
\draw [fill=black] (-2,0) circle [radius=0.2];
\draw [fill=white] (0,0) circle [radius=0.2];
\draw [fill=white] (2,0) circle [radius=0.2];
\draw[green,thick,dashed] (-1.6,0.7) -- (0,2.4) -- (1.6, 0.7) -- (-1.6,0.7);
\node at (2.5,1.6) {\Large $\stackrel{\rvacA}{\longmapsto}$};
\end{scope}
\begin{scope}[shift={(5,-3)}]
\draw[thick] (-0.15, 1.85) -- (-0.85, 1.15);
\draw[thick] (0.15, 1.85) -- (0.85, 1.15);
\draw[thick] (-1.15, 0.85) -- (-1.85, 0.15);
\draw[thick] (-0.85, 0.85) -- (-0.15, 0.15);
\draw[thick] (0.85, 0.85) -- (0.15, 0.15);
\draw[thick] (1.15, 0.85) -- (1.85, 0.15);
\draw [fill=white] (0,2) circle [radius=0.2];
\draw [fill=white] (-1,1) circle [radius=0.2];
\draw [fill=black] (1,1) circle [radius=0.2];
\draw [fill=white] (-2,0) circle [radius=0.2];
\draw [fill=white] (0,0) circle [radius=0.2];
\draw [fill=white] (2,0) circle [radius=0.2];
\draw[green,thick] (-1.6,0.7) -- (0,2.4) -- (1.6, 0.7) -- (-1.6,0.7);
\end{scope}
\end{tikzpicture}
\end{center}
We will go over a similar example, with more explanation, again in \cref{ex:no-draw-Dyck}.
\end{example}

\subsection{Piecewise-linear lifts}

Now we explain the extensions of our actions on $\calf(\sfp)$ and $\cala(\sfp)$ from the combinatorial to the piecewise-linear (PL) realm.

Let $\sfphat$ denote the poset obtained from $\sfp$ by adjoining a new minimal element~$\hatz$ and a new maximal element~$\hato$. Then define the following affine spaces of real-valued functions on $\sfphat$ or $\sfp$:
\begin{align*}
\calfpl_{\kappa}(\sfp) &\coloneqq \{ \pi \in \rr^{\sfphat}\colon \pi(\hatz) = 0, \pi(\hato)=\kappa\}, \\
\caljpl_{\kappa}(\sfp) &\coloneqq \{ \pi \in \rr^{\sfphat}\colon \pi(\hatz) = \kappa, \pi(\hato)=0\},\\
\calapl_{\kappa}(\sfp) &\coloneqq \{ \pi \in \rr^{\sfp}\}.
\end{align*}
Here $\kappa \in \rr$ is a parameter. When its values is clear from context, we will use the shorthands $\calfpl(\sfp)$, $\caljpl(\sfp)$, and $\calapl(\sfp)$. Observe that these three spaces are all basically the same (and we often implicitly identify them all with $\rr^{\sfp}$ by forgetting the values at $\hatz$ and $\hato$); but we think of them separately as the piecewise-linear analogs of $\calf(\sfp)$, $\calj(\sfp)$, and~$\cala(\sfp)$.

There are some important polytopes which live in $\calfpl_1(\sfp)$, $\caljpl_1(\sfp)$, and $\calapl_1(\sfp)$. Namely:
\begin{itemize}
\item the \dfn{order polytope} $\calo\calp(\sfp)\subseteq \calfpl_{1}(\sfp)$, where
\[\calo\calp(\sfp) \coloneqq \left\{\pi \in \calfpl_{1}(\sfp)\colon \pi(x)\leq \pi(y) \textrm{ whenever $x\leq y \in \sfphat$}\right\}; \]
\item the \dfn{order-reversing polytope} $\calo\calr(\sfp)\subseteq \caljpl_{1}(\sfp)$, where
\[\calo\calr(\sfp) \coloneqq \left\{\pi \in \caljpl_{1}(\sfp)\colon \pi(x)\geq \pi(y) \textrm{ whenever $x\leq y \in \sfphat$}\right\}; \]
\item the \dfn{chain polytope} $\calc(\sfp) \subseteq \calapl_{1}(\sfp)$, where
\[\calc(\sfp) \coloneqq \left\{\pi \in \calapl_{1}(\sfp)\colon 0 \leq \sum_{x \in C} \pi(x) \leq 1 \textrm{ for any chain $C\subseteq \sfp$}\right\}.\]
\end{itemize}
Recall that a \dfn{chain} of $\sfp$ is a subset $C\subseteq \sfp$ in which any two elements are comparable. Some of the inequalities in the above descriptions of these polytopes are redundant; for example, the facets of the order polytope $\calo\calp(\sfp)$ are given by
\[ \pi(x) \leq \pi(y) \textrm{ whenever $x \lessdot y \in \sfphat$},\]
and the facets of the chain polytope $\calc(\sfp)$ are given by
\begin{align*}
0\leq \pi(x) &\textrm{ for all $x \in \sfp$}, \\
 \sum_{x \in C} \pi(x) \leq 1 &\textrm{ for any \emph{maximal} chain $C\subseteq \sfp$}.
 \end{align*}
Here by \dfn{maximal chain} we mean a maximal by inclusion chain.

We identify each subset $S\subseteq \sfp$ with its indicator function. With this identification in mind, Stanley~\cite{stanley1986twoposet} showed that:
\begin{itemize}
\item the vertices of $\calo\calp(\sfp)$ (resp.~of $\calo\calr(\sfp)$) are the $F \in \calf(\sfp)$ (resp.~$J \in \calj(\sfp)$), 
\item the vertices of $\calc(\sfp)$ are the $A\in \cala(\sfp)$.
\end{itemize}
This explains how we transfer information from the piecewise-linear realm to the combinatorial realm: we \dfn{specialize} (i.e., restrict) to the vertices of these polyotpes.

We proceed to explain the piecewise-linear lifts of rowmotion and toggling introduced by Einstein and Propp~\cite{einstein2018combinatorial}. In many cases we use exactly the same notation as in the combinatorial realms for these piecewise-linear maps, and let context distinguish them. Of course, the PL extensions specialize to their combinatorial analogs.

We first explain the PL analog of the definition of rowmotion as the composition of three bijections. These bijections are:
\begin{itemize}
\item \dfn{complementation} $\Theta\colon \rr^{\sfphat} \to \rr^{\sfphat}$, with $(\Theta\pi)(x) \coloneqq \kappa-\pi(x)$ for all~$x \in \sfphat$;
\item \dfn{up-transfer} $\up \colon\caljpl_{\kappa}(\sfp) \to \calapl_{\kappa}(\sfp)$, with
\[(\up \pi)(x) \coloneqq \pi(x) - \max\left\{\pi(y)\colon x\lessdot y \in \sfphat\right\}\]
for all $x\in \sfp$;
\item \dfn{down-transfer} $\down \colon\calfpl_{\kappa}(\sfp) \to \calapl_{\kappa}(\sfp)$, with 
\[(\down\pi) (x) \coloneqq \pi(x)- \max\left\{\pi(y)\colon y\lessdot x \in \sfphat\right\}\]
for all $x\in \sfp$.
\end{itemize}
Evidently, $\Theta^{-1}=\Theta$. Also, note that
\begin{align*}
(\up^{-1}\pi)(x) &= \max\left\{ \pi(y_1)+\pi(y_2)+\cdots +\pi(y_k)\colon x=y_1\lessdot y_2 \lessdot \cdots \lessdot y_k\lessdot\hato \in \sfphat \right\} \\
&=\pi(x)+ \max\left\{(\up^{-1}\pi)(y)\colon x\lessdot y\in \sfphat\right\}
\end{align*}
for all $x\in\sfp$, and similarly,
\begin{align*}
(\down^{-1}\pi)(x) &= \max\left\{  \pi(y_1)+\pi(y_2)+\cdots +\pi(y_k)\colon \hatz\lessdot y_1\lessdot y_2 \lessdot \cdots \lessdot y_k=x \in \sfphat \right\} \\
&=\pi(x)+ \max\left\{(\down^{-1}\pi)(y)\colon y\lessdot x \in \sfphat\right\}
\end{align*}
for all $x\in \sfp$.

Complementation is an involution that maps $\calo\calp(\sfp)$ to $\calo\calr(\sfp)$ and vice versa. Down-transfer (which is equivalent to Stanley's ``transfer map''~\cite{stanley1986twoposet}) maps $\calo\calp(\sfp)$ to~$\calc(\sfp)$, while up-transfer maps $\calo\calr(\sfp)$ to $\calc(\sfp)$.

\begin{definition}
\dfn{PL order filter rowmotion}, denoted $\rowFpl: \calfpl(\sfp) \to \calfpl(\sfp)$, is given by $\rowFpl \coloneqq \Theta \circ \up^{-1} \circ \down$. 
\end{definition}

\begin{definition}
\dfn{PL antichain rowmotion}, denoted $\rowApl: \calapl(\sfp) \to \calapl(\sfp)$, is given by $\rowApl \coloneqq \down \circ \Theta \circ \up^{-1}$.
\end{definition}

Next, we go over the toggling description of these maps. 

\begin{definition}
Let $p \in \sfp$.  Then the \dfn{PL order filter toggle} at $p$ is the map $t_p \colon \calfpl_{\kappa}(\sfp)\to \calfpl_{\kappa}(\sfp)$ defined by
\[(t_p\pi)(x) \coloneqq \begin{cases} \pi(x) &\textrm{if $x\neq p$};\\
\max\left\{\pi(y)\colon y\lessdot x \in \sfphat\right\} + \min\left\{\pi(y)\colon x\lessdot y \in \sfphat\right\} - \pi(x) &\textrm{if $x=p$},\end{cases}\]
for all $x\in \sfp$.
\end{definition}

\begin{definition}
Let $p \in \sfp$. Let $\mc_p(\sfp)$ denote the set of all maximal chains $C\subseteq \sfp$ with $p\in C$. Then the \dfn{PL antichain toggle} at $p$ is the map $\tau_p \colon \calapl_{\kappa}(\sfp)\to \calapl_{\kappa}(\sfp)$ defined by
\[(\tau_p\pi)(x) \coloneqq \begin{cases} \pi(x) &\textrm{if $x\neq p$};\\
\kappa-\max\left \{\displaystyle\sum_{y \in C} \pi(y)\colon C \in \mc_p(\sfp)\right\} &\textrm{if $x=p$},\end{cases}\]
for all $x\in \sfp$.
\end{definition}

These PL toggles are again involutions and have the same commutativity properties as their combinatorial analogs. An important observation is that (when $\kappa=1$) the order toggles $t_p$ preserve the order polytope $\calo\calp(\sfp)$; and similarly the antichain toggles $\tau_p$ preserve the chain polytope $\calc(\sfp)$. In this case both of these kinds of toggles also preserve the lattice $\frac{1}{m}\zz^{\sfp}$ for any $m \in \zz_{> 0}$.

The PL versions of rowmotion are built out of these toggles in exactly the same way as in the combinatorial realm, as shown by Einstein--Propp~\cite{einstein2018combinatorial} and the second author~\cite{joseph2019antichain}:

\begin{prop}[{Einstein--Propp~\cite{einstein2018combinatorial}; c.f.~\cite[Thm.~5.12]{joseph2020birational}}]
Let $(x_1,x_2,\dots,x_n)$ be any linear extension of $\sfp$.  Then $\rowFpl=t_{x_1} t_{x_2} \cdots t_{x_n}$.
\end{prop}

\begin{prop}[{Joseph~\cite[Thm.~3.21]{joseph2019antichain}}]
Let $(x_1,x_2,\dots,x_n)$ be any linear extension of $\sfp$.  Then $\rowApl=\tau_{x_n} \tau_{x_{n-1}} \cdots \tau_{x_1}$.
\end{prop}

Now we define the PL extension of rowvacuation. We need the rank toggles for these. For $0\leq i \leq r$, the \dfn{PL order filter rank toggle} $\bft_i\colon \calfpl(\sfp)\to \calfpl(\sfp)$ is
\[\bft_i \coloneqq \prod\limits_{p \in \sfp_i} t_p,\]
and the \dfn{PL antichain rank toggle} $\bftau_i\colon \calapl(\sfp)\to \calapl(\sfp)$ is
\[\bftau_i \coloneqq \prod\limits_{p \in \sfp_i} \tau_p.\]
Clearly, $\rowFpl = \bft_0 \bft_1 \cdots \bft_r$ and $\rowApl = \bftau_{r} \bftau_{r-1} \cdots \bftau_0$.

\begin{definition}
\dfn{PL order filter rowvacuation}, $\rvacFpl\colon \calfpl(\sfp) \to \calfpl(\sfp)$, is
\[\rvacFpl \coloneqq (\bft_r) (\bft_{r-1} \bft_r) \cdots (\bft_1 \bft_2 \cdots \bft_{r-1} \bft_r) (\bft_0 \bft_1 \bft_2 \cdots \bft_{r-1} \bft_r).\]
\end{definition}

\begin{definition}
\dfn{PL antichain rowvacuation}, $\rvacApl\colon \calapl(\sfp) \to \calapl(\sfp)$, is
\[\rvacApl \coloneqq (\bftau_r)(\bftau_r \bftau_{r-1})\cdots (\bftau_r \bftau_{r-1} \cdots \bftau_2 \bftau_1) (\bftau_r \bftau_{r-1} \cdots \bftau_2 \bftau_1 \bftau_0).\]
\end{definition}

We can also define $\drvacFpl$ and $\drvacApl$ similarly as before. All of the basic properties we discussed earlier concerning the combinatorial versions of rowmotion and rowvacuation (i.e., \cref{prop:row-rvac-dihedral,prop:rvac_sew_row,prop:rvac-OI-ant} and \cref{lem:rvac_ind}) continue to hold for their PL analogs. Rather than restate all of these properties here, we will state and prove their birational generalizations in \cref{prop:b_dihedral,prop:b_row_iterates,prop:b_com_diagram} and  \cref{lem:b_rvac_ind} below.

\subsection{Birational lifts} 

Now we do everything again in the birational realm. The idea is to \dfn{detropicalize} the above piecewise-linear maps, that is, everywhere replace $\max$ with addition, and addition with multiplication. (Strictly speaking, detropicalization is not well-defined because more identities are satisfied by $(\max,+)$ than by~$(+,\times)$, but in practice there is usually a ``natural'' way to detropicalize a given expression.) In this way we will obtain subtraction-free rational expressions. We could work with $\kk$-valued functions on $\sfp$, where $\kk$ is an arbitrary field of characteristic zero, and treat these expressions as rational maps (i.e., defined outside of a Zariski closed set); this is done in~\cite{grinberg2016birational1, grinberg2015birational2}. However, following~\cite{einstein2018combinatorial}, we find it simpler to work with $\rr_{>0}$-valued functions, for which the values of the subtraction-free rational expressions will be defined everywhere.

So define the following sets of $\rr_{>0}$-valued functions on $\sfphat$ or $\sfp$:
\begin{align*}
\calfb_{\kappa}(\sfp) &\coloneqq \left\{ \pi \in \rr_{>0}^{\sfphat}\colon \pi\big(\hatz\big) = 1, \pi\big(\hato\big)=\kappa\right\}, \\
\caljb_{\kappa}(\sfp) &\coloneqq \left\{ \pi \in \rr_{>0}^{\sfphat}\colon \pi\big(\hatz\big) = \kappa, \pi\big(\hato\big)=1\right\},\\
\calab_{\kappa}(\sfp) &\coloneqq \left\{ \pi \in \rr_{>0}^{\sfp}\right\}.
\end{align*}
Here $\kappa\in\rr_{>0}$ is a parameter; we use $\calfb(\sfp)$, $\caljb(\sfp)$, $\calab(\sfp)$ when its value is clear from context. Also we often implicitly identify all of these sets with $\rr_{>0}^{\sfp}$ by forgetting the values at $\hatz$ and $\hato$.

It is well-known that ``birational identities tropicalize to PL identities,'' although, as mentioned, not vice versa. What this means for us is that if some subtraction-free birational identity holds on all of $\calfb(\sfp)$ then its \dfn{tropicalization} -- the result of replacing addition by $\max$, and replacing multiplication by addition, including replacing the multiplicative identity $1$ by the additive identity $0$ -- holds identically on all of $\calfpl(\sfp)$ (and similarly for the other birational/PL spaces). See~\cite[Lem~7.1]{einstein2018combinatorial} and~\cite[Rem.~10]{grinberg2016birational1} for a precise explanation of tropicalization.

We proceed to describe the birational lifts of rowmotion and toggling introduced by Einstein and Propp~\cite{einstein2018combinatorial}. They will of course tropicalize to their PL analogs. Everything that follows is directly analogous to what we did at the PL level above.

We first explain the birational analog of the definition of rowmotion as the composition of three bijections. These bijections are:
\begin{itemize}
\item \dfn{complementation} $\Theta\colon \rr_{>0}^{\sfphat} \to \rr_{>0}^{\sfphat}$, with $\displaystyle (\Theta\pi)(x) \coloneqq \frac{\kappa}{\pi(x)}$ for all~$x \in \sfphat$;
\item \dfn{up-transfer} $\up \colon\caljb_{\kappa}(\sfp) \to \calab_{\kappa}(\sfp)$, with
\[\displaystyle (\up \pi)(x) \coloneqq \frac{\pi(x)}{\sum_{x\lessdot y} \pi(y)}\]
for all $x\in \sfp$;
\item \dfn{down-transfer} $\down \colon\calfb_{\kappa}(\sfp) \to \calab_{\kappa}(\sfp)$, with 
\[\displaystyle (\down \pi)(x) \coloneqq \frac{\pi(x)}{\sum_{y\lessdot x} \pi(y)}\]
for all $x\in \sfp$.
\end{itemize}
Evidently, $\Theta^{-1}=\Theta$. Also, note that
\[(\up^{-1}\pi)(x) = \sum_{x=y_1\lessdot  \cdots \lessdot y_k\lessdot\hato} \; \prod_{i=1}^{k}\pi(y_i) =\pi(x) \cdot \sum_{x \lessdot y} (\up^{-1}\pi)(y),\]
for all $x\in\sfp$, and similarly,
\[ (\down^{-1}\pi)(x) = \sum_{\hatz\lessdot y_1\lessdot  \cdots \lessdot y_k=x} \; \prod_{i=1}^{k}\pi(y_i) =\pi(x) \cdot \sum_{y \lessdot x} (\up^{-1}\pi)(y),\]
for all $x\in \sfp$.

\begin{definition}
\dfn{Birational order filter rowmotion}, $\rowFb: \calfb(\sfp) \to \calfb(\sfp)$, is $\rowFb \coloneqq \Theta \circ \up^{-1} \circ \down$. 
\end{definition}

\begin{definition}
\dfn{Birational antichain rowmotion}, $\rowAb: \calab(\sfp) \to \calab(\sfp)$, is $\rowAb \coloneqq \down \circ \Theta \circ \up^{-1}$.
\end{definition}

Next, we go over the toggling description of these maps. 

\begin{definition}
Let $p \in \sfp$.  Then the \dfn{birational order filter toggle} at $p$ is the map $t_p \colon \calfb_{\kappa}(\sfp)\to \calfb_{\kappa}(\sfp)$ defined by
\[ (t_p\pi)(x) \coloneqq \begin{cases} \pi(x) &\textrm{if $x\neq p$};\\
\displaystyle \frac{\sum_{y\lessdot x} \pi(y)}{\pi(x) \cdot \sum_{x\lessdot y} \frac{1}{\pi(y)}} &\textrm{if $x=p$},\end{cases}\]
for all $x\in \sfp$.
\end{definition}

\begin{definition}
Let $p \in \sfp$. Then the \dfn{birational antichain toggle} at $p$ is the map $\tau_p \colon \calab_{\kappa}(\sfp)\to \calab_{\kappa}(\sfp)$ defined by
\[(\tau_p\pi)(x) \coloneqq \begin{cases} \pi(x) &\textrm{if $x\neq p$};\\
\displaystyle \frac{\kappa}{\displaystyle \sum_{C \in \mc_p(\sfp)} \prod_{y\in C} \pi(y)} &\textrm{if $x=p$},\end{cases}\]
for all $x\in \sfp$.
\end{definition}

These birational toggles are again involutions and have the same commutativity properties as their combinatorial analogs. The birational versions of rowmotion are built out of these toggles in exactly the same way as in the combinatorial realm, as shown by Einstein--Propp~\cite{einstein2018combinatorial} and Joseph--Roby~\cite{joseph2020birational}:

\begin{prop}[{Einstein--Propp~\cite{einstein2018combinatorial}; c.f.~\cite[Thm.~5.12]{joseph2020birational}}]
Let $(x_1,x_2,\dots,x_n)$ be any linear extension of $\sfp$.  Then $\rowFb=t_{x_1} t_{x_2} \cdots t_{x_n}$.
\end{prop}

\begin{prop}[{Joseph--Roby~\cite[Thm~3.6.]{joseph2020birational}}]
Let $(x_1,x_2,\dots,x_n)$ be any linear extension of $\sfp$.  Then $\rowAb=\tau_{x_n} \tau_{x_{n-1}} \cdots \tau_{x_1}$.
\end{prop}

Now we define the birational extension of rowvacuation. For $0\leq i \leq r$, the \dfn{birational order filter rank toggle} $\bft_i\colon \calfb(\sfp)\to \calfb(\sfp)$ is
\[\bft_i \coloneqq \prod\limits_{p \in \sfp_i} t_p,\]
and the \dfn{birational antichain rank toggle} $\bftau_i\colon \calab(\sfp)\to \calab(\sfp)$ is
\[\bftau_i \coloneqq \prod\limits_{p \in \sfp_i} \tau_p.\]
Clearly, $\rowFb = \bft_0 \bft_1 \cdots \bft_r$ and $\rowAb = \bftau_{r} \bftau_{r-1} \cdots \bftau_0$.

\begin{definition}
\dfn{Birational order filter rowvacuation}, $\rvacFb\colon \calfb(\sfp) \to \calfb(\sfp)$, is
\[\rvacFb \coloneqq (\bft_r) (\bft_{r-1} \bft_r) \cdots (\bft_1 \bft_2 \cdots \bft_{r-1} \bft_r) (\bft_0 \bft_1 \bft_2 \cdots \bft_{r-1} \bft_r).\]
\dfn{Birational order filter dual rowvacuation}, $\drvacFb\colon \calfb(\sfp) \to \calfb(\sfp)$, is
\[\drvacFb \coloneqq (\bft_0) (\bft_1 \bft_0) \cdots (\bft_{r-1} \cdots  \bft_2 \bft_1 \bft_0) (\bft_r \bft_{r-1} \cdots  \bft_2 \bft_1 \bft_0).\]
\end{definition}

\begin{definition}
\dfn{Birational antichain rowvacuation}, $\rvacAb\colon \calab(\sfp) \to \calab(\sfp)$, is
\[\rvacAb \coloneqq (\bftau_r)(\bftau_r \bftau_{r-1})\cdots (\bftau_r \bftau_{r-1} \cdots \bftau_2 \bftau_1) (\bftau_r \bftau_{r-1} \cdots \bftau_2 \bftau_1 \bftau_0).\]
\dfn{Birational antichain dual rowvacuation}, $\drvacApl \colon \calab(\sfp) \to \calab(\sfp)$, is
\[\drvacAb \coloneqq (\bftau_0) (\bftau_0 \bftau_1) \cdots (\bftau_0 \bftau_1 \bftau_2 \cdots \bftau_{r-1}) (\bftau_0 \bftau_1 \bftau_2 \cdots \bftau_{r-1} \bftau_r).\]
\end{definition}

Again we have dualities $^{*}\colon \calfb(\sfp)\to\calfb(\sfp)$ and $^{*}\colon \calab(\sfp)\to\calab(\sfp)$ given by $\pi^* \coloneqq \Theta(\pi)$ and $\pi^*\coloneqq \pi$, respectively; and again we have $\drvacFb(\pi) = \rvacFb(\pi^*)^*$ and $\drvacAb(\pi) = \rvacAb(\pi^*)^*$.

All of the basic properties we discussed earlier concerning the combinatorial versions of rowmotion and rowvacuation (i.e., \cref{prop:row-rvac-dihedral,prop:rvac_sew_row,prop:rvac-OI-ant} and \cref{lem:rvac_ind}) continue to hold for their birational analogs:

\begin{prop} \label{prop:b_dihedral}
For any graded poset $\sfp$ of rank $r$,
\begin{itemize}
    \item $\rvacFb$ and $\drvacFb$ are involutions;
    \item $\rvacFb \circ \rowFb = (\rowFb)^{-1} \circ \rvacFb$;
    \item $\drvacFb \circ \rowFb = (\rowFpl)^{-1} \circ \drvacFb$;
    \item $(\rowFb)^{r+2} = \drvacFb \circ \rvacFb$.
\end{itemize}
\end{prop}

\begin{prop} \label{prop:b_row_iterates}
For $\pi \in \calfb(\sfp)$ and $x\in \sfp_i$, $(\rvacFb\pi)(x)=((\rowFb)^{i+1}\pi)(x)$.
\end{prop}

\begin{prop} \label{prop:b_com_diagram}
The following diagrams commute:
\begin{center}
\phantom{1}\hfill \begin{tikzpicture}[scale=0.8]
\node at (0,1.8) {$\calfb(\sfp)$};
\node at (0,0) {$\calab(\sfp)$};
\node at (3.25,1.8) {$\calfb(\sfp)$};
\node at (3.25,0) {$\calab(\sfp)$};
\draw[semithick, ->] (0,1.3) -- (0,0.5);
\node[left] at (0,0.9) {$\down$};
\draw[semithick, ->] (0.8,0) -- (2.4,0);
\node[below] at (1.5,0) {$\rowAb$};
\draw[semithick, ->] (0.8,1.8) -- (2.4,1.8);
\node[above] at (1.5,1.8) {$\rowFb$};
\draw[semithick, ->] (3.25,1.3) -- (3.25,0.5);
\node[right] at (3.25,0.9) {$\down$};
\end{tikzpicture} \hfill \begin{tikzpicture}[scale=0.8]
\node at (0,1.8) {$\calfb(\sfp)$};
\node at (0,0) {$\calab(\sfp)$};
\node at (3.25,1.8) {$\calfb(\sfp)$};
\node at (3.25,0) {$\calab(\sfp)$};
\draw[semithick, ->] (0,1.3) -- (0,0.5);
\node[left] at (0,0.9) {$\down$};
\draw[semithick, ->] (0.8,0) -- (2.4,0);
\node[below] at (1.5,0) {$\rvacAb$};
\draw[semithick, ->] (0.8,1.8) -- (2.4,1.8);
\node[above] at (1.5,1.8) {$\rvacFb$};
\draw[semithick, ->] (3.25,1.3) -- (3.25,0.5);
\node[right] at (3.25,0.9) {$\down$};
\end{tikzpicture} \hfill \begin{tikzpicture}[scale=0.8]
\node at (0,1.8) {$\calfb(\sfp)$};
\node at (0,0) {$\calab(\sfp)$};
\node at (3.25,1.8) {$\calfb(\sfp)$};
\node at (3.25,0) {$\calab(\sfp)$};
\draw[semithick, ->] (0,1.3) -- (0,0.5);
\node[left] at (0,0.9) {$\up\circ\Theta$};
\draw[semithick, ->] (0.8,0) -- (2.4,0);
\node[below] at (1.5,0) {$\drvacAb$};
\draw[semithick, ->] (0.8,1.8) -- (2.4,1.8);
\node[above] at (1.5,1.8) {$\drvacFb$};
\draw[semithick, ->] (3.25,1.3) -- (3.25,0.5);
\node[right] at (3.25,0.9) {$\up\circ\Theta$};
\end{tikzpicture}
\hfill\phantom{1}
\end{center}
(Note that $\up\circ\Theta = (\rowAb)^{-1} \circ \down$.) Hence the first three bulleted items in \cref{prop:b_dihedral} hold with $\calf$ replaced by $\cala$.
\end{prop}

\begin{lemma} \label{lem:b_rvac_ind}
Let $\pi \in \calab_{\kappa}(\sfp)$ and $x\in \sfp$.
\begin{itemize}
\item If $x \in \sfp_0$, then $(\rvacAb \pi)(x)= (\bftau_0 \pi)(x)$.
\item If $x \in \sfp_{\geq 1}$, then $(\rvacAb \pi)(x)= \big((\rowAb)^{-1}\circ \rvacAb \overline{\pi}\big)(x)$, where $\overline{\pi}$ is the restriction $\overline{\pi} \coloneqq \pi\hspace{-3.5pt} \mid_{\sfp_{\geq 1}} \in \calab_{\kappa}(\sfp_{\geq 1})$.
\end{itemize}
\end{lemma}

We now finally give the proofs of \cref{prop:b_dihedral,prop:b_row_iterates,prop:b_com_diagram} and of \cref{lem:b_rvac_ind} (which will imply via tropicalization and specialization the analogous results at the piecewise-linear and combinatorial level, whose proofs we omitted above).

\begin{proof}[Proof of \cref{prop:b_dihedral}]
 As explained in~\cite{stanley2009promotion}, the promotion and evacuation operators acting on the linear extensions of a poset can be written as compositions of involutions $\bft_i$ which have exactly the same form as those defining rowmotion and rowvacuation. The proof of the analog of \cref{prop:b_dihedral} for promotion and evacuation which Stanley gives in~\cite[Thm.~2.1]{stanley2009promotion} only uses the facts that $\bft_i^2=1$ and $\bft_i\bft_j=\bft_j\bft_i$ for $|i-j|>1$, i.e., it amounts to a computation in the corresponding ``right-angled Coxeter group.'' (Note that these basic properties of the rank toggles, which we stated in \cref{prop:basic-properties-OI-rank-toggles}, continue to hold at the birational level.) Therefore, the proof of \cref{prop:b_dihedral} is the same as the proof of~\cite[Thm.~2.1]{stanley2009promotion}.
\end{proof}

\begin{proof}[Proof of \cref{prop:b_row_iterates}]
First note that
\[ (\rvacFb \pi)(x) = \big((\bft_i \cdots \bft_{r-1} \bft_r) \cdots (\bft_1 \bft_2 \cdots \bft_{r-1} \bft_r) (\bft_0 \bft_1 \bft_2 \cdots \bft_{r-1} \bft_r) \, \pi\big)(x). \]
We will prove the stronger claim that for any $y \in \sfp_j$ with $j \geq i$, 
\[\big((\bft_i \cdots \bft_{r-1} \bft_r) \cdots (\bft_1 \bft_2 \cdots \bft_{r-1} \bft_r) (\bft_0 \bft_1 \bft_2 \cdots \bft_{r-1} \bft_r) \, \pi \big)(y) = ((\rowFb)^{i+1}\pi)(y).\]
This is clear for $i=0$. We proceed by induction on $i$. 

The key point is that when applying a toggle $t_p$ to a $\pi\in \calfb(\sfp)$, all that matters for determining the value $(t_p \, \pi)(p)$ is $\pi(q)$ for $q$ that are either equal to, are covered by, or cover $p$. In particular, for an element $y$ of rank $j$, all that matters is the values at elements of rank $\geq j-1$. By our induction hypothesis, $(\bft_{i-1} \cdots \bft_{r-1} \bft_r) \cdots (\bft_1 \bft_2 \cdots \bft_{r-1} \bft_r) (\bft_0 \bft_1 \bft_2 \cdots \bft_{r-1} \bft_r) \, \pi$ agrees with $(\rowFb)^{i}\pi$ at elements of rank $\geq i-1$. Hence, viewing $\rowFb=\bft_0 \bft_1 \bft_2 \cdots \bft_{r-1} \bft_r$ as a composition of toggles, we see that $(\bft_i \cdots \bft_{r-1} \bft_r) \cdots (\bft_1 \bft_2 \cdots \bft_{r-1} \bft_r) (\bft_0 \bft_1 \bft_2 \cdots \bft_{r-1} \bft_r) \, \pi$ agrees with $(\rowFb)^{i+1}\pi$ at elements of rank $\geq i$, as claimed.
\end{proof}

\begin{proof}[Proof of \cref{prop:b_com_diagram}]
The commutativity of the leftmost diagram is immediate from the definitions of $\rowFb$ and $\rowAb$ as compositions of the maps $\Theta$, $\up$, $\down$, and their inverses. So we proceed to prove the commutativity of the middle and rightmost diagrams.

We will prove the commutativity of the right diagram; the commutativity of the left diagram will then follow from consideration of the dual poset $\sfp^*$. (Note also that we could define $\rowFb$ and $\rowAb$ as compositions of toggles, and prove their conjugacy via $\down$ using an argument similar to what follows.)

There is an isomorphism between the order filter and antichain toggle groups. Specifically, as explained in~\cite[Thm.~3.12]{joseph2020birational}, if we set
\[\bftau_i^* \coloneqq \bft_0 \bft_1 \cdots \bft_{i-1} \bft_i \bft_{i-1} \cdots \bft_1 \bft_0,\]
then the following diagram commutes:
\begin{center}
\begin{tikzpicture}[scale=2/3]
\node at (0,1.8) {$\calfb(\sfp)$};
\node at (0,0) {$\calab(\sfp)$};
\node at (3.25,1.8) {$\calfb(\sfp)$};
\node at (3.25,0) {$\calab(\sfp)$};
\draw[semithick, ->] (0,1.3) -- (0,0.5);
\node[left] at (0,0.9) {$\up\circ\Theta$};
\draw[semithick, ->] (0.8,0) -- (2.4,0);
\node[below] at (1.5,0) {$\bftau_i$};
\draw[semithick, ->] (0.8,1.8) -- (2.4,1.8);
\node[above] at (1.5,1.8) {$\bftau_i^*$};
\draw[semithick, ->] (3.25,1.3) -- (3.25,0.5);
\node[right] at (3.25,0.9) {$\up\circ\Theta$};
\end{tikzpicture}
\end{center}
That is, $\bftau^*_i$ is a composition of order filter rank toggles $\bft_j$ that mimics the action of the antichain rank toggle $\bftau_i$.

Thus, to prove that the right diagram commutes, it suffices to show that
\[(\bftau^*_0) (\bftau^*_0 \bftau^*_1) \cdots (\bftau^*_0 \bftau^*_1 \bftau^*_2 \cdots \bftau^*_{r-1})(\bftau^*_0 \bftau^*_1 \bftau^*_2 \cdots \bftau^*_{r-1} \bftau^*_r)=\drvacFb,\]
where we recall that
\[\drvacFb=(\bft_0) (\bft_1 \bft_0) \cdots (\bft_{r-1} \cdots  \bft_2 \bft_1 \bft_0) (\bft_r \bft_{r-1} \cdots  \bft_2 \bft_1 \bft_0).\]
To do this, we use induction to show that, for any $0\leq k \leq r$,
\[ \bftau^*_0 \bftau^*_1 \cdots \bftau^*_{k-1} \bftau^*_k = \bft_k \bft_{k-1} \cdots  \bft_1 \bft_0.\]
By definition, the base case $\bftau^*_0 = \bft_0$ is true.  Now assume that $\bftau^*_0 \bftau^*_1 \cdots \bftau^*_{k-1} \bftau^*_k = \bft_k \bft_{k-1} \cdots  \bft_1 \bft_0$.  Then
\begin{align*}
    \bftau^*_0 \bftau^*_1 \cdots \bftau^*_{k-1} \bftau^*_k \bftau^*_{k+1} &=
    (\bft_k \bft_{k-1} \cdots  \bft_1 \bft_0)
    (\bft_0 \bft_1 \cdots \bft_{k-1} \bft_k \bft_{k+1} \bft_k \bft_{k-1} \cdots  \bft_1 \bft_0)\\
    &= \bft_{k+1} \bft_k \bft_{k-1} \cdots  \bft_1 \bft_0,
\end{align*}
as required.
\end{proof}

\begin{proof}[Proof of \cref{lem:b_rvac_ind}]

From the definition 
\[\rvacAb = (\bftau_r)(\bftau_r \bftau_{r-1})\cdots (\bftau_r \bftau_{r-1} \cdots \bftau_2 \bftau_1) (\bftau_r \bftau_{r-1} \cdots \bftau_2 \bftau_1 \bftau_0)\]
the statement about when $x \in \sfp_0$ is immediate.  We now focus on the case $x \in \sfp_{\geq 1}$. 

For $0\leq i \leq r$, define
\begin{align*}
\rowFbj{i} &\coloneqq \bft_i \bft_{i+1} \cdots \bft_r, \\
\rvacFbj{i} &\coloneqq (\bft_r) (\bft_{r-1} \bft_r) \cdots (\bft_{i+1} \bft_{i+2} \cdots \bft_{r-1} \bft_r) (\bft_{i} \bft_{i+1} \bft_{i+2} \cdots \bft_{r-1} \bft_r), \\
\rowAbj{i} &\coloneqq  \bftau_{r} \bftau_{r-1} \cdots \bftau_i, \\
\rvacAbj{i} &\coloneqq (\bftau_r)(\bftau_r \bftau_{r-1})\cdots (\bftau_r \bftau_{r-1} \cdots \bftau_{i+2} \bftau_{i+1}) (\bftau_r \bftau_{r-1} \cdots \bftau_{i+2} \bftau_{i+1} \bftau_{i}).
\end{align*}
The same arguments as in the proof of \cref{prop:b_dihedral} imply that 
\[\rvacFbj{i} \circ \rowFbj{i} = (\rowFbj{i})^{-1} \circ \rvacFbj{i}.\]
And the same arguments as in the proof of \cref{prop:b_com_diagram} imply that the following diagrams commute:
\begin{center}
\begin{tikzpicture}[scale=0.8]
\node at (0,1.8) {$\calfb(\sfp)$};
\node at (0,0) {$\calab(\sfp)$};
\node at (3.25,1.8) {$\calfb(\sfp)$};
\node at (3.25,0) {$\calab(\sfp)$};
\draw[semithick, ->] (0,1.3) -- (0,0.5);
\node[left] at (0,0.9) {$\down$};
\draw[semithick, ->] (0.8,0) -- (2.4,0);
\node[below] at (1.5,0) {$\rowAbj{i}$};
\draw[semithick, ->] (0.8,1.8) -- (2.4,1.8);
\node[above] at (1.5,1.8) {$\rowFbj{i}$};
\draw[semithick, ->] (3.25,1.3) -- (3.25,0.5);
\node[right] at (3.25,0.9) {$\down$};
\end{tikzpicture} \qquad \qquad \begin{tikzpicture}[scale=0.8]
\node at (0,1.8) {$\calfb(\sfp)$};
\node at (0,0) {$\calab(\sfp)$};
\node at (3.25,1.8) {$\calfb(\sfp)$};
\node at (3.25,0) {$\calab(\sfp)$};
\draw[semithick, ->] (0,1.3) -- (0,0.5);
\node[left] at (0,0.9) {$\down$};
\draw[semithick, ->] (0.8,0) -- (2.4,0);
\node[below] at (1.5,0) {$\rvacAbj{i}$};
\draw[semithick, ->] (0.8,1.8) -- (2.4,1.8);
\node[above] at (1.5,1.8) {$\rvacFbj{i}$};
\draw[semithick, ->] (3.25,1.3) -- (3.25,0.5);
\node[right] at (3.25,0.9) {$\down$};
\end{tikzpicture}
\end{center}
Hence, we also have
\[\rvacAbj{i} \circ \rowAbj{i} = (\rowAbj{i})^{-1} \circ \rvacAbj{i}.\]

Therefore,
\begin{align*}
\rvacAb &= (\bftau_r)(\bftau_r \bftau_{r-1})\cdots (\bftau_r \bftau_{r-1} \cdots \bftau_2 \bftau_1) (\bftau_r \bftau_{r-1} \cdots \bftau_2 \bftau_1 \bftau_0) \\
&= \rvacAbj{1} \circ \rowAbj{1} \circ \bftau_0 \\
&= (\rowAbj{1})^{-1} \circ \rvacAbj{1} \circ \bftau_0 \\
&= (\bftau_1 \cdots \bftau_{r-1} \bftau_r) \circ  (\bftau_r)(\bftau_r \bftau_{r-1})\cdots (\bftau_r \bftau_{r-1} \cdots \bftau_2 \bftau_1) \circ (\bftau_0).
\end{align*}

Now we come to the key claim in the proof, which is that for any $\sigma \in \calab(\sfp)$ and $p \in \sfp_{\geq 1}$,
\begin{gather*} 
((\bftau_1 \cdots \bftau_{r-1} \bftau_r) (\bftau_r)(\bftau_r \bftau_{r-1})\cdots (\bftau_r \bftau_{r-1} \cdots \bftau_2 \bftau_1) \sigma)(p) \\
= ((\overline{\bftau_1} \cdots \overline{\bftau_{r-1}} \; \overline{\bftau_r}) (\overline{\bftau_r})(\overline{\bftau_r} \; \overline{\bftau_{r-1}})\cdots (\overline{\bftau_r} \; \overline{\bftau_{r-1}} \cdots \overline{\bftau_2} \; \overline{\bftau_1}) \overline{\sigma})(p),
\end{gather*}
where $\overline{\sigma} \coloneqq \sigma\hspace{-3.5pt}\mid_{\sfp_{\geq 1}} \in \calab(\sfp_{\geq 1})$ and the $\overline{\bftau_i}\colon \calab(\sfp_{\geq 1})\to \calab(\sfp_{\geq 1})$ denote the analogous antichain rank toggles for $\sfp_{\geq 1}$. Taking $\sigma \coloneqq \bftau_0 \pi$, this will complete the proof the lemma because it is precisely this composition of the $\overline{\bftau_i}$ which constitute the map $(\rowAb)^{-1}\circ \rvacAb\colon \calab(\sfp_{\geq 1})\to \calab(\sfp_{\geq 1})$.

Actually, we will prove an even stronger claim. Namely: let $\bfT$ be any composition of the $\bftau_i$, where all $i\geq 2$, and $\overline{\bfT}$ the corresponding composition of the $\overline{\bftau_i}$; then we have $(\bftau_1\bfT\bftau_1\sigma)(p)=(\overline{\bftau_1}\,\overline{\bfT}\,\overline{\bftau_1} \; \overline{\sigma})(p)$ for any $\sigma \in \calab(\sfp)$ and $p \in \sfp_{\geq 1}$. Setting
\[\bfT\coloneqq (\overline{\bftau_2} \cdots \overline{\bftau_{r-1}} \; \overline{\bftau_r}) \circ (\overline{\bftau_r})(\overline{\bftau_r} \; \overline{\bftau_{r-1}})\cdots (\overline{\bftau_r} \; \overline{\bftau_{r-1}} \cdots \overline{\bftau_2})\]
recovers the previous claim.

We proceed to prove the stronger claim. For any $p \in \sfp_{\geq 1}$, we have
\[(\bftau_1\sigma)(p) = \begin{cases} \displaystyle\frac{(\overline{\bftau_1}\; \overline{\sigma})(p)}{\sum_{q\lessdot p}\sigma(q)} &\textrm{ if $p\in \sfp_{1}$}; \\ (\overline{\bftau_1} \; \overline{\sigma})(p) &\textrm{if $p\in \sfp_{\geq 2}$}. \end{cases} \]
Moreover, we continue to have
\[(\bfT\bftau_1\sigma)(p) = \begin{cases} \displaystyle\frac{(\overline{\bfT} \; \overline{\bftau_1} \; \overline{\sigma})(p)}{\sum_{q\lessdot p}\sigma(q)} &\textrm{ if $p\in \sfp_{1}$}; \\ (\overline{\bfT} \; \overline{\bftau_1} \; \overline{\sigma})(p) &\textrm{if $p\in \sfp_{\geq 2}$}, \end{cases} \]
for all $p \in \sfp_{\geq 1}$. This can be seen inductively: the point is that whenever a term of $\frac{\ast}{\sum_{r\lessdot q}\sigma(r)}$ for some $q \in \sfp_{\geq 1}$ appears as a result of one of the toggles in $\bfT$, it will come multiplied by a term of $\sum_{r\lessdot q}\sigma(r)$ which cancels the denominator (since all maximal chains that pass through $q$ pass through one of the $r$ with $r\lessdot q$). Finally, when we apply $\bftau_1$ to $\bfT\bftau_1\sigma$, we will cancel all terms of $(\sum_{r\lessdot q}\sigma(r))^{-1}$ for $q \in \sfp_{\geq 1}$. So indeed we will have $(\bftau_1\bfT\bftau_1\sigma)(p)=(\overline{\bftau_1} \; \overline{\bfT} \; \overline{\bftau_1} \; \overline{\sigma})(p)$ for all $p \in \sfp_{\geq 1}$, as claimed.
\end{proof}

\subsection{Homomesies for rowmotion and rowvacuation} \label{subsec:homomesies}

Before we end this section, we want to explain one more fact which holds for all graded posets $\sfp$. This fact is about homomesies for rowmotion and rowvacuation, specifically, about transferring homomesies for rowvacuation to rowmotion. We discussed homomesies briefly in \cref{sec:intro}, but let us review the definition now. 

\begin{definition}
Let $\varphi$ be an invertible operator acting on a set $X$. We say that a statistic $f\colon X \to \rr$ is \dfn{homomesic} with respect to the action of $\varphi$ on $X$ if for every finite\footnote{When $X$ is finite, for example $X=\calf(\sfp)$ or $\cala(\sfp)$, then of course every $\varphi$-orbit will be finite. But, e.g., piecewise-linear and birational rowmotion tend to have infinite order and infinite orbits. We could work with a more robust definition of homomesy which also considers the infinite orbits by taking limits in some way, but then we would have to worry about issues of convergence. However, these issues will not really concern us because, for the very special families of posets that we most care about like $\sfa^n$ and $\rect{a}{b}$, piecewise-linear and birational rowmotion have finite order and hence finite orbits.} $\varphi$-orbit $O$, the average $\frac{1}{\#O} \sum_{x \in O} f(x)$ is equal to the same constant. If this constant is $c\in \rr$ then we say $f$ is \dfn{$c$-mesic}.
\end{definition}

The preceding definition is the correct notion of homomesy for combinatorial and PL maps, but for birational maps we need to work ``multiplicatively.''

\begin{definition}
Let $\varphi$ be an invertible operator acting on a set $X$. We say that a statistic $f\colon X \to \rr_{>0}$ is \dfn{multiplicatively homomesic} with respect to the action of~$\varphi$ on $X$ if for every finite $\varphi$-orbit $O$, the multiplicative average $\left(\prod_{x \in O} f(x)\right)^{\frac{1}{\#O}}$ is equal to the same constant. (Here we take positive $n$th roots.) If this constant is~$c\in \rr_{>0}$ then we say $f$ is \dfn{multiplicatively $c$-mesic}.
\end{definition}

The systematic investigation of homomesies was initiated by Propp and Roby~\cite{propp2015homomesy}. There has been a particular emphasis on exhibiting homomesies for rowmotion, including its piecewise-linear and birational extensions~\cite{panyushev2009orbits, armstrong2013uniform, propp2015homomesy, haddadan2014homomesy, rush2015homomesy, einstein2018combinatorial, musiker2018paths, hopkins2019minuscule, joseph2021birational, okada2020birational}.

As we explained in \cref{sec:intro}, our primary objective in the present paper is to study homomesies for the Lalanne--Kreweras involution (and its PL/birational extensions). In the next section we will prove the Lalanne--Kreweras involution is the same as rowvacuation for the poset $\sfa^n$. The next lemma explains how we can automatically transfer some homomesies from rowvacuation to rowmotion. In this way, our main results in this paper also imply homomesy results for rowmotion of $\sfa^n$.

\begin{lemma} \label{lem:homo_transfer}
Let $\sfp$ be a graded poset of rank $r$.
\begin{itemize}
\item \emph{(Combinatorial version)} For $0\leq i \leq r$, let $g_i\colon \calf(\sfp)\to\rr$ be statistics for which $g_i(F)$ only depends on $F\cap \sfp_i$. Then if $f \coloneqq \sum_{i=0}^{r}g_i$ is $c$-mesic with respect to the action of $\rvacF$, $f$ is also $c$-mesic with respect to~$\rowF$.
\item \emph{(PL version)} For $0\leq i \leq r$, let $g_i\colon \calfpl(\sfp)\to\rr$ be statistics for which $g_i(\pi)$ only depends on $\pi\hspace{-3.5pt}\mid_{\sfp_i}$. Then if $f \coloneqq \sum_{i=0}^{r}g_i$ is $c$-mesic with respect to the action of $\rvacFpl$, $f$ is also $c$-mesic with respect to~$\rowFpl$.
\item \emph{(Birational version)} For $0\leq i \leq r$, let $g_i\colon \calfb(\sfp)\to\rr_{>0}$ be statistics for which $g_i(\pi)$ only depends on $\pi\hspace{-3.5pt}\mid_{\sfp_i}$. Then if $f \coloneqq \prod_{i=0}^{r}g_i$ is multiplicatively $c$-mesic with respect to the action of $\rvacFb$, $f$ is also multiplicatively $c$-mesic with respect to~$\rowFb$.
\end{itemize}
\end{lemma}

\begin{proof}
We prove the birational version.

We have by supposition that for any $\pi\in \calfb(\sfp)$,
\[ \left(\prod_{i=0}^{r}g_i(\pi) \prod_{i=0}^{r} g_i(\rvacFb \pi) \right)^{\frac{1}{2}} = c. \]
Since $g_i(\pi)$ depends only on the values of $\pi$ at the $i$th rank $\sfp_i$, and since \cref{prop:b_row_iterates} tells us that $(\rvacFb\pi)(p)=((\rowFb)^{i+1}\pi)(p)$ for all $p \in \sfp_i$, we have
\[ \left(\prod_{i=0}^{r}g_i(\pi) \prod_{i=0}^{r} g_i((\rowFb)^{i+1} \pi) \right)^{\frac{1}{2}} = c. \]
Now let $O$ be a finite $\rowFb$-orbit. Then from the above we have
\begin{align*}
c &= \prod_{\pi \in O} \left(\prod_{i=0}^{r}g_i(\pi) \prod_{i=0}^{r} g_i((\rowFb)^{i+1} \pi) \right)^{\frac{1}{2\#O}}\\
&= \left(\prod_{\pi \in O} \prod_{i=0}^{r}g_i(\pi) \right)^{\frac{1}{2\#O}} \left(\prod_{\pi \in O} \prod_{i=0}^{r} g_i((\rowFb)^{i+1} \pi) \right)^{\frac{1}{2\#O}} \\
&= \left(\prod_{\pi \in O} \prod_{i=0}^{r}g_i(\pi) \right)^{\frac{1}{2\#O}} \left(\prod_{\pi \in O} \prod_{i=0}^{r}g_i(\pi) \right)^{\frac{1}{2\#O}} = \left(\prod_{\pi \in O} \prod_{i=0}^{r}g_i(\pi) \right)^{\frac{1}{\#O}},
\end{align*}
where from the second to the third lines we used the fact that the product was over an $\rowFb$-orbit, so we are free to shift terms by powers of $\rowFb$. We conclude that $\prod_{i=0}^{r}g_i$ is indeed $c$-mesic for $\rowFb$.
\end{proof}

\begin{remark}
Examples of statistics $f$ satisfying the conditions of \cref{lem:homo_transfer} include (at the birational level):
\begin{itemize}
\item statistics of the form $\pi \mapsto \prod_{p\in \sfp}(\pi(p))^{c_p}$ for coefficients $c_p \in \mathbb{R}$;
\item statistics of the form $\pi \mapsto \prod_{p\in \sfp}(\down\pi(p))^{c_p}$ for coefficients $c_p \in \mathbb{R}$.
\end{itemize}
These include all the major kinds of statistics (such as order filter cardinality, antichain cardinality, etc.) that prior rowmotion homomesy research has focused on.
\end{remark}

\begin{remark}
The argument in the proof of \cref{lem:homo_transfer} is similar to the ``recombination'' technique of Einstein--Propp~\cite{einstein2018combinatorial}.
\end{remark}

\section{The Lalanne--Kreweras involution is rowvacuation} \label{sec:lk_is_rvac}

In this section, we prove that the Lalanne--Kreweras involution is the same as rowvacuation for the poset~$\sfa^n$. We do this by showing that the Lalanne--Kreweras involution satisfies the same recursion as rowvacuation (i.e., \cref{lem:rvac_ind}). We have an obvious isomorphism $\sfa^n_{\geq i} \simeq \sfa^{n-i}$, and via this identification we can consider applying $\lk$ to the restriction $A\cap \sfa^n_{\geq i} \in \cala(\sfa^{n-i})$ of an antichain $A \in \cala(\sfa^n)$. The recursive description of $\lk$ is then given by the following lemma.

\begin{lemma} \label{lem:equiv-def-LK}
Let $A \in \cala(\sfa^n)$ and $p \in \sfa^n$.
\begin{itemize}
\item If $p \in \sfa^n_0$, then $p \in \lk(A)$ if and only if $p \in \bftau_0(A)$ (i.e., if and only if $A$ does not contain any element $q \geq p$).
\item If $p \in \sfa^n_{\geq 1} \simeq \sfa^{n-1}$, then $p \in \lk(A)$ if and only if $p \in (\rowA^{-1} \circ \lk )(\overline{A})$, where $\overline{A} \coloneqq A \cap \sfa^n_{\geq 1} \in \cala(\sfa^{n-1})$.
\end{itemize}
\end{lemma}

Before we prove \cref{lem:equiv-def-LK}, we first go over a detailed example of how it can be used to recursively compute the Lalanne--Kreweras involution, without reference to the definition of $\lk$ we gave in \cref{sec:intro}.

\begin{figure}[b]
\begin{center}
\begin{tikzpicture}[scale=6/7]
\begin{scope}[shift={(0,2)}]
\node at (-3,1.5) {$\sfa^1$:};
\draw [thick] (0,1.5) circle [radius=0.2];
\draw[ultra thick, ->] (1.9,1.5) -- (3.1,1.5);
\draw[red,thick,dashed] (-0.3,1.2) -- (-0.3,1.8) -- (0.3,1.8) -- (0.3,1.2) -- (-0.3,1.2);
\node[above] at (2.5,1.5) {$\LK$};
\end{scope}
\begin{scope}[shift={(5,2)}]
\draw [thick,fill] (0,1.5) circle [radius=0.2];
\draw[ultra thick, ->] (1.9,1.5) -- (3.1,1.5);
\node[above] at (2.5,1.5) {$\rowA^{-1}$};
\end{scope}
\begin{scope}[shift={(10,2)}]
\draw[red,thick] (-0.3,1.2) -- (-0.3,1.8) -- (0.3,1.8) -- (0.3,1.2) -- (-0.3,1.2);
\draw [thick] (0,1.5) circle [radius=0.2];
\end{scope}
\begin{scope}
\node at (-3,1.5) {$\sfa^2$:};
\draw[thick] (-0.15, 1.85) -- (-0.85, 1.15);
\draw[thick] (0.15, 1.85) -- (0.85, 1.15);
\draw [thick] (0,2) circle [radius=0.2];
\draw [thick] (-1,1) circle [radius=0.2];
\draw [thick] (1,1) circle [radius=0.2];
\draw[ultra thick, ->] (1.9,1.5) -- (3.1,1.5);
\draw[red,thick,dashed] (-0.3,1.7) -- (-0.3,2.3) -- (0.3,2.3) -- (0.3,1.7) -- (-0.3,1.7);
\draw[green,thick,dashed] (-1.6,0.7) -- (0,2.4) -- (1.6, 0.7) -- (-1.6,0.7);
\node[above] at (2.5,1.5) {$\LK$};
\end{scope}
\begin{scope}[shift={(5,0)}]
\draw[thick] (-0.15, 1.85) -- (-0.85, 1.15);
\draw[thick] (0.15, 1.85) -- (0.85, 1.15);
\draw [thick] (0,2) circle [radius=0.2];
\draw [thick,fill] (-1,1) circle [radius=0.2];
\draw [thick,fill] (1,1) circle [radius=0.2];
\draw[red,thick] (-0.3,1.7) -- (-0.3,2.3) -- (0.3,2.3) -- (0.3,1.7) -- (-0.3,1.7);
\draw[ultra thick, ->] (1.9,1.5) -- (3.1,1.5);
\node[above] at (2.5,1.5) {$\rowA^{-1}$};
\end{scope}
\begin{scope}[shift={(10,0)}]
\draw[thick] (-0.15, 1.85) -- (-0.85, 1.15);
\draw[thick] (0.15, 1.85) -- (0.85, 1.15);
\draw [thick] (0,2) circle [radius=0.2];
\draw [thick] (-1,1) circle [radius=0.2];
\draw [thick] (1,1) circle [radius=0.2];
\draw[green,thick] (-1.6,0.7) -- (0,2.4) -- (1.6, 0.7) -- (-1.6,0.7);
\end{scope}
\node at (-3,-2) {$\sfa^3$:};
\begin{scope}[shift={(0,-3)}]
\draw[thick] (-0.15, 1.85) -- (-0.85, 1.15);
\draw[thick] (0.15, 1.85) -- (0.85, 1.15);
\draw[thick] (-1.15, 0.85) -- (-1.85, 0.15);
\draw[thick] (-0.85, 0.85) -- (-0.15, 0.15);
\draw[thick] (0.85, 0.85) -- (0.15, 0.15);
\draw[thick] (1.15, 0.85) -- (1.85, 0.15);
\draw [thick] (0,2) circle [radius=0.2];
\draw [thick] (-1,1) circle [radius=0.2];
\draw [thick] (1,1) circle [radius=0.2];
\draw [thick] (-2,0) circle [radius=0.2];
\draw [thick,fill] (0,0) circle [radius=0.2];
\draw [thick,fill] (2,0) circle [radius=0.2];
\draw[ultra thick, ->] (1.9,1) -- (3.1,1);
\draw[green,thick,dashed] (-1.6,0.7) -- (0,2.4) -- (1.6, 0.7) -- (-1.6,0.7);
\draw[blue,thick,dashed] (-2.6,-0.3) -- (0,2.4) -- (2.6,-0.3) -- (-2.6,-0.3);
\node[above] at (2.5,1) {$\LK$};
\end{scope}
\begin{scope}[shift={(5,-3)}]
\draw[thick] (-0.15, 1.85) -- (-0.85, 1.15);
\draw[thick] (0.15, 1.85) -- (0.85, 1.15);
\draw[thick] (-1.15, 0.85) -- (-1.85, 0.15);
\draw[thick] (-0.85, 0.85) -- (-0.15, 0.15);
\draw[thick] (0.85, 0.85) -- (0.15, 0.15);
\draw[thick] (1.15, 0.85) -- (1.85, 0.15);
\draw [thick] (0,2) circle [radius=0.2];
\draw [thick] (-1,1) circle [radius=0.2];
\draw [thick] (1,1) circle [radius=0.2];
\draw [thick,fill] (-2,0) circle [radius=0.2];
\draw [thick] (0,0) circle [radius=0.2];
\draw [thick] (2,0) circle [radius=0.2];
\draw[green,thick] (-1.6,0.7) -- (0,2.4) -- (1.6, 0.7) -- (-1.6,0.7);
\draw[ultra thick, ->] (1.9,1) -- (3.1,1);
\node[above] at (2.5,1) {$\rowA^{-1}$};
\end{scope}
\begin{scope}[shift={(10,-3)}]
\draw[thick] (-0.15, 1.85) -- (-0.85, 1.15);
\draw[thick] (0.15, 1.85) -- (0.85, 1.15);
\draw[thick] (-1.15, 0.85) -- (-1.85, 0.15);
\draw[thick] (-0.85, 0.85) -- (-0.15, 0.15);
\draw[thick] (0.85, 0.85) -- (0.15, 0.15);
\draw[thick] (1.15, 0.85) -- (1.85, 0.15);
\draw [thick] (0,2) circle [radius=0.2];
\draw [thick] (-1,1) circle [radius=0.2];
\draw [thick,fill] (1,1) circle [radius=0.2];
\draw [thick] (-2,0) circle [radius=0.2];
\draw [thick] (0,0) circle [radius=0.2];
\draw [thick] (2,0) circle [radius=0.2];
\draw[blue,thick] (-2.6,-0.3) -- (0,2.4) -- (2.6,-0.3) -- (-2.6,-0.3);
\end{scope}
\node at (-3,-6) {$\sfa^4$:};
\begin{scope}[shift={(1,-6.5)}]
\draw[thick] (-0.15, 1.85) -- (-0.85, 1.15);
\draw[thick] (0.15, 1.85) -- (0.85, 1.15);
\draw[thick] (-1.15, 0.85) -- (-1.85, 0.15);
\draw[thick] (-0.85, 0.85) -- (-0.15, 0.15);
\draw[thick] (0.85, 0.85) -- (0.15, 0.15);
\draw[thick] (1.15, 0.85) -- (1.85, 0.15);
\draw[thick] (0.85, -0.85) -- (0.15, -0.15);
\draw[thick] (1.15, -0.85) -- (1.85, -0.15);
\draw[thick] (2.85, -0.85) -- (2.15, -0.15);
\draw[thick] (-0.85, -0.85) -- (-0.15, -0.15);
\draw[thick] (-1.15, -0.85) -- (-1.85, -0.15);
\draw[thick] (-2.85, -0.85) -- (-2.15, -0.15);
\draw [thick] (0,2) circle [radius=0.2];
\draw [thick] (-1,1) circle [radius=0.2];
\draw [thick] (1,1) circle [radius=0.2];
\draw [thick] (-2,0) circle [radius=0.2];
\draw [thick,fill] (0,0) circle [radius=0.2];
\draw [thick,fill] (2,0) circle [radius=0.2];
\draw [thick] (-3,-1) circle [radius=0.2];
\draw [thick] (-1,-1) circle [radius=0.2];
\draw [thick] (1,-1) circle [radius=0.2];
\draw [thick] (3,-1) circle [radius=0.2];
\draw[blue,thick,dashed] (-2.6,-0.3) -- (0,2.4) -- (2.6,-0.3) -- (-2.6,-0.3);
\end{scope}
\draw[ultra thick, ->] (4.5,-6) -- (5.7,-6);
\node[above] at (5.1,-6) {$\LK$};
\begin{scope}[shift={(9.2,-6.5)}]
\draw[thick] (-0.15, 1.85) -- (-0.85, 1.15);
\draw[thick] (0.15, 1.85) -- (0.85, 1.15);
\draw[thick] (-1.15, 0.85) -- (-1.85, 0.15);
\draw[thick] (-0.85, 0.85) -- (-0.15, 0.15);
\draw[thick] (0.85, 0.85) -- (0.15, 0.15);
\draw[thick] (1.15, 0.85) -- (1.85, 0.15);
\draw[thick] (0.85, -0.85) -- (0.15, -0.15);
\draw[thick] (1.15, -0.85) -- (1.85, -0.15);
\draw[thick] (2.85, -0.85) -- (2.15, -0.15);
\draw[thick] (-0.85, -0.85) -- (-0.15, -0.15);
\draw[thick] (-1.15, -0.85) -- (-1.85, -0.15);
\draw[thick] (-2.85, -0.85) -- (-2.15, -0.15);
\draw [thick] (0,2) circle [radius=0.2];
\draw [thick] (-1,1) circle [radius=0.2];
\draw [thick,fill] (1,1) circle [radius=0.2];
\draw [thick] (-2,0) circle [radius=0.2];
\draw [thick] (0,0) circle [radius=0.2];
\draw [thick] (2,0) circle [radius=0.2];
\draw [thick,fill] (-3,-1) circle [radius=0.2];
\draw [thick] (-1,-1) circle [radius=0.2];
\draw [thick] (1,-1) circle [radius=0.2];
\draw [thick] (3,-1) circle [radius=0.2];
\draw[blue,thick] (-2.6,-0.3) -- (0,2.4) -- (2.6,-0.3) -- (-2.6,-0.3);
\end{scope}
\end{tikzpicture}
\end{center}
\caption{Goes with \cref{ex:no-draw-Dyck} as an example illustrating \cref{lem:equiv-def-LK}.}
\label{fig:comb}
\end{figure}

\begin{example}\label{ex:no-draw-Dyck}
Let $A=\{[2,3],[3,4]\} \in \cala(\sfa^4)$. We depict this antichain below.
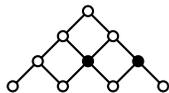
\begin{center}
\begin{tikzpicture}[scale=1/3]
\draw[thick] (-0.15, 1.85) -- (-0.85, 1.15);
\draw[thick] (0.15, 1.85) -- (0.85, 1.15);
\draw[thick] (-1.15, 0.85) -- (-1.85, 0.15);
\draw[thick] (-0.85, 0.85) -- (-0.15, 0.15);
\draw[thick] (0.85, 0.85) -- (0.15, 0.15);
\draw[thick] (1.15, 0.85) -- (1.85, 0.15);
\draw[thick] (0.85, -0.85) -- (0.15, -0.15);
\draw[thick] (1.15, -0.85) -- (1.85, -0.15);
\draw[thick] (2.85, -0.85) -- (2.15, -0.15);
\draw[thick] (-0.85, -0.85) -- (-0.15, -0.15);
\draw[thick] (-1.15, -0.85) -- (-1.85, -0.15);
\draw[thick] (-2.85, -0.85) -- (-2.15, -0.15);
\draw [thick] (0,2) circle [radius=0.2];
\draw [thick] (-1,1) circle [radius=0.2];
\draw [thick] (1,1) circle [radius=0.2];
\draw [thick] (-2,0) circle [radius=0.2];
\draw [thick,fill] (0,0) circle [radius=0.2];
\draw [thick,fill] (2,0) circle [radius=0.2];
\draw [thick] (-3,-1) circle [radius=0.2];
\draw [thick] (-1,-1) circle [radius=0.2];
\draw [thick] (1,-1) circle [radius=0.2];
\draw [thick] (3,-1) circle [radius=0.2];
\end{tikzpicture}
\end{center}
Then $\LK(A)=\{[1,1],[2,4]\}$. We will compute $\LK(A)$ using \cref{lem:equiv-def-LK} and show that we obtain this same antichain. For the four elements in the bottom rank, we use the first bulleted item in \cref{lem:equiv-def-LK}. This tells us that only the leftmost element of the bottom row is in $\LK(A)$.

Now we consider the non-minimal elements. We chop off the bottom rank and obtain the following antichain $\overline{A}$ of $\sfa^3$:
\begin{center}
\begin{tikzpicture}[scale=1/3]
\draw[thick] (-0.15, 1.85) -- (-0.85, 1.15);
\draw[thick] (0.15, 1.85) -- (0.85, 1.15);
\draw[thick] (-1.15, 0.85) -- (-1.85, 0.15);
\draw[thick] (-0.85, 0.85) -- (-0.15, 0.15);
\draw[thick] (0.85, 0.85) -- (0.15, 0.15);
\draw[thick] (1.15, 0.85) -- (1.85, 0.15);
\draw [thick] (0,2) circle [radius=0.2];
\draw [thick] (-1,1) circle [radius=0.2];
\draw [thick] (1,1) circle [radius=0.2];
\draw [thick] (-2,0) circle [radius=0.2];
\draw [thick,fill] (0,0) circle [radius=0.2];
\draw [thick,fill] (2,0) circle [radius=0.2];
\end{tikzpicture}
\end{center}
We compute $\LK$ and then inverse rowmotion on this antichain. If we wish to do this without using our earlier descriptions of $\LK$,  then we use \cref{lem:equiv-def-LK} again, considering separately the bottom rank and the other elements.

Continuing in this way, we actually need to begin with just the top rank.  We start with the empty antichain \begin{tikzpicture}[scale=2/3]
\draw [thick] (0,2) circle [radius=0.2];\end{tikzpicture} of the single-element poset $\sfa^1$ and compute $\LK$ and then inverse rowmotion on this. Then we move up to $\sfa^2$ and so on recursively.  This computation is done in \cref{fig:comb}. We indeed obtain $\LK(A)=\{[1,1],[2,4]\}$.
\end{example}

In order to prove \cref{lem:equiv-def-LK}, we use the following description of inverse rowmotion on $\cala( \sfa^n )$ due to Panyushev~\cite{panyushev2009orbits}.

\begin{prop}[{\cite[Proof of Thm.~3.2]{panyushev2009orbits}}]\label{prop:row-inv}
Let $A=\{[i_1,j_1], [i_2,j_2], \cdots, [i_c, j_c]\}$ with $i_1 < i_2 < \cdots < i_c$ be an antichain of $\sfa^n$. We will represent $A$ by a matrix
\[\begin{bmatrix}
i_1 & i_2 & \cdots & i_c\\
j_1 & j_2 & \cdots & j_c
\end{bmatrix}\]
where each column is an element of $A$. Now consider the matrix
\[\begin{bmatrix}
1 & i_1+1 & i_2+1 & \cdots & i_c+1\\
j_1-1 & j_2-1 & \cdots & j_c-1 & n
\end{bmatrix}.\]
There may be some invalid columns consisting of an entry $k$ above an entry $k-1$.\footnote{If $2\leq k \leq n$, this happens exactly when $[k-1,k-1]$ and $[k,k]$ are both in the original antichain~$A$.  If $k=1$, this happens exactly when $[1,1]\in A$.  If $k=n+1$, this happens exactly when $[n,n]\in A$.} Remove any invalid columns, and the remaining matrix corresponds to $\rowA^{-1}(A)$.
\end{prop}

\begin{example}
Consider the antichain $A=\{[2,3], [4,4],[5,5]\}\in\cala(\sfa^5)$. We take the matrix $\begin{bmatrix} 2&4&5\\3&4&5\end{bmatrix}$ and transform it to $\begin{bmatrix} 1&3&5&6\\2&3&4&5\end{bmatrix}$ as in \cref{prop:row-inv}. However, the rightmost two columns are both invalid.  Removing these, we see that $\rowA^{-1}(A) = \{[1,2],[3,3]\}$, as shown below.
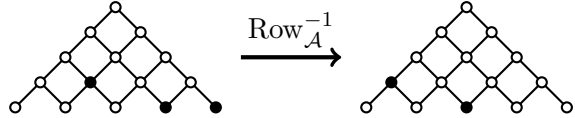
\begin{center}
\begin{tikzpicture}[scale=1/3]
\draw[thick] (-0.15, 1.85) -- (-0.85, 1.15);
\draw[thick] (0.15, 1.85) -- (0.85, 1.15);
\draw[thick] (-1.15, 0.85) -- (-1.85, 0.15);
\draw[thick] (-0.85, 0.85) -- (-0.15, 0.15);
\draw[thick] (0.85, 0.85) -- (0.15, 0.15);
\draw[thick] (1.15, 0.85) -- (1.85, 0.15);
\draw[thick] (0.85, -0.85) -- (0.15, -0.15);
\draw[thick] (1.15, -0.85) -- (1.85, -0.15);
\draw[thick] (2.85, -0.85) -- (2.15, -0.15);
\draw[thick] (-0.85, -0.85) -- (-0.15, -0.15);
\draw[thick] (-1.15, -0.85) -- (-1.85, -0.15);
\draw[thick] (-2.85, -0.85) -- (-2.15, -0.15);
\draw[thick] (0.15, -1.85) -- (0.85, -1.15);
\draw[thick] (1.85, -1.85) -- (1.15, -1.15);
\draw[thick] (2.15, -1.85) -- (2.85, -1.15);
\draw[thick] (3.85, -1.85) -- (3.15, -1.15);
\draw[thick] (-0.15, -1.85) -- (-0.85, -1.15);
\draw[thick] (-1.85, -1.85) -- (-1.15, -1.15);
\draw[thick] (-2.15, -1.85) -- (-2.85, -1.15);
\draw[thick] (-3.85, -1.85) -- (-3.15, -1.15);
\draw [thick] (0,2) circle [radius=0.2];
\draw [thick] (-1,1) circle [radius=0.2];
\draw [thick] (1,1) circle [radius=0.2];
\draw [thick] (-2,0) circle [radius=0.2];
\draw [thick] (0,0) circle [radius=0.2];
\draw [thick] (2,0) circle [radius=0.2];
\draw [thick] (-3,-1) circle [radius=0.2];
\draw [thick,fill] (-1,-1) circle [radius=0.2];
\draw [thick] (1,-1) circle [radius=0.2];
\draw [thick] (3,-1) circle [radius=0.2];
\draw [thick] (-4,-2) circle [radius=0.2];
\draw [thick] (-2,-2) circle [radius=0.2];
\draw [thick] (0,-2) circle [radius=0.2];
\draw [thick,fill] (2,-2) circle [radius=0.2];
\draw [thick,fill] (4,-2) circle [radius=0.2];
\node[above] at (7,0) {$\rowA^{-1}$};
\draw[ultra thick, ->] (5,0) -- (9,0);
\begin{scope}[shift={(14,0)}]
\draw[thick] (-0.15, 1.85) -- (-0.85, 1.15);
\draw[thick] (0.15, 1.85) -- (0.85, 1.15);
\draw[thick] (-1.15, 0.85) -- (-1.85, 0.15);
\draw[thick] (-0.85, 0.85) -- (-0.15, 0.15);
\draw[thick] (0.85, 0.85) -- (0.15, 0.15);
\draw[thick] (1.15, 0.85) -- (1.85, 0.15);
\draw[thick] (0.85, -0.85) -- (0.15, -0.15);
\draw[thick] (1.15, -0.85) -- (1.85, -0.15);
\draw[thick] (2.85, -0.85) -- (2.15, -0.15);
\draw[thick] (-0.85, -0.85) -- (-0.15, -0.15);
\draw[thick] (-1.15, -0.85) -- (-1.85, -0.15);
\draw[thick] (-2.85, -0.85) -- (-2.15, -0.15);
\draw[thick] (0.15, -1.85) -- (0.85, -1.15);
\draw[thick] (1.85, -1.85) -- (1.15, -1.15);
\draw[thick] (2.15, -1.85) -- (2.85, -1.15);
\draw[thick] (3.85, -1.85) -- (3.15, -1.15);
\draw[thick] (-0.15, -1.85) -- (-0.85, -1.15);
\draw[thick] (-1.85, -1.85) -- (-1.15, -1.15);
\draw[thick] (-2.15, -1.85) -- (-2.85, -1.15);
\draw[thick] (-3.85, -1.85) -- (-3.15, -1.15);
\draw [thick] (0,2) circle [radius=0.2];
\draw [thick] (-1,1) circle [radius=0.2];
\draw [thick] (1,1) circle [radius=0.2];
\draw [thick] (-2,0) circle [radius=0.2];
\draw [thick] (0,0) circle [radius=0.2];
\draw [thick] (2,0) circle [radius=0.2];
\draw [thick,fill] (-3,-1) circle [radius=0.2];
\draw [thick] (-1,-1) circle [radius=0.2];
\draw [thick] (1,-1) circle [radius=0.2];
\draw [thick] (3,-1) circle [radius=0.2];
\draw [thick] (-4,-2) circle [radius=0.2];
\draw [thick] (-2,-2) circle [radius=0.2];
\draw [thick,fill] (0,-2) circle [radius=0.2];
\draw [thick] (2,-2) circle [radius=0.2];
\draw [thick] (4,-2) circle [radius=0.2];
\end{scope}
\end{tikzpicture}
\end{center}
\end{example}

\begin{proof}[Proof of \cref{lem:equiv-def-LK}]
Let $A=\{[i_1,j_1], [i_2,j_2], \cdots, [i_c, j_c]\}\in\cala(\sfa^n)$ and $[i,j]\in \sfa^n$. We will first prove the first bulleted item.  Let $[i,j]$ be a minimal element of $\sfa^n$.  So $i=j$. Suppose there is no $v\in A$ with $v\geq [i,i]$. The elements of $\sfa^n$ that are $\geq [i,i]$ are those of the form $[k,\ell]$ with $k\leq i\leq \ell$, so $A$ contains no such element. Thus, there exists $h$ for which $i_s, j_s < i$ for all $s\leq h$, and $i_s, j_s >i$ for all $s\geq h+1$. Now consider $\LK(A)$. We have
\begin{align*}
    \{i'_1 < i'_2 < \cdots < i'_m\} &= [n] \setminus  \{j_1,j_2,\dots,j_c\},\\
    \{j'_1 < j'_2 < \cdots < j'_m\} &= [n] \setminus  \{i_1,i_2,\dots,i_c\}.
\end{align*}
So $i'_s, j'_s<i$ for all $s\leq i-1-h$, $i'_s, j'_s>i$ for all $s\geq i+1-h$, and $i'_{i-h}=j'_{i-h}=i$. So $[i,i]\in \LK(A)$.

On the other hand, suppose $[i,i]\in\LK(A)$. Then there is some $h\in[c]$ such that  $i'_h = j'_h = i$. So $i'_s,j'_s<i$ for all $s<h$ and $i'_s,j'_s>i$ for all $s>h$. Therefore, $i_s,j_s<i$ for all $s\leq i-h$ and $i_s,j_s>i$ for all $s> i-h$. This means that $A$ cannot contain any $[k,\ell]$ with $k\leq i\leq \ell$; these are exactly the elements that are $\geq[i,i]$ in~$\sfa^n$.

Now let us consider elements of $\sfa^n$ that are not minimal; these have the form $[i,j]$ with $j>i$. Name the elements $\overline{A}$ as $[\ibar_1, \jbar_1], [\ibar_2, \jbar_2], \cdots [\ibar_k, \jbar_k]$. Due to the shift in indexing between $\sfa^{n-1}$ and the subposet of $\sfa^n$ containing all non-minimal elements, these are the intervals of the form $[i_h,j_h - 1]$ where $[i_h,j_h]\in A$ with $i_h\not= j_h$. We can represent this as the columns of the matrix
\[\begin{bmatrix}
\ibar_1 & \ibar_2 & \cdots & \ibar_k\\
\jbar_1 & \jbar_2 & \cdots & \jbar_k
\end{bmatrix}.\]
Now $\LK_{\sfa^{n-1}}(\overline{A})$ can be represented by a matrix, where we take the matrix for $\overline{A}$ and we complement the bottom (resp.~top) row from $[n-1]$ and make it the new top (resp.~bottom) row, with each row's elements listed in increasing order. We can write this new matrix as
\[M= \begin{bmatrix}
\ibar'_1 & \ibar'_2 & \cdots & \ibar'_\ell\\
\jbar'_1 & \jbar'_2 & \cdots & \jbar'_\ell
\end{bmatrix}.\]
The top row $\{\ibar'_1, \ibar'_2, \cdots, \ibar'_\ell\}$ of $M$ contains each element $i'_h-1$ except if 1 is some $i'_h$ then it does not contain $1-1=0$.  The bottom row $\{\jbar'_1, \jbar'_2, \cdots, \jbar'_\ell\}$ of $M$ contains each element $j'_h$ except $n$ if $n$ is some $j'_h$. However, since $[\ibar_1, \jbar_1], [\ibar_2, \jbar_2], \cdots [\ibar_k, \jbar_k]$ does not necessarily contain \emph{all} $[i_h,j_h - 1]$, just the ones for which $i_h\not= j_h$, the top row of $M$ can contain extra elements in addition to all $i'_h-1$, and similarly the bottom row of $M$ can contain extra elements that are not $j'_h$.  These extra elements come in pairs $s-1,s$ with an $s-1$ in the top row one spot northwest of an $s$. Thus we apply antichain rowmotion, to $\LK_{\sfa^{n-1}}(\overline{A})$; by \cref{prop:row-inv} we delete the invalid columns of the matrix
\[N= \begin{bmatrix}
1 & \ibar'_1+1 & \ibar'_2+1 & \cdots & \ibar'_\ell+1\\
\jbar'_1-1 & \jbar'_2-1 & \cdots & \jbar'_\ell-1 & n-1
\end{bmatrix}\]
to get $\rowA( \LK_{\sfa^{n-1}}(\overline{A}) )$. All of the ``extra'' elements mentioned previously lie in these deleted columns. We see that, if $i\not=j$, then $[i,j]\in\LK(A)$ if and only if~$[i,j-1]$ is not a deleted column of $N$.
\end{proof}

\begin{thm}\label{thm:lk_is_rvac}
For every $A\in \cala(\sfa^n)$, $\LK(A) = \rvacA(A)$.
\end{thm}

\begin{proof}
This follows immediately from \cref{lem:rvac_ind} and \cref{lem:equiv-def-LK}.
\end{proof}

\begin{remark}
From \cref{thm:lk_is_rvac,prop:rvac-OI-ant} it follows that $\lk \circ \rowA = \rowA^{-1} \circ \lk$. This was proved earlier by Panyushev~\cite[Thm.~3.5]{panyushev2009orbits}.
\end{remark}

So now we have natural candidates for the PL and birational extensions of the Lalanne--Kreweras involution: PL and birational rowvacuation of $\sfa^n$.

\begin{definition}
The \dfn{PL Lalanne--Kreweras involution} is 
\[\lkpl \coloneqq \rvacApl\colon \calapl(\sfa^n) \to \calapl(\sfa^n).\] 
\end{definition}

\begin{definition}
The \dfn{birational Lalanne--Kreweras involution} is 
\[\lkb \coloneqq \rvacAb\colon \calab(\sfa^n) \to \calab(\sfa^n).\]
\end{definition}

From the general properties of rowvacuation we explained in \cref{sec:basics}, we know that $\lkpl$ and $\lkb$ are involutions. By their construction as compositions of toggles, $\lkb$ tropicalizes to $\lkpl$, and (when $\kappa=1$) $\lkpl$ preserves the chain polytope $\calc(\sfa^n)$ and restricts to $\lk$ on the vertices of $\calc(\sfa^n)$.\footnote{In fact, $\lkpl$ restricts to an action on the finite set $\frac{1}{m}\zz^{\sfa^n} \cap \calc(\sfa^n)$ for all $m \in \zz_{>0}$; see also \cref{rem:pl_fixed_points}.}  So we have already established much of \cref{thm:main_intro}; what remains to do is to establish the second bulleted item in that theorem (i.e., the homomesy properties of $\lkb$), which we do in \cref{sec:homomesies}.

\begin{figure}
\begin{center}
\begin{tikzpicture}
\begin{scope}[shift={(0,1.5)}]
\node at (-3,1) {$\sfa^1$:};
\node at (-1,1) {$  z  $};
\draw[ultra thick, ->] (0.6,1) -- (1.8,1);
\node[above] at (1.2,1) {$\lkb$};
\end{scope}
\begin{scope}[shift={(3.4,1)}]
\node at (0,1.5) {$  \frac{\kappa}{z} $};
\draw[ultra thick, ->] (1.6,1.5) -- (2.8,1.5);
\node[above] at (2.2,1.5) {$(\rowAb)^{-1}$};
\end{scope}
\begin{scope}[shift={(7.8,1)}]
\draw[red,thick] (-0.3,1.2) -- (-0.3,1.8) -- (0.3,1.8) -- (0.3,1.2) -- (-0.3,1.2);
\node at (0,1.5) {$  z  $};
\end{scope}
\begin{scope}
\node at (-3,1) {$\sfa^2$:};
\draw[thick] (-1.3, 1.2) -- (-1.7, 0.8);
\draw[thick] (-0.7, 1.2) -- (-0.3, 0.8);
\node at (-1,1.5) {$  z  $};
\node at (-2,0.5) {$  x  $};
\node at (0,0.5) {$  y  $};
\draw[ultra thick, ->] (0.6,1) -- (1.8,1);
\node[above] at (1.2,1) {$\lkb$};
\end{scope}
\begin{scope}[shift={(3.4,-0.5)}]
\draw[thick] (-0.3, 1.7) -- (-0.7, 1.3);
\draw[thick] (0.3, 1.7) -- (0.7, 1.3);
\node at (0,2) {$  z  $};
\node at (-1,1) {$  \frac{\kappa}{xz}  $};
\node at (1,1) {$  \frac{\kappa}{yz}  $};
\draw[red,thick] (-0.3,1.7) -- (-0.3,2.3) -- (0.3,2.3) -- (0.3,1.7) -- (-0.3,1.7);
\draw[ultra thick, ->] (1.6,1.5) -- (2.8,1.5);
\node[above] at (2.2,1.5) {$(\rowAb)^{-1}$};
\end{scope}
\begin{scope}[shift={(7.8,-0.5)}]
\draw[thick] (-0.3, 1.7) -- (-0.7, 1.3);
\draw[thick] (0.3, 1.7) -- (0.7, 1.3);
\node at (0,2) {$  \frac{xy}{x+y}  $};
\node at (-1,1) {$  \frac{(x+y)z}{y} $};
\node at (1,1) {$  \frac{(x+y)z}{x}  $};
\draw[green,thick] (-2.2,0.6) -- (0,2.6) -- (2.2, 0.6) -- (-2.2,0.6);
\end{scope}
\node at (-3,-2) {$\sfa^3$:};
\begin{scope}[shift={(-0.4,-3)},xscale=5/6]
\draw[thick] (-0.3, 1.7) -- (-0.7, 1.3);
\draw[thick] (0.3, 1.7) -- (0.7, 1.3);
\draw[thick] (-1.3, 0.7) -- (-1.7, 0.3);
\draw[thick] (-0.7, 0.7) -- (-0.3, 0.3);
\draw[thick] (0.7, 0.7) -- (0.3, 0.3);
\draw[thick] (1.3, 0.7) -- (1.7, 0.3);
\node at (0,2) {$  z  $};
\node at (-1,1) {$  x  $};
\node at (1,1) {$  y  $};
\node at (-2,0) {$  u  $};
\node at (0,0) {$  v  $};
\node at (2,0) {$  w  $};
\end{scope}
\draw[ultra thick, ->] (1.3,-1.7) -- (2.5,-1.7);
\node[above] at (1.9,-1.7) {$\lkb$};
\begin{scope}[shift={(5,-3)},xscale=5/6]
\draw[thick] (-0.3, 1.7) -- (-0.7, 1.3);
\draw[thick] (0.3, 1.7) -- (0.7, 1.3);
\draw[thick] (-1.3, 0.7) -- (-1.7, 0.3);
\draw[thick] (-0.7, 0.7) -- (-0.3, 0.3);
\draw[thick] (0.7, 0.7) -- (0.3, 0.3);
\draw[thick] (1.3, 0.7) -- (1.7, 0.3);
\draw[green,thick] (-2.5,0.6) -- (0,2.6) -- (2.5, 0.6) -- (-2.5,0.6);
\node at (0,2) {$  \frac{xy}{x+y}  $};
\node at (-1,1) {$  \frac{(x+y)z}{y} $};
\node at (1,1) {$  \frac{(x+y)z}{x}  $};
\node at (-2,0) {$  \frac{\kappa}{uxz}  $};
\node at (0,0) {$  \frac{\kappa}{v(x+y)z} $};
\node at (2,0) {$  \frac{\kappa}{wyz}  $};
\end{scope}
\draw[ultra thick, ->] (5,-3.55) -- (5,-4.45);
\node[left] at (5,-4) {$(\rowAb)^{-1}$};
\begin{scope}[shift={(5,-7.6)},xscale=5/3]
\draw[thick] (-0.3, 1.7) -- (-0.7, 1.3);
\draw[thick] (0.3, 1.7) -- (0.7, 1.3);
\draw[thick] (-1.3, 0.7) -- (-1.7, 0.3);
\draw[thick] (-0.7, 0.7) -- (-0.3, 0.3);
\draw[thick] (0.7, 0.7) -- (0.3, 0.3);
\draw[thick] (1.3, 0.7) -- (1.7, 0.3);
\draw[blue,thick] (-3.1,-0.3) -- (0,3) -- (3.1,-0.3) -- (-3.1,-0.3);
\node at (0,2) {$  \frac{uvw}{uv+uw+vw}  $};
\node at (-1,1) {$  \frac{(uv+uw+vw)xy}{w(ux+vx+vy)} $};
\node at (1,1) {$  \frac{(uv+uw+vw)xy}{u(vx+vy+wy)}  $};
\node at (-2,0) {$  \frac{(ux+vx+vy)z}{vy}  $};
\node at (0,0) {$  \frac{(ux+vx+vy)(vx+vy+wy)z}{(uv+uw+vw)xy} $};
\node at (2,0) {$  \frac{(vx+vy+wy)z}{vx}  $};
\end{scope}
\node at (-3,-9) {$\sfa^4$:};
\begin{scope}[shift={(3,-9.6)},xscale=5/6,yscale=2/3]
\draw[thick] (-0.3, 1.7) -- (-0.7, 1.3);
\draw[thick] (0.3, 1.7) -- (0.7, 1.3);
\draw[thick] (-1.3, 0.7) -- (-1.7, 0.3);
\draw[thick] (-0.7, 0.7) -- (-0.3, 0.3);
\draw[thick] (0.7, 0.7) -- (0.3, 0.3);
\draw[thick] (1.3, 0.7) -- (1.7, 0.3);
\draw[thick] (0.7, -0.7) -- (0.3, -0.3);
\draw[thick] (1.3, -0.7) -- (1.7, -0.3);
\draw[thick] (2.7, -0.7) -- (2.3, -0.3);
\draw[thick] (-0.7, -0.7) -- (-0.3, -0.3);
\draw[thick] (-1.3, -0.7) -- (-1.7, -0.3);
\draw[thick] (-2.7, -0.7) -- (-2.3, -0.3);
\node at (0,2) {$  z  $};
\node at (-1,1) {$  x  $};
\node at (1,1) {$  y  $};
\node at (-2,0) {$  u  $};
\node at (0,0) {$  v  $};
\node at (2,0) {$  w  $};
\node at (-3,-1) {$  q  $};
\node at (-1,-1) {$  r  $};
\node at (1,-1) {$  s  $};
\node at (3,-1) {$  t  $};
\end{scope}
\draw[ultra thick, ->] (3,-10.65) -- (3,-11.35);
\node[left] at (3,-10.9) {$\lkb$};
\begin{scope}[shift={(3,-14.5)},xscale=5/3]
\draw[thick] (-0.3, 1.7) -- (-0.7, 1.3);
\draw[thick] (0.3, 1.7) -- (0.7, 1.3);
\draw[thick] (-1.3, 0.7) -- (-1.7, 0.3);
\draw[thick] (-0.7, 0.7) -- (-0.3, 0.3);
\draw[thick] (0.7, 0.7) -- (0.3, 0.3);
\draw[thick] (1.3, 0.7) -- (1.7, 0.3);
\draw[thick] (0.7, -0.7) -- (0.3, -0.3);
\draw[thick] (1.3, -0.7) -- (1.7, -0.3);
\draw[thick] (2.7, -0.7) -- (2.3, -0.3);
\draw[thick] (-0.7, -0.7) -- (-0.3, -0.3);
\draw[thick] (-1.3, -0.7) -- (-1.7, -0.3);
\draw[thick] (-2.7, -0.7) -- (-2.3, -0.3);
\draw[blue,thick] (-3.1,-0.3) -- (0,3) -- (3.1,-0.3) -- (-3.1,-0.3);
\node at (0,2) {$  \frac{uvw}{uv+uw+vw}  $};
\node at (-1,1) {$  \frac{(uv+uw+vw)xy}{w(ux+vx+vy)} $};
\node at (1,1) {$  \frac{(uv+uw+vw)xy}{u(vx+vy+wy)}  $};
\node at (-2,0) {$  \frac{(ux+vx+vy)z}{vy}  $};
\node at (0,0) {$  \frac{(ux+vx+vy)(vx+vy+wy)z}{(uv+uw+vw)xy} $};
\node at (2,0) {$  \frac{(vx+vy+wy)z}{vx}  $};
\node at (-3,-1) {$  \frac{\kappa}{quxz}  $};
\node at (-1,-1) {$  \frac{\kappa}{r(ux+vx+vy)z}  $};
\node at (1,-1) {$  \frac{\kappa}{s(vx+vy+wy)z}  $};
\node at (3,-1) {$  \frac{\kappa}{twyz}  $};
\end{scope}
\end{tikzpicture}
\end{center}
\caption{$\lkb$ for posets up to $\sfa^4$.}
\label{fig:birational}
\end{figure}

We end this section by giving an example, like \cref{ex:no-draw-Dyck}, of how the recursive descriptions of antichain rowvacuation can be used to compute $\lkb$: in \cref{fig:birational} we illustrate the use of \cref{lem:b_rvac_ind} to compute $\lkb$ on a generic labeling of $\sfa^4$; notice that along the way we compute how $\lkb$ acts on a generic labeling of $\sfa^1$, $\sfa^2$, and $\sfa^3$ as well.

\section{Homomesies for the Lalanne--Kreweras involution} \label{sec:homomesies}

In this section we prove the homomesy results for the piecewise-linear and birational Lalanne--Kreweras involution. For conciseness we work exclusively at the birational level.

Recall from \cref{sec:intro} the statistics $\h_i \colon \cala(\sfa^n)\to \rr$, for $1 \leq i \leq n$, given by
\[\h_i(A) \coloneqq \#\{j\colon [i,j] \in A\} + \#\{j\colon [j,i]\in A\}.\]
And recall that the antichain cardinality and major index statistics are linear combinations of the~$\h_i$:
\[ \#A = \frac{1}{2}(\h_1(A) + \h_2(A) + \cdots + \h_n(A)); \quad \quad \maj(A) = \h_1(A) + 2\h_2(A) + \cdots + n\h_n(A).\]
Any linear combination of homomesies is again a homomesy; so in terms of understanding the homomesies of $\lk$ it suffices to concentrate on the $\h_i$.

We define the birational analogs $\hb_i\colon\calab(\sfa^n)\to \rr$, for $1\leq i \leq n$, to be
\[\hb_i(\pi) \coloneqq \prod_{i\leq j \leq n} \pi([i,j]) \cdot \prod_{1 \leq j \leq i} \pi([j,i]).\]
In exactly the same way that antichain cardinality and major index are linear combinations of the $\h_i$ at the combinatorial level, the birational analogs of antichain cardinality and major index (i.e., the statistics appearing in the second bulleted item in \cref{thm:main_intro}) are multiplicative combinations of the $\hb_i$.

The main result we prove in this section is:

\begin{thm} \label{thm:homo}
The statistics $\hb_i$ are all multiplicatively $\kappa$-mesic with respect to the action of $\lkb$ on $\calab_\kappa(\sfa^n)$. 
\end{thm}

\begin{example}
Consider the generic labeling $\pi\in \calab(\sfa^3)$ shown in Figure~\ref{fig:lk_example}.
As seen in that figure,
\[h_2(\pi) \cdot h_2\big(\lkb(\pi)\big) =
v^2 x y \left( \frac{\kappa}{vz(x+y)} \right)^2 \cdot \frac{z(x+y)}{y} \cdot \frac{z(x+y)}{x}=\kappa^2.\]
\end{example}

Since, as just mentioned, the birational analogs of antichain cardinality and major index are multiplicative combinations of the $\hb_i$, their homomesy under $\lkb$ follows from \cref{thm:homo}. Hence, when we prove \cref{thm:homo} we will have completed the proof of \cref{thm:main_intro}.

As mentioned in \cref{sec:intro}, the combinatorial version of \cref{thm:homo} is easy to see from the definition of $\lk$ we gave in that section. However, we do not know of any straightforward way to see the birational version of \cref{thm:homo}. Our proof will use many intermediary results, including significant results proved in other papers.

Specifically, we will prove \cref{thm:homo} by applying a lot of the machinery that has been developed over the past several years to understand birational rowmotion. In fact, we will mostly employ results proved for rowmotion not on the poset~$\sfa^n$, but on the \dfn{rectangle} poset (also known as the product of two chains) $\rect{a}{b}$. The elements of~$\rect{a}{b}$ are ordered pairs $(i,j)$ for $1 \leq i \leq a$, $1 \leq j \leq b$, with the usual partial order $(i,j) \leq (i',j')$ if $i\leq i'$ and $j \leq j'$. (Observe how we use $(i,j)$ for elements of~$\rect{a}{b}$ but $[i,j]$ for elements of $\sfa^n$.) The poset $\rect{a}{b}$ is graded of rank $a+b-2$, with $\rk(i,j) = i+j-2$. See \cref{fig:rectangle} for a depiction of the rectangle, and in particular contrast its coordinate system with that of $\sfa^n$, which we saw in \cref{fig:dyck_bij}.

\begin{figure}
\begin{tikzpicture}[scale=2/3]
\draw[thick] (-0.3, 1.7) -- (-0.7, 1.3);
\draw[thick] (0.3, 1.7) -- (0.7, 1.3);
\draw[thick] (-1.3, 0.7) -- (-1.7, 0.3);
\draw[thick] (-0.7, 0.7) -- (-0.3, 0.3);
\draw[thick] (0.7, 0.7) -- (0.3, 0.3);
\draw[thick] (1.3, 0.7) -- (1.7, 0.3);
\draw[thick] (0.7, -0.7) -- (0.3, -0.3);
\draw[thick] (1.3, -0.7) -- (1.7, -0.3);
\draw[thick] (-0.7, -0.7) -- (-0.3, -0.3);
\draw[thick] (-1.3, -0.7) -- (-1.7, -0.3);
\draw[thick] (0.7, -1.3) -- (0.3, -1.7);
\draw[thick] (-0.7, -1.3) -- (-0.3, -1.7);
\node at (0,2) {\small $(3,3)$};
\node at (-1,1) {\small $(3,2)$};
\node at (1,1) {\small $(2,3)$};
\node at (-2,0) {\small $(3,1)$};
\node at (0,0) {\small $(2,2)$};
\node at (2,0) {\small $(1,3)$};
\node at (-1,-1) {\small $(2,1)$};
\node at (1,-1) {\small $(1,2)$};
\node at (0,-2) {\small $(1,1)$};
\end{tikzpicture}
\caption{How we draw $\rect{3}{3}$: note in particular the coordinates.} \label{fig:rectangle}
\end{figure}
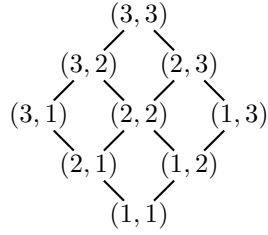

The rectangle is the poset that has received the most attention with respect to rowmotion. Grinberg and Roby~\cite{grinberg2015birational2} proved the following remarkable ``reciprocity theorem'' concerning birational rowmotion of $\rect{a}{b}$:

\begin{thm}[{Grinberg and Roby~\cite[Thm.~32]{grinberg2015birational2}}] \label{thm:reciprocity}
For $\pi \in \calfb_{\kappa}(\rect{a}{b})$, we have
\[ \pi(a+1-i,b+1-j) = \frac{\kappa}{\big((\rowFb)^{i+j-1}\pi\big)(i,j)} \]
for all $(i,j)\in\rect{a}{b}$.
\end{thm}

\begin{remark} \label{rem:rect_rvac}
Thanks to~ \cref{prop:b_row_iterates}, \cref{thm:reciprocity} is equivalent to the statement that $\rvacFb$ on $\calfb_{\kappa}(\rect{a}{b})$ is the map $\pi \mapsto \kappa\cdot ({\textrm{$180^{\circ}$ rotation of $\pi$}})^{-1}$. With \cref{prop:b_dihedral}, this implies that $\rowFb$ on $\calfb_{\kappa}(\rect{a}{b})$ has order dividing $a+b$. See also the discussion in \cref{subsec:other_rvac}.
\end{remark}

In order to use results about rowmotion of $\rect{a}{b}$, like \cref{thm:reciprocity}, to say something about rowmotion of $\sfa^n$, Grinberg and Roby considered a certain embedding of $\sfa^n$ into $\rect{n+1}{n+1}$. 

Specifically, define $\iab \colon \calfb_{\kappa}(\sfa^n) \to \calfb_{4\kappa^2}(\rect{n+1}{n+1})$ by
\[ (\iab \pi)(i,j) = \begin{cases} 4\kappa \cdot \pi([n+2-i,j-1])  &\textrm{if $i+j > n+2$}; \\
2\kappa &\textrm{if $i+j=n+2$}; \\
\kappa \cdot \big((\rvacFb \pi)([j,n+1-i])\big)^{-1} &\textrm{if $i+j < n+2$},
 \end{cases}\]
for all $(i,j) \in \rect{n+1}{n+1}$. In other words, to embed $\pi$ into $\rect{n+1}{n+1}$, ranks $0$ through $n-1$ are filled with (an upside-down copy of) $\kappa \cdot (\rvacFb \pi)^{-1}$, rank $n$ is filled with $2\kappa$, and ranks $n+1$ through $2n$ are filled with $4\kappa \cdot \pi$.

\begin{figure}
\begin{tikzpicture}
\node at (0,1.5) {\begin{tikzpicture}
\node at (-1.5,1.5) {$\pi=$};
\draw[thick] (-0.3, 1.7) -- (-0.7, 1.3);
\draw[thick] (0.3, 1.7) -- (0.7, 1.3);
\node at (0,2) {\small $z$};
\node at (-1,1) {\small $x$};
\node at (1,1) {\small $y$};
\end{tikzpicture}};

\node at (0,-1.5) {\begin{tikzpicture}
\node at (-2.2,1.5) {$\rvacFb\pi=$};
\draw[thick] (-0.3, 1.7) -- (-0.7, 1.3);
\draw[thick] (0.3, 1.7) -- (0.7, 1.3);
\node at (0,2) {$\frac{\kappa(x+y)}{xy}$};
\node at (-1,1) {$\frac{\kappa(x+y)}{xz}$};
\node at (1,1) {$\frac{\kappa(x+y)}{yz}$};
\end{tikzpicture}};

\node at (7,0){\begin{tikzpicture}
\node at (-3,1.5) {$4\kappa\pi\to$};
\node at (-3,-1.5) {\raisebox{\depth}{\scalebox{1}[-1]{$\kappa(\rvacFb\pi)^{-1}\to$}}};
\draw[thick] (-0.3, 1.7) -- (-0.7, 1.3);
\draw[thick] (0.3, 1.7) -- (0.7, 1.3);
\draw[thick] (-1.3, 0.7) -- (-1.7, 0.3);
\draw[thick] (-0.7, 0.7) -- (-0.3, 0.3);
\draw[thick] (0.7, 0.7) -- (0.3, 0.3);
\draw[thick] (1.3, 0.7) -- (1.7, 0.3);
\draw[thick] (0.7, -0.7) -- (0.3, -0.3);
\draw[thick] (1.3, -0.7) -- (1.7, -0.3);
\draw[thick] (-0.7, -0.7) -- (-0.3, -0.3);
\draw[thick] (-1.3, -0.7) -- (-1.7, -0.3);
\draw[thick] (0.7, -1.3) -- (0.3, -1.7);
\draw[thick] (-0.7, -1.3) -- (-0.3, -1.7);
\node at (0,2) {\small $4\kappa z$};
\node at (-1,1) {\small $4\kappa x$};
\node at (1,1) {\small $4\kappa y$};
\node at (-2,0) {\small $2\kappa$};
\node at (0,0) {\small $2\kappa$};
\node at (2,0) {\small $2\kappa$};
\node at (-1,-1) { $\frac{xz}{x+y}$};
\node at (1,-1) { $\frac{yz}{x+y}$};
\node at (0,-2) { $\frac{xy}{x+y}$};
\node at (3,0) {$=\iab\pi$};
\end{tikzpicture}};
\end{tikzpicture}
\caption{The embedding $\iab$ of $\sfa^2$ into $\rect{3}{3}$.} \label{fig:embedding}
\end{figure}

\begin{example} \label{ex:embedding}
\Cref{fig:embedding} depicts the embedding $\iab\colon \calfb(\sfa^2)\to\calfb(\rect{3}{3})$ for a generic $\pi \in \calfb(\sfa^2)$. Observe how to make $\iab\pi$: a copy of $4\kappa\pi$ is pasted above the middle rank, the middle rank is filled with $2\kappa$, and an upside-down copy of $\kappa(\rvacFb\pi)^{-1}$ is pasted below the middle rank.
\end{example}

Let $\flip\colon \sfa^n \to \sfa^n$ be the poset automorphism defined as $\flip([i,j]) \coloneqq [j,n+1-i]$ for all $[i,j]\in \sfa^n$ (this is reflection across the vertical axis of symmetry the way we have been drawing $\sfa^n$). Similarly, define $\flip\colon \rect{n+1}{n+1} \to \rect{n+1}{n+1}$ to be the poset automorphism with $\flip(i,j)\coloneqq (j,i)$ for all $(i,j) \in \rect{n+1}{n+1}$ (this is again reflection across the vertical axis of symmetry). These automorphisms extend in the obvious way to functions on these posets.

The point of $\iab$ is the following lemma.

\begin{lemma}[{cf., Grinberg and Roby~\cite[Lem.~67]{grinberg2015birational2}}] \label{lem:embedding} The following properties are satisfied by the embedding $\iab \colon \calfb_{\kappa}(\sfa^n) \to \calfb_{4\kappa^2}(\rect{n+1}{n+1})$:
\begin{itemize}
\item it is equivariant with respect to $\flip$ and $\rowFb$, i.e.,
\begin{align*}
\flip \circ \iab &= \iab \circ \flip, \\
\rowFb \circ \iab &= \iab \circ \rowFb;
\end{align*}
\item its image is $\{\pi \in \calfb_{4\kappa^2}(\rect{n+1}{n+1})\colon (\rowFb)^{n+1}(\pi) = \flip(\pi)\}$.
\end{itemize}
\end{lemma}

\begin{remark} \label{rem:2s}
Note the factors of $2$ and $4$ which appear in the definition of the embedding $\iab$. These factors appear because the elements in $\rect{n+1}{n+1}$ of rank $n+1$ cover two elements of rank $n$; and similarly the elements of rank $n-1$ are covered by two elements of rank $n$.
\end{remark}

\begin{remark}
The embedding $\iab$ is similar in spirit to a staircase-into-rectangle embedding considered by Pon and Wang~\cite{pon2011promotion}. However, in that paper they worked with linear extensions (a.k.a., standard Young tableaux) under promotion and evacuation, rather than order ideals under rowmotion and rowvacuation as we do here.
\end{remark}

Since \cref{lem:embedding} is essentially proved in~\cite{grinberg2015birational2}, we just provide a sketch here.

\begin{proof}[Proof sketch of \cref{lem:embedding}]
For the first bulleted item: the $\flip$-equivariance is clear. The $\rowFb$-equivariance is also easily verified by using the description of $\rvacFb$ in \cref{prop:b_row_iterates} and considering carefully what applying each rank toggle does (while in particular bearing \cref{rem:2s} in mind).

For the second bulleted item: this follows from \cref{thm:reciprocity}, together with, again, \cref{prop:b_row_iterates}.
\end{proof}

An important consequence of \cref{lem:embedding} is that $(\rowFb)^{n+1}=\flip$ for $\sfa^n$.

\medskip

Remember that (from the point of view of \cref{thm:homo}) we are really interested in $\calab(\sfa^n)$ and not $\calfb(\sfa^n)$. In order to bring $\calab(\sfa^n)$ into the picture, we are going to need another result from the literature, this time one proved by the second author and Roby~\cite{joseph2021birational}.

For $\pi\in \calab_{\kappa}(\rect{a}{b})$, we define the \dfn{Stanley--Thomas word} of $\pi$, $\st(\pi)\in\rr_{>0}^{a+b}$, to be the $(a+b)$-tuple of positive real numbers whose $i$th entry is 
\[ \st(\pi)_{i} \coloneqq \begin{cases} \prod_{j=1}^{b} \pi(i,j) &\textrm{if $1\leq i \leq a$}; \\
\kappa/\big(\prod_{j=1}^{a}\pi(j,i-a)\big) &\textrm{if $a+1\leq i \leq a+b$}. \end{cases}\]
At the combinatorial level, the Stanley--Thomas word was defined, briefly, by Stanley in~\cite{stanley2009promotion} and further elucidated by Hugh Thomas; it was used by Propp and Roby~\cite{propp2015homomesy} to prove homomesy results for rowmotion acting on $\cala(\sfa^n)$. It was then lifted to the birational level by the second author and Roby~\cite{joseph2021birational}. The importance of the Stanley--Thomas word is the following theorem.

\begin{thm}[{Joseph--Roby~\cite[Thm.~3.10]{joseph2021birational}}] \label{thm:st_rot}
The Stanley--Thomas word rotates under rowmotion, i.e., for $\pi\in \calab_{\kappa}(\rect{a}{b})$,
\[ \st(\rowAb\pi)_i =\begin{cases} \st(\pi)_{a+b} &\textrm{if $i=1$}; \\ \st(\pi)_{i-1} &\textrm{if $2\leq i \leq a+b$}.\end{cases} \]
\end{thm}

The following proposition and corollary explain why the Stanley--Thomas word is so useful for our purposes.

\begin{prop} \label{prop:st_flip}
For $\pi \in \calab_{\kappa}(\rect{n}{n})$ we have
\[ \st(\flip\pi)_i =  \begin{cases} \kappa/\st(\pi)_{n+i} &\textrm{if $1\leq i \leq n$}; \\ \kappa/\st(\pi)_{i-n} &\textrm{if $n+1\leq i \leq 2n$}. \end{cases} \]
\end{prop}
\begin{proof}
This is immediate from the definition of the Stanley--Thomas word.
\end{proof}

\begin{cor} \label{cor:st_constant}
For $\pi \in \calab_{\kappa}(\rect{n}{n})$ with $(\rowAb)^{n}(\pi)=\flip(\pi)$, we have that $\st(\pi) = (\sqrt{\kappa},\sqrt{\kappa},\ldots,\sqrt{\kappa})$ is constantly equal to $\sqrt{\kappa}$. In particular, for $\pi \in \calfb_{\kappa}(\sfa^n)$, we have that $\st(\down \iab\pi)=(2\kappa,2\kappa,\ldots,2\kappa)$ is constantly equal to $2\kappa$.
\end{cor}
\begin{proof}
The first sentence is a straightforward combination of \cref{thm:st_rot} and \cref{prop:st_flip}. The second sentence then follows from \cref{lem:embedding} (bearing in mind that $\down$ commutes with rowmotion and with $\flip$).
\end{proof}

\begin{figure}
\begin{tikzpicture}
\node at (0,1.5) {\begin{tikzpicture}
\node at (2,1.5) {$=\down\pi$};
\draw[thick] (-0.3, 1.7) -- (-0.7, 1.3);
\draw[thick] (0.3, 1.7) -- (0.7, 1.3);
\node at (0,2) { $\frac{z}{x+y}$};
\node at (-1,1) {\small $x$};
\draw[ultra thick, blue] (-1,1) circle [radius=0.3];
\node at (1,1) {\small $y$};
\draw[ultra thick, red,rotate around={-45:(0.4,1.5)}] (0.4,1.5) ellipse (1 and 0.5);
\end{tikzpicture}};

\node at (0,-1.5) {\begin{tikzpicture}
\node at (2.7,1.5) {$=\down\rvacFb\pi$};
\draw[thick] (-0.3, 1.7) -- (-0.7, 1.3);
\draw[thick] (0.3, 1.7) -- (0.7, 1.3);
\node at (0,2) {$\frac{z}{x+y}$};
\node at (-1,1) {$\frac{\kappa(x+y)}{xz}$};
\draw[ultra thick,green,rotate around={45:(-0.5,1.5)}] (-0.5,1.5) ellipse (1.5 and 0.6);
\node at (1,1) {$\frac{\kappa(x+y)}{yz}$};
\draw[ultra thick, orange] (1,1) circle [radius=0.7];
\end{tikzpicture}};

\node at (-7,0){\begin{tikzpicture}
\draw[thick] (-0.3, 1.7) -- (-0.7, 1.3);
\draw[thick] (0.3, 1.7) -- (0.7, 1.3);
\draw[thick] (-1.3, 0.7) -- (-1.7, 0.3);
\draw[thick] (-0.7, 0.7) -- (-0.3, 0.3);
\draw[thick] (0.7, 0.7) -- (0.3, 0.3);
\draw[thick] (1.3, 0.7) -- (1.7, 0.3);
\draw[thick] (0.7, -0.7) -- (0.3, -0.3);
\draw[thick] (1.3, -0.7) -- (1.7, -0.3);
\draw[thick] (-0.7, -0.7) -- (-0.3, -0.3);
\draw[thick] (-1.3, -0.7) -- (-1.7, -0.3);
\draw[thick] (0.7, -1.3) -- (0.3, -1.7);
\draw[thick] (-0.7, -1.3) -- (-0.3, -1.7);
\node at (0,2) { $\frac{z}{x+y}$};
\node at (-1,1) {\small $x$};
\draw[ultra thick, blue] (-1,1) circle [radius=0.3];
\node at (1,1) {\small $y$};
\draw[ultra thick, red,rotate around={-45:(0.4,1.5)}] (0.4,1.5) ellipse (1 and 0.5);
\node at (-2,0) { $\frac{2\kappa(x+y)}{xz}$};
\node at (0,0) { $\frac{2\kappa}{z}$};
\node at (2,0) { $\frac{2\kappa(x+y)}{yz}$};
\draw[ultra thick, orange] (2,0) circle [radius=0.7];
\node at (-1,-1) { $\frac{z}{y}$};
\node at (1,-1) { $\frac{z}{x}$};
\draw[ultra thick,green,rotate around={-45:(0.4,-0.5)}] (0.4,-0.5) ellipse (1 and 0.5);
\node at (0,-2) { $\frac{xy}{x+y}$};
\node at (-3.5,0) {$\down\iab\pi=$};
\end{tikzpicture}};
\end{tikzpicture}
\caption{How the embedding $\iab$ interacts with $\down$.} \label{fig:embedding_down}
\end{figure}

\begin{example}
This is a continuation of \cref{ex:embedding}. See what $\down\iab\pi$ looks like for a generic element $\pi \in \calfb(\sfa^2)$ on the left in \cref{fig:embedding_down}. Observe how the product of entries along any $45^{\circ}$ or $-45^{\circ}$ diagonal is $2\kappa$, in agreement with \cref{cor:st_constant}.
\end{example}

We now have a good understanding of $\down\iab\pi$ for $\pi \in \calfb(\sfa^n)$, and we know that $\iab\pi$ is built out of $\pi$ and $\rvacFb\pi$; to finish the proof of \cref{thm:homo}, we need to understand how $\down\iab\pi$ relates to $\down\pi$ and $\down\rvacFb\pi$. This is explained by the following proposition.

\begin{prop} \label{prop:down_technical} 
Let $\pi \in \calfb_{\kappa}(\sfa^n)$. Then for all $1\leq i \leq n$ we have:
\begin{enumerate}[(a)]
\item $\prod_{i\leq j \leq n} \down\pi([i,j])=\prod_{i\leq j \leq n} \down\iab\pi(n+2-i,j+1)$;
\item $\prod_{1\leq j \leq i} \down\pi([j,i])=\prod_{1\leq j \leq i} \down\iab\pi(n+2-j,i+1)$;
\item $2\prod_{i\leq j \leq n} \down\rvacFb\pi([i,j])=\prod_{i \leq j \leq n} \down\iab\pi(n+1-j,i+1)$;
\item $2\prod_{1\leq j \leq i} \down\rvacFb\pi([j,i])=\prod_{1\leq j \leq i} \down\iab\pi(n+2-i,j)$.
\end{enumerate}
\end{prop} 

\begin{example} \label{ex:down_technical}
This is a continuation of \cref{ex:embedding}. We verify part~(b) of \cref{prop:down_technical} in the case $n=2$ by looking at \cref{fig:embedding_down} and checking that the product of the entries circled in blue in $\down\iab\pi$ is the same as the product of the entries circled in blue in $\down\pi$, and similarly for the entries circled in red. Actually, the reader will observe that the top two ranks of $\down\iab\pi$ are exactly $\down\pi$. Next we verify part~(c) of \cref{prop:down_technical} by checking that the product of the entries circled in green in~$\down\iab\pi$ is twice the product of the entries circled in green in $\down\rvacFb\pi$, and similarly for the entries circled in orange. Now the reader will observe that the bottom three ranks of $\down\iab\pi$ are \emph{not} the same as $\down\rvacFb\pi$ (for one thing, there are only two ranks in $\down\rvacFb\pi$). Nonetheless, the products along the diagonals still work out exactly as claimed in \cref{prop:down_technical}.
\end{example}

\begin{proof}[Proof of \cref{prop:down_technical}]
For (a) and (b): as mentioned in \cref{ex:down_technical}, it is not hard to see that we in fact have $\down\pi([i,j])=\down\iab\pi(n+2-i,j+1)$ for all $[i,j]\in\sfa^n$. This immediately yields the claim.

For (c) and (d): here we need to be a little more careful. It is easy to see that (c) and (d) are equivalent via $\flip$; so let us prove (c). Set $\pi' := \rvacFb \pi$. We compute
\[
\prod_{i\leq j \leq n} \hspace{-5pt} \down\rvacFb\pi([i,j]) = \hspace{-5pt} \prod_{i \leq j \leq n} \hspace{-5pt} \down\pi'([i,j])  = \frac{\prod_{i \leq j \leq n} \pi'([i,j])}{\prod_{i \leq j \leq n-1} \big(\pi'([i,j])+\pi'([i+1,j+1])\big)}.
\]
Then, denoting $Z \coloneqq \prod_{i \leq j \leq n} \down\iab\pi(n+1-j,i+1)$, we compute
\begin{align*}
Z &= \frac{2\kappa \cdot \prod_{i+1\leq j\leq n-1}\iab\pi(n+1-j,i+1)}{\prod_{i\leq j\leq n-1}\big(\iab\pi(n+1-j,i) + \iab\pi(n-j,i+1)\big)\cdot \iab\pi(1,i)}
\\&= \frac{2\kappa \cdot \prod_{i\leq j\leq n-1}\iab\pi(n-j,i+1)}{\prod_{i\leq j\leq n-1}\big(\iab\pi(n+1-j,i) + \iab\pi(n-j,i+1)\big)\cdot \iab\pi(1,i)}\\
&=  \frac{2\kappa \cdot \prod_{i\leq j \leq n-1}\big(\kappa\pi'([i+1,j+1])^{-1}\big)}{\prod_{i \leq j \leq n-1} \big(\kappa\pi'([i,j])^{-1}+\kappa\pi'([i+1,j+1])^{-1}\big) \cdot \kappa \pi'([i,n])^{-1}} \\
&=  \frac{2 \cdot \prod_{i\leq j \leq n-1}\pi'([i+1,j+1])^{-1}}{\prod_{i \leq j \leq n-1} \big(\pi'([i,j])^{-1}+\pi'([i+1,j+1])^{-1}\big) \cdot \pi'([i,n])^{-1}} \\
&=   \frac{2 \cdot \prod_{i\leq j \leq n-1}\big(\pi'([i,j]) \; \pi'([i+1,j+1])\big) \cdot \pi'([i,n])}{\prod_{i \leq j \leq n-1} \big(\pi'([i,j])+\pi'([i+1,j+1])\big) \prod_{i\leq j \leq n-1}\pi'([i+1,j+1])} \\
&=  \frac{2\cdot \prod_{i \leq j \leq n} \pi'([i,j])}{\prod_{i \leq j \leq n-1} \big(\pi'([i,j])+\pi'([i+1,j+1])\big)} \\
&= 2 \cdot \prod_{i\leq j \leq n} \down\rvacFb\pi([i,j]),
\end{align*}
where to go from the fourth to fifth lines we used the identity $\frac{1}{a^{-1}+b^{-1}}=\frac{ab}{a+b}$, and to get the last line we used $\down\pi'([i,j]) = \frac{\pi'([i,j])}{\pi'([i-1,j])+\pi'([i,j+1])}$ for all $[i,j]\in\sfa^{n}$.
\end{proof}  

We can now prove \cref{thm:homo}.

\begin{proof}[Proof of \cref{thm:homo}]
So let $\pi \in \calab(\sfa^n)$ and let $1\leq i \leq n$. Set $\widetilde{\pi} := \down^{-1}(\pi)$. Applying parts~(a) and~(d) of \cref{prop:down_technical} to $\widetilde{\pi}$, we have
\begin{align*}
2 \hspace{-5pt} \prod_{i\leq j \leq n}  \hspace{-5pt} \pi([i,j])  \hspace{-5pt} \prod_{1\leq j \leq i}  \hspace{-5pt}\lkb \pi([j,i]) &=  \hspace{-5pt} \prod_{i\leq j \leq n}  \hspace{-5pt} \down\iab\widetilde{\pi}(n+2-i,j+1) \hspace{-5pt} \prod_{1\leq j \leq i}  \hspace{-5pt} \down\iab\widetilde{\pi}(n+2-i,j) \\
= \prod_{1\leq j \leq n+1} \down\iab\widetilde{\pi}(n+2-i,j) &= \st_{n+2-i}(\down\iab\widetilde{\pi}) = 2\kappa,
\end{align*}
where in the last equality we used \cref{cor:st_constant}. Analogously, we have
\[ 2 \prod_{i\leq j \leq n} \lkb\pi([i,j])) \prod_{1\leq j \leq i} \pi([j,i]) = 2\kappa.\]
Hence,
\[ 4 \hb_i(\pi) \hb_i(\lkb\pi) = 4\kappa^2,\]
and so indeed $\hb_i$ is multiplicatively $\kappa$-mesic with respect to the action of $\lkb$.
\end{proof}

We have now completed the proof of all the statements in \cref{thm:main_intro}.

\medskip

We conclude this section by discussing homomesies for rowmotion. In \cref{subsec:homomesies} we saw how homomesies for rowvacuation automatically yield homomesies for rowmotion. Applying \cref{lem:homo_transfer} to \cref{thm:homo} yields the following corollary.

\begin{cor} \label{cor:row_homo} \hfill
The statistics $\hb_i$ are all multiplicatively $\kappa$-mesic with respect to the action of $\rowAb$ on $\calab_\kappa(\sfa^n)$.
\end{cor}

Of course, \cref{cor:row_homo} also implies that a multiplicative combination of the $\hb_i$ (such as the birational antichain cardinality or major index) is multiplicatively homomesic under $\rowAb$. And, via tropicalization and specialization, we can say similarly for $\rowApl$ and $\rowA$.

The homomesy of the antichain cardinality statistic for rowmotion acting on the Type~A root poset (in fact, for \emph{all} root posets) was conjectured by Panyushev~\cite{panyushev2009orbits} and proved by Armstrong, Stump, and Thomas~\cite{armstrong2013uniform}. This was one of the main motivating examples for the introduction of the concept of homomesy in~\cite{propp2015homomesy}. The birational version of this antichain cardinality homomesy result is proved in~\cite{hopkins2019minuscule}. The homomesy of the major index for rowmotion was observed by Jim Propp (private communication) and inspired some of our present research. We later learned that the $\h_i$ homomesies for rowmotion were previously observed, conjecturally, by David Einstein. We remark that another way to prove the $\h_i$ homomesies for rowmotion is to write them as a linear combination of ``rook'' statistics plus a linear combination of signed toggleability statistics, as discussed in~\cite{chan2017expected} (see also~\cite{defant2021homomesy}).

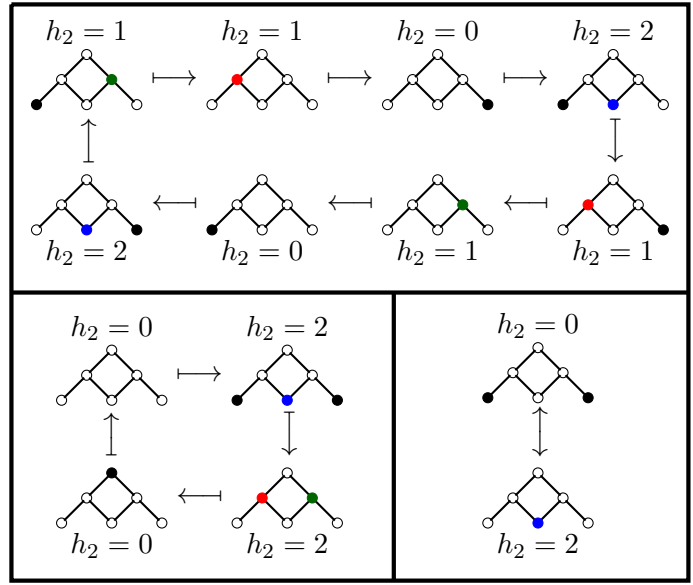
\begin{figure}
\begin{center}
\begin{tikzpicture}[scale=1/3]
\begin{scope}[shift={(0,1)}]
\draw[thick] (-0.1, 1.9) -- (-0.9, 1.1);
\draw[thick] (0.1, 1.9) -- (0.9, 1.1);
\draw[thick] (-1.1, 0.9) -- (-1.9, 0.1);
\draw[thick] (-0.9, 0.9) -- (-0.1, 0.1);
\draw[thick] (0.9, 0.9) -- (0.1, 0.1);
\draw[thick] (1.1, 0.9) -- (1.9, 0.1);
\draw (0,2) circle [radius=0.2];
\draw (-1,1) circle [radius=0.2];
\draw[green,fill] (1,1) circle [radius=0.2];
\draw[fill] (-2,0) circle [radius=0.2];
\draw (0,0) circle [radius=0.2];
\draw (2,0) circle [radius=0.2];
\node[above] at (0,2) {$\h_2=1$};
\end{scope}
\node at (3.5,2) {$\longmapsto$};
\begin{scope}[shift={(7,1)}]
\draw[thick] (-0.1, 1.9) -- (-0.9, 1.1);
\draw[thick] (0.1, 1.9) -- (0.9, 1.1);
\draw[thick] (-1.1, 0.9) -- (-1.9, 0.1);
\draw[thick] (-0.9, 0.9) -- (-0.1, 0.1);
\draw[thick] (0.9, 0.9) -- (0.1, 0.1);
\draw[thick] (1.1, 0.9) -- (1.9, 0.1);
\draw (0,2) circle [radius=0.2];
\draw[red,fill] (-1,1) circle [radius=0.2];
\draw (1,1) circle [radius=0.2];
\draw (-2,0) circle [radius=0.2];
\draw (0,0) circle [radius=0.2];
\draw (2,0) circle [radius=0.2];
\node[above] at (0,2) {$\h_2=1$};
\end{scope}
\node at (10.5,2) {$\longmapsto$};
\begin{scope}[shift={(14,1)}]
\draw[thick] (-0.1, 1.9) -- (-0.9, 1.1);
\draw[thick] (0.1, 1.9) -- (0.9, 1.1);
\draw[thick] (-1.1, 0.9) -- (-1.9, 0.1);
\draw[thick] (-0.9, 0.9) -- (-0.1, 0.1);
\draw[thick] (0.9, 0.9) -- (0.1, 0.1);
\draw[thick] (1.1, 0.9) -- (1.9, 0.1);
\draw (0,2) circle [radius=0.2];
\draw (-1,1) circle [radius=0.2];
\draw (1,1) circle [radius=0.2];
\draw (-2,0) circle [radius=0.2];
\draw (0,0) circle [radius=0.2];
\draw[fill] (2,0) circle [radius=0.2];
\node[above] at (0,2) {$\h_2=0$};
\end{scope}
\node at (17.5,2) {$\longmapsto$};
\begin{scope}[shift={(21,1)}]
\draw[thick] (-0.1, 1.9) -- (-0.9, 1.1);
\draw[thick] (0.1, 1.9) -- (0.9, 1.1);
\draw[thick] (-1.1, 0.9) -- (-1.9, 0.1);
\draw[thick] (-0.9, 0.9) -- (-0.1, 0.1);
\draw[thick] (0.9, 0.9) -- (0.1, 0.1);
\draw[thick] (1.1, 0.9) -- (1.9, 0.1);
\draw (0,2) circle [radius=0.2];
\draw (-1,1) circle [radius=0.2];
\draw (1,1) circle [radius=0.2];
\draw[fill] (-2,0) circle [radius=0.2];
\draw[blue,fill] (0,0) circle [radius=0.2];
\draw (2,0) circle [radius=0.2];
\node[above] at (0,2) {$\h_2=2$};
\end{scope}

\begin{scope}[shift={(0,-4)}]
\draw[thick] (-0.1, 1.9) -- (-0.9, 1.1);
\draw[thick] (0.1, 1.9) -- (0.9, 1.1);
\draw[thick] (-1.1, 0.9) -- (-1.9, 0.1);
\draw[thick] (-0.9, 0.9) -- (-0.1, 0.1);
\draw[thick] (0.9, 0.9) -- (0.1, 0.1);
\draw[thick] (1.1, 0.9) -- (1.9, 0.1);
\draw (0,2) circle [radius=0.2];
\draw (-1,1) circle [radius=0.2];
\draw (1,1) circle [radius=0.2];
\draw (-2,0) circle [radius=0.2];
\draw[blue,fill] (0,0) circle [radius=0.2];
\draw[fill] (2,0) circle [radius=0.2];
\node[below] at (0,0) {$\h_2=2$};
\end{scope}
\node at (3.5,-3) {$\longmapsfrom$};
\begin{scope}[shift={(7,-4)}]
\draw[thick] (-0.1, 1.9) -- (-0.9, 1.1);
\draw[thick] (0.1, 1.9) -- (0.9, 1.1);
\draw[thick] (-1.1, 0.9) -- (-1.9, 0.1);
\draw[thick] (-0.9, 0.9) -- (-0.1, 0.1);
\draw[thick] (0.9, 0.9) -- (0.1, 0.1);
\draw[thick] (1.1, 0.9) -- (1.9, 0.1);
\draw (0,2) circle [radius=0.2];
\draw (-1,1) circle [radius=0.2];
\draw (1,1) circle [radius=0.2];
\draw[fill] (-2,0) circle [radius=0.2];
\draw (0,0) circle [radius=0.2];
\draw (2,0) circle [radius=0.2];
\node[below] at (0,0) {$\h_2=0$};
\end{scope}
\node at (10.5,-3) {$\longmapsfrom$};
\begin{scope}[shift={(14,-4)}]
\draw[thick] (-0.1, 1.9) -- (-0.9, 1.1);
\draw[thick] (0.1, 1.9) -- (0.9, 1.1);
\draw[thick] (-1.1, 0.9) -- (-1.9, 0.1);
\draw[thick] (-0.9, 0.9) -- (-0.1, 0.1);
\draw[thick] (0.9, 0.9) -- (0.1, 0.1);
\draw[thick] (1.1, 0.9) -- (1.9, 0.1);
\draw (0,2) circle [radius=0.2];
\draw (-1,1) circle [radius=0.2];
\draw[green,fill] (1,1) circle [radius=0.2];
\draw (-2,0) circle [radius=0.2];
\draw (0,0) circle [radius=0.2];
\draw (2,0) circle [radius=0.2];
\node[below] at (0,0) {$\h_2=1$};
\end{scope}
\node at (17.5,-3) {$\longmapsfrom$};
\begin{scope}[shift={(21,-4)}]
\draw[thick] (-0.1, 1.9) -- (-0.9, 1.1);
\draw[thick] (0.1, 1.9) -- (0.9, 1.1);
\draw[thick] (-1.1, 0.9) -- (-1.9, 0.1);
\draw[thick] (-0.9, 0.9) -- (-0.1, 0.1);
\draw[thick] (0.9, 0.9) -- (0.1, 0.1);
\draw[thick] (1.1, 0.9) -- (1.9, 0.1);
\draw (0,2) circle [radius=0.2];
\draw[red,fill] (-1,1) circle [radius=0.2];
\draw (1,1) circle [radius=0.2];
\draw (-2,0) circle [radius=0.2];
\draw (0,0) circle [radius=0.2];
\draw[fill] (2,0) circle [radius=0.2];
\node[below] at (0,0) {$\h_2=1$};
\end{scope}
\node at (20.8,0.5) {$\longmapsdown$};
\node at (-0.2,0.5) {$\longmapsup$};
\begin{scope}[shift={(1,-10.8)}]
\draw[thick] (-0.1, 1.9) -- (-0.9, 1.1);
\draw[thick] (0.1, 1.9) -- (0.9, 1.1);
\draw[thick] (-1.1, 0.9) -- (-1.9, 0.1);
\draw[thick] (-0.9, 0.9) -- (-0.1, 0.1);
\draw[thick] (0.9, 0.9) -- (0.1, 0.1);
\draw[thick] (1.1, 0.9) -- (1.9, 0.1);
\draw (0,2) circle [radius=0.2];
\draw (-1,1) circle [radius=0.2];
\draw (1,1) circle [radius=0.2];
\draw (-2,0) circle [radius=0.2];
\draw (0,0) circle [radius=0.2];
\draw (2,0) circle [radius=0.2];
\node[above] at (0,2) {$\h_2=0$};
\end{scope}
\node at (4.5,-9.8) {$\longmapsto$};
\begin{scope}[shift={(8,-10.8)}]
\draw[thick] (-0.1, 1.9) -- (-0.9, 1.1);
\draw[thick] (0.1, 1.9) -- (0.9, 1.1);
\draw[thick] (-1.1, 0.9) -- (-1.9, 0.1);
\draw[thick] (-0.9, 0.9) -- (-0.1, 0.1);
\draw[thick] (0.9, 0.9) -- (0.1, 0.1);
\draw[thick] (1.1, 0.9) -- (1.9, 0.1);
\draw (0,2) circle [radius=0.2];
\draw (-1,1) circle [radius=0.2];
\draw (1,1) circle [radius=0.2];
\draw[fill] (-2,0) circle [radius=0.2];
\draw[blue,fill] (0,0) circle [radius=0.2];
\draw[fill] (2,0) circle [radius=0.2];
\node[above] at (0,2) {$\h_2=2$};
\end{scope}
\node at (7.8,-11.2) {$\longmapsdown$};
\begin{scope}[shift={(8,-15.7)}]
\draw[thick] (-0.1, 1.9) -- (-0.9, 1.1);
\draw[thick] (0.1, 1.9) -- (0.9, 1.1);
\draw[thick] (-1.1, 0.9) -- (-1.9, 0.1);
\draw[thick] (-0.9, 0.9) -- (-0.1, 0.1);
\draw[thick] (0.9, 0.9) -- (0.1, 0.1);
\draw[thick] (1.1, 0.9) -- (1.9, 0.1);
\draw (0,2) circle [radius=0.2];
\draw[red,fill] (-1,1) circle [radius=0.2];
\draw[green,fill] (1,1) circle [radius=0.2];
\draw (-2,0) circle [radius=0.2];
\draw (0,0) circle [radius=0.2];
\draw (2,0) circle [radius=0.2];
\node[below] at (0,0) {$\h_2=2$};
\end{scope}
\node at (0.7,-11.2) {$\longmapsup$};
\begin{scope}[shift={(1,-15.7)}]
\draw[thick] (-0.1, 1.9) -- (-0.9, 1.1);
\draw[thick] (0.1, 1.9) -- (0.9, 1.1);
\draw[thick] (-1.1, 0.9) -- (-1.9, 0.1);
\draw[thick] (-0.9, 0.9) -- (-0.1, 0.1);
\draw[thick] (0.9, 0.9) -- (0.1, 0.1);
\draw[thick] (1.1, 0.9) -- (1.9, 0.1);
\draw[fill] (0,2) circle [radius=0.2];
\draw (-1,1) circle [radius=0.2];
\draw (1,1) circle [radius=0.2];
\draw (-2,0) circle [radius=0.2];
\draw (0,0) circle [radius=0.2];
\draw (2,0) circle [radius=0.2];
\node[below] at (0,0) {$\h_2=0$};
\end{scope}
\node at (4.5,-14.7) {$\longmapsfrom$};
\begin{scope}[shift={(18,-10.7)}]
\draw[thick] (-0.1, 1.9) -- (-0.9, 1.1);
\draw[thick] (0.1, 1.9) -- (0.9, 1.1);
\draw[thick] (-1.1, 0.9) -- (-1.9, 0.1);
\draw[thick] (-0.9, 0.9) -- (-0.1, 0.1);
\draw[thick] (0.9, 0.9) -- (0.1, 0.1);
\draw[thick] (1.1, 0.9) -- (1.9, 0.1);
\draw (0,2) circle [radius=0.2];
\draw (-1,1) circle [radius=0.2];
\draw (1,1) circle [radius=0.2];
\draw[fill] (-2,0) circle [radius=0.2];
\draw (0,0) circle [radius=0.2];
\draw[fill] (2,0) circle [radius=0.2];
\node[above] at (0,2) {$\h_2=0$};
\end{scope}
\node at (17.8,-11) {\begin{rotate}{-90}$\longleftrightarrow$\end{rotate}};
\begin{scope}[shift={(18,-15.7)}]
\draw[thick] (-0.1, 1.9) -- (-0.9, 1.1);
\draw[thick] (0.1, 1.9) -- (0.9, 1.1);
\draw[thick] (-1.1, 0.9) -- (-1.9, 0.1);
\draw[thick] (-0.9, 0.9) -- (-0.1, 0.1);
\draw[thick] (0.9, 0.9) -- (0.1, 0.1);
\draw[thick] (1.1, 0.9) -- (1.9, 0.1);
\draw (0,2) circle [radius=0.2];
\draw (-1,1) circle [radius=0.2];
\draw (1,1) circle [radius=0.2];
\draw (-2,0) circle [radius=0.2];
\draw[blue,fill] (0,0) circle [radius=0.2];
\draw (2,0) circle [radius=0.2];
\node[below] at (0,0) {$\h_2=2$};
\end{scope}

\draw[ultra thick] (-3,5) -- (24,5) -- (24,-18) -- (-3,-18) -- (-3,5);

\draw[ultra thick] (-3,-6.5) -- (24,-6.5);
\draw[ultra thick] (12.25,-6.5) -- (12.25,-18);

\end{tikzpicture}
\end{center}

\caption{The three orbits of antichain rowmotion (in the combinatorial realm) on $\sfa^4$, showing the homomesy of $\h_2$ as proved in \cref{cor:row_homo}.} \label{fig:h2}
\end{figure}

\begin{example}
\Cref{fig:h2} illustrates \cref{cor:row_homo}, at the combinatorial level, for~$\sfa^4$: observe $\h_2 = {\color{red}\mathds{1}_{[1,2]}} + 2\cdot {\color{blue}\mathds{1}_{[2,2]}} + {\color{green}\mathds{1}_{[2,3]}}$ has average 1 across each orbit. (Here for $p\in\sfp$ we use $\mathds{1}_p$ to be the indicator function of a poset element, i.e., $\mathds{1}_p(A)$ is equal to $1$ if $p\in A$ and is equal to $0$ otherwise.)
\end{example}

\begin{remark}
These $\h_i$ statistics are directly analogous to statistics studied in~\cite{einstein2016noncrossing}. There it was shown that they are homomesic for an action defined as a product of toggles on noncrossing partitions. The antichains of $\sf{A}_n$ correspond to ``nonnesting'' partitions. Note that both kinds of partitions are counted by the Catalan numbers.
\end{remark}

\section{Rowvacuation for the poset \texorpdfstring{$\sfb^n$}{Bn}} \label{sec:bn}

Let~$G$ be a subgroup of $\aut(\sfp)$, the group of poset automorphisms of a poset $\sfp$. The \dfn{quotient poset} $\sfp/G$ is the poset whose elements are $G$-orbits of $\sfp$, with $Gp \leq Gq$ whenever $p\leq q \in \sfp$. (Here $Gp$ is the orbit $Gp\coloneqq \{g\cdot p\colon g\in G\}$.) It is well-known and easy to see that this indeed defines a partial order.

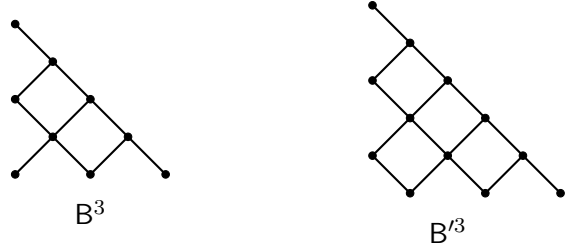
\begin{figure}
\begin{tikzpicture}
\node at (0,0) {\begin{tikzpicture}[scale=0.5]
\draw[fill] (0,0) circle[radius=0.1];
\draw[fill] (2,0) circle[radius=0.1];
\draw[fill] (4,0) circle[radius=0.1];
\draw[fill] (1,1) circle[radius=0.1];
\draw[fill] (3,1) circle[radius=0.1];
\draw[fill] (0,2) circle[radius=0.1];
\draw[fill] (2,2) circle[radius=0.1];
\draw[fill] (1,3) circle[radius=0.1];
\draw[fill] (0,4) circle[radius=0.1];
\draw[thick] (0,0)--(1,1)--(2,0)--(3,1)--(4,0);
\draw[thick] (0,2)--(1,1)--(2,2)--(3,1);
\draw[thick] (0,2)--(1,3)--(2,2);
\draw[thick] (0,4)--(1,3);
\node at (2,-1) {$\sfb^3$};
\end{tikzpicture}};
\node at (5,0) {\begin{tikzpicture}[scale=0.5]
\draw[fill] (1,-1) circle[radius=0.1];
\draw[fill] (3,-1) circle[radius=0.1];
\draw[fill] (5,-1) circle[radius=0.1];
\draw[fill] (0,0) circle[radius=0.1];
\draw[fill] (2,0) circle[radius=0.1];
\draw[fill] (4,0) circle[radius=0.1];
\draw[fill] (1,1) circle[radius=0.1];
\draw[fill] (3,1) circle[radius=0.1];
\draw[fill] (0,2) circle[radius=0.1];
\draw[fill] (2,2) circle[radius=0.1];
\draw[fill] (1,3) circle[radius=0.1];
\draw[fill] (0,4) circle[radius=0.1];
\draw[thick] (0,0)--(1,-1)--(2,0)--(3,-1)--(4,0)--(5,-1);
\draw[thick] (0,0)--(1,1)--(2,0)--(3,1)--(4,0);
\draw[thick] (0,2)--(1,1)--(2,2)--(3,1);
\draw[thick] (0,2)--(1,3)--(2,2);
\draw[thick] (0,4)--(1,3);
\node at (2,-2) {$\sfb'^3$};
\end{tikzpicture}};
\end{tikzpicture}
\caption{The posets $\sfb^3$ and $\sfb'^3$.} \label{fig:b_posets}
\end{figure}

The poset $\sfa^n$ has a nontrivial automorphism, $\flip$, which we can quotient by to produce an interesting poset. Specifically, we define 
\[\sfb^n \coloneqq \sfa^{2n-1}/\langle\flip\rangle; \qquad \qquad \sfb'^n \coloneqq \sfa^{2n}/\langle\flip\rangle.\]
These posets are depicted in \cref{fig:b_posets}. The poset $\sfb^n$ is the root poset of the Type~B root system. The poset~$\sfb'^n$ is not a root poset but shares many similar properties. The poset $\sfb^n$ (respectively~$\sfb'^n$) is graded of rank $2n-2$ (resp.~$2n-1$), with rank functions induced from $\sfa^{2n-1}$ (resp.~$\sfa^{2n}$). 

We can use our knowledge of $\sfa^n$ to say something about these quotient posets. We will be more abridged in our discussions of these posets than we were with $\sfa^n$ above: for instance, we will not separately emphasize corollaries at the combinatorial level; and we will not discuss rowmotion (although rowmotion for these posets has been studied~\cite{panyushev2009orbits, armstrong2013uniform, grinberg2015birational2, hopkins2019minuscule, hopkins2020cyclic}). We will focus on rowvacuation of $\sfb^n$ and $\sfb'^n$.

The key to studying rowvacuation of these quotient posets is the following embedding (cf.~\cite[\S11]{grinberg2015birational2}). Define $\ibpl\colon \calfpl_{\kappa}(\sfb^n)\to\calfpl_{\kappa}(\sfa^{2n-1})$ by 
\[(\ibpl \pi)(p) \coloneqq \pi(\langle \flip \rangle p), \]
for all $p \in \sfa^{2n-1}$. Define $\ibpl\colon \calfpl_{\kappa}(\sfb'^n)\to\calfpl_{\kappa}(\sfa^{2n})$ similarly.

\begin{lemma} \label{lem:b_embedding_pl}
The embedding $\ibpl\colon \calfpl_{\kappa}(\sfb^n)\to\calfpl_{\kappa}(\sfa^{2n-1})$ is $\bft_i$-equivariant for any $0\leq i \leq 2n-2$; in particular, it is $\rowFpl$- and $\rvacFpl$-equivariant. Similarly for the embedding $\ibpl\colon \calfpl_{\kappa}(\sfb'^n)\to\calfpl_{\kappa}(\sfa^{2n})$.
\end{lemma}
\begin{proof}
This is straightforward.
\end{proof}

For $1\leq i \leq 2n-1$, define $\hpl_i\colon \calapl(\sfb^n)\to \rr$ by
\[\hpl_i(\pi) \coloneqq \sum_{i\leq j \leq 2n-1} \pi(\langle\flip\rangle[i,j]) + \sum_{1 \leq j \leq i} \pi(\langle \flip \rangle[j,i]).\]
Note that $\hpl_i = \hpl_{2n-i}$. Define $\hpl_i\colon \sfb'^n\to \rr$, for $1\leq i \leq 2n$, similarly.

\begin{cor} \label{cor:b_homo_pl}
The statistics $\hpl_i\colon \calapl_{\kappa}(\sfb^n)\to \rr$, for $1\leq i \leq 2n-1$, are $\kappa$-mesic for $\rvacApl$. The same is true for the $\hpl_i\colon \calapl_{\kappa}(\sfb'^n)\to \rr$.
\end{cor}
\begin{proof}
For $\pi \in \calapl(\sfb^n)$, we have
\[(\down\ibpl\down^{-1}\pi) ([i,j]) = \pi(\langle \flip \rangle[i,j]), \]
for all $[i,j] \in \sfa^{2n-1}$. Hence $\hpl_i(\pi)=\hpl_i(\down\ibpl\down^{-1}\pi)$, and so the result follows from \cref{thm:homo} and \cref{lem:b_embedding_pl}.
\end{proof}

We conjecture an additional rowvacuation homomesy for $\sfb^n$ beyond those stated in \cref{cor:b_homo_pl}. Namely, define $\hpl_*\colon \calapl(\sfb^n)\to \rr$ by
\[\hpl_*(\pi) \coloneqq \sum_{\substack{p\in \sfa^{2n-1},\\ \flip(p)=p}} \pi(\langle\flip\rangle p).\]

\begin{conj} \label{conj:b_homo_pl}
The statistic $\hpl_*\colon \calapl_{\kappa}(\sfb^n)\to \rr$ is $\kappa$-mesic for $\rvacApl$.
\end{conj}

\begin{remark}
Observe that for any $\pi \in \calapl(\sfb^n)$, we have
\[ \sum_{p\in \sfb^{n}} \pi(p) = \frac{1}{4}\left( \hpl_1+\hpl_2+\cdots +\hpl_{2n-1} + \hpl_*\right).\]
So \cref{conj:b_homo_pl} would imply that $\pi \mapsto \sum_{p\in \sfb^{n}} \pi(p)$ (i.e., the PL analog of antichain cardinality) is $\frac{n\kappa}{2}$-mesic for $\rvacApl$ acting on $\calapl(\sfb^n)$. At the combinatorial level, this was shown by Panyushev~\cite[\S5]{panyushev2004adnilpotent}. Note that for $\sfb'^n$, antichain cardinality is \emph{not} homomesic under rowvacuation.
\end{remark}

We can try to repeat all of the above at the birational level. However, there is a technical obstruction having to do with ``factors of $2$'' (we saw factors of $2$ were also an issue in \cref{rem:2s}). Thus, we are only able to replicate the arguments at the birational level for $\sfb'^n$ and not for $\sfb^n$.

So define $\ibb\colon \calfb_{\kappa}(\sfb'^n)\to\calfb_{\kappa/2}(\sfa^{2n})$ by 
\[(\ibb \pi)(p) \coloneqq \pi(\langle \flip \rangle p), \]
for all $p \in \sfa^{2n}$.

\begin{lemma} \label{lem:b_embedding_b}
The embedding $\ibb\colon \calfb_{\kappa}(\sfb'^n)\to\calfb_{\kappa/2}(\sfa^{2n})$ is $\bft_i$-equivariant for any $0\leq i \leq 2n-2$; in particular, it is $\rowFb$- and $\rvacFb$-equivariant.
\end{lemma}
\begin{proof}
Bearing in mind the factors of $2$ issue, this is straightforward.
\end{proof}

For $1\leq i \leq 2n$, define $\hb_i\colon \calab(\sfb'^n)\to \rr_{>0}$ by
\[\hb_i(\pi) \coloneqq \prod_{i\leq j \leq 2n} \pi(\langle\flip\rangle[i,j]) \cdot \prod_{1 \leq j \leq i} \pi(\langle \flip \rangle[j,i]).\]
So $\hb_i$ is the natural detropicalization of $\hpl_i$. Note that $\hb_i = \hb_{2n+1-i}$.

\begin{cor} \label{cor:b_homo_b}
The statistics $\hb_i\colon \calab_{\kappa}(\sfb'^n)\to \rr_{>0}$, for $1\leq i \leq 2n$, are multiplicatively $\kappa$-mesic for $\rvacAb$.
\end{cor}
\begin{proof}
For $\pi \in \calab(\sfb'^n)$, we have
\[(\down\ibb\down^{-1}\pi) ([i,j]) = \begin{cases} \pi(\langle \flip \rangle[i,j])/2 &\textrm{if $\flip([i,j])=[i,j]$}; \\
 \pi(\langle \flip \rangle[i,j]) &\textrm{if $\flip([i,j])\neq[i,j]$},\end{cases}\]
for all $[i,j] \in \sfa^{2n}$. In any $\hb_i$ there will be exactly one term corresponding to an~$[i,j]$ with $\flip([i,j])=[i,j]$. So the result follows from \cref{thm:homo} and \cref{lem:b_embedding_b}, because the $\frac{1}{2}$ in this term will exactly cancel with the $\frac{1}{2}$ in the $\frac{\kappa}{2}$ of $\calfb_{\kappa/2}(\sfa^{2n})$.
\end{proof}

Even though the embedding technique does not work for $\sfb^n$ at the birational level, we conjecture that the same rowvacuation homomesies continue to hold. Namely, for $1\leq i \leq 2n-1$, define $\hb_i\colon \calab(\sfb^n)\to \rr_{>0}$ similarly to how we did with $\sfb'^n$ above; and define $\hb_*\colon \calab(\sfb^n)\to \rr_{>0}$ as the natural detropicalization of $\hpl_*$.

\begin{conj} \label{conj:b_homo_b}
The statistics $\hb_i\colon \calab_{\kappa}(\sfb^n)\to \rr_{>0}$, for $1\leq i \leq 2n-1$, and $\hb_*\colon \calab_{\kappa}(\sfb^n)\to \rr_{>0}$, are multiplicatively $\kappa$-mesic for $\rvacAb$.
\end{conj}

\section{Some related enumeration} \label{sec:enumeration}

In this section we will consider some enumeration related to the operators we have been studying. Specifically, we will count fixed points of elements of $\langle \lk, \rowA \rangle$ acting on $\cala(\sfa^n)$.  

Recall the poset automorphism $\flip\colon \sfa^n\to\sfa^n$, which is reflection across the vertical axis of symmetry. And recall that \cref{lem:embedding} implies $\rowF^{n+1}=\flip$ for~$\sfa^n$. Since $\down$ evidently commutes with $\flip$, as do essentially all the operators we have considered, we also know that $\rowA^{n+1}=\flip$ for $\sfa^n$. (At the combinatorial level this was conjectured by Panyushev~\cite{panyushev2009orbits} and proved by Armstrong, Stump, and Thomas~\cite{armstrong2013uniform}.) The number of elements of $\cala(\sfa^n)$ fixed by $\flip$, in other words, the number of symmetric Dyck paths in $\dyck_{n+1}$, is well-known to be $\binom{n+1}{\lfloor (n+1)/2 \rfloor}$.

Affirming a conjecture of Bessis and Reiner~\cite{bessis2011cyclic}, Armstrong--Stump--Thomas~\cite{armstrong2013uniform} proved that $\langle \rowA \rangle$ acting on $\cala(\sfa^n)$ exhibits the cyclic sieving phenomenon with the sieving polynomial being the $q$-Catalan polynomial $\cat(n+1;q)$. This means that the numbers of fixed points of elements of $\langle \rowA \rangle$ acting on $\cala(\sfa^n)$ are given by plugging roots of unity into $\cat(n+1;q)$. We will not go into details about the cyclic sieving phenomenon, but let us remark that since $\cat(n+1;q)$ has a nice product formula, the Armstrong--Stump--Thomas result implies the number of fixed points of $\rowA^i$ has a nice product formula for any $i$. The case $i=1$ of their result is just the fact that $\#\cala(\sfa^n)=\cat(n+1)=\cat(n+1;q\coloneqq 1)$; while the case $i=n+1$ recovers the $\flip$ fixed point count, in agreement with $\cat(n+1;q\coloneqq -1)=\binom{n+1}{\lfloor (n+1)/2 \rfloor}$.

Since $\langle \lk, \rowA \rangle$ is a dihedral group, elements of the form $\lk \circ \rowA^i$ and $\lk \circ \rowA^j$ are conjugate whenever $i$ and $j$ have the same parity. So from the point of view of fixed point counts, there are two cases we need to consider: $\lk$ and $\lk \circ \rowA$.

The case $\lk $ was addressed by Panyushev~\cite{panyushev2004adnilpotent}.

\begin{thm}[{Panyushev~\cite[Thm.~4.6]{panyushev2004adnilpotent}}]
\[ \#\{A\in \cala(\sfa^n)\colon \lk (A)=A\} = \begin{cases} 0 &\textrm{if $n$ is odd}; \\ \cat(n/2) &\textrm{if $n$ is even}.\end{cases}\]
\end{thm}

Completing the problem of counting fixed points of $\langle \lk, \rowA \rangle$ acting on $\cala(\sfa^n)$, we now give a formula for the number of fixed points of $\lk \circ \rowA$.

\begin{thm} \label{thm:fix-under-rowLK}
\[ \#\{A\in \cala(\sfa^n)\colon \lk \circ \rowA(A)=A\} =\binom{n+1}{\lfloor (n+1)/2 \rfloor}.\]
\end{thm}

\begin{figure}
    \centering
\begin{tikzpicture}[scale=0.64]
\begin{scope}
\draw[thick] (-0.15, 1.85) -- (-0.85, 1.15);
\draw[thick] (0.15, 1.85) -- (0.85, 1.15);
\draw [thick] (0,2) circle [radius=0.2];
\draw [thick] (-1,1) circle [radius=0.2];
\draw [thick] (1,1) circle [radius=0.2];
\end{scope}
\begin{scope}[shift={(4.5,0)}]
\draw[thick] (-0.15, 1.85) -- (-0.85, 1.15);
\draw[thick] (0.15, 1.85) -- (0.85, 1.15);
\draw [thick] (0,2) circle [radius=0.2];
\draw [thick,fill] (-1,1) circle [radius=0.2];
\draw [thick,fill] (1,1) circle [radius=0.2];
\end{scope}
\begin{scope}[shift={(9,0)}]
\draw[thick] (-0.15, 1.85) -- (-0.85, 1.15);
\draw[thick] (0.15, 1.85) -- (0.85, 1.15);
\draw [thick,fill] (0,2) circle [radius=0.2];
\draw [thick] (-1,1) circle [radius=0.2];
\draw [thick] (1,1) circle [radius=0.2];
\end{scope}
\begin{scope}[shift={(13.5,0)}]
\draw[thick] (-0.15, 1.85) -- (-0.85, 1.15);
\draw[thick] (0.15, 1.85) -- (0.85, 1.15);
\draw [thick] (0,2) circle [radius=0.2];
\draw [thick,fill] (-1,1) circle [radius=0.2];
\draw [thick] (1,1) circle [radius=0.2];
\end{scope}
\begin{scope}[shift={(18,0)}]
\draw[thick] (-0.15, 1.85) -- (-0.85, 1.15);
\draw[thick] (0.15, 1.85) -- (0.85, 1.15);
\draw [thick] (0,2) circle [radius=0.2];
\draw [thick] (-1,1) circle [radius=0.2];
\draw [thick,fill] (1,1) circle [radius=0.2];
\end{scope}
\draw [thick,<->] (8.2,2.5) arc [radius=2.25, start angle=50, end angle=130];
\draw [thick,<->] (17.2,0.5) arc [radius=2.25, start angle=-50, end angle=-130];
\draw [thick,->] (0.14,0.6) arc [radius=0.4, start angle=70, end angle=-250];
\draw [thick,->] (0.14,2.4) arc [radius=0.4, start angle=-70, end angle=250];
\draw [thick,->] (4.64,0.6) arc [radius=0.4, start angle=70, end angle=-250];
\draw [thick,->] (9.14,0.6) arc [radius=0.4, start angle=70, end angle=-250];
\draw [thick,->] (13.64,2.4) arc [radius=0.4, start angle=-70, end angle=250];
\draw [thick,->] (18.14,2.4) arc [radius=0.4, start angle=-70, end angle=250];
\node at (0,-0.57) {$\flip$};
\node at (4.5,-0.57) {$\flip$};
\node at (9,-0.57) {$\flip$};
\node at (15.75,-0.57) {$\flip$};
\node at (0,3.57) {$\LK \circ \rowA$};
\node at (6.75,3.57) {$\LK \circ \rowA$};
\node at (13.75,3.57) {$\LK \circ \rowA$};
\node at (18,3.57) {$\LK \circ \rowA$};
\end{tikzpicture}
\caption{The $\cat(3)=5$ antichains of $\sfa^2$.  There are $3=\binom{3}{1}$ of them fixed by $\flip$, and likewise $3$ of them fixed by $\LK \circ \rowA$.} \label{fig:A2-flip-fix}
\end{figure}
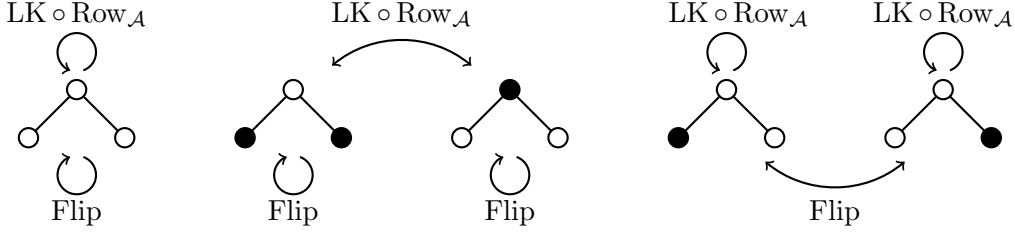

The astute reader may notice that \cref{thm:fix-under-rowLK} says that the number of antichains in $\cala(\sfa^n)$ fixed by $\lk \circ \rowA$ is the same as the number fixed by $\flip$. See \cref{fig:A2-flip-fix} for an illustration of this when $n=2$. Indeed, the way one can prove \cref{thm:fix-under-rowLK} is by showing that $\lk \circ \rowA$ and $\flip$ are conjugate in the antichain toggle group (which of course implies they have the same orbit structure).

Actually, it is more convenient to use use order filter toggles here rather than antichain toggles. It is easier to work with order filter toggles because $\bft_i \bft_j = \bft_j \bft_i$ whenever $|i-j|\not=1$, whereas antichain rank toggles (for different ranks) never commute. At any rate, the key lemma need to prove \cref{thm:fix-under-rowLK} is:

\begin{lemma}\label{lem:conj to flip}
$\rvacF\circ\rowF$ is conjugate to $\flip$ in the order filter toggle group of~$\sfa^n$.
\end{lemma}

For considerations of space, we will not go through all the details of the proof of \cref{lem:conj to flip} here (and anyways it is a relatively straightforward computation); but let us state two propositions which aid in the proof.

\begin{prop}\label{prop:snow way out}
If $k$ and $n$ have the same parity, then $\flip \circ \bft_k$ is conjugate to $\flip$ in the order filter toggle group of $\sfa^n$.
\end{prop}

\begin{prop}\label{prop:dino dash}
For $0\leq k \leq n-1$, define $d_k\colon \calf(\sfa^n)\to \calf(\sfa^n)$ by
\[ d_k \coloneqq (\bft_0 \bft_1 \bft_2 \cdots \bft_k) (\bft_0 \bft_1 \bft_2 \cdots \bft_{k-1})\cdots (\bft_0 \bft_1) (\bft_0).\]
If $k$ has the same parity as $n$, then $\flip\circ d_k$ is conjugate to $\flip$ in the order filter toggle group of $\sfa^n$.
\end{prop}

In summary, for any element of $\langle \lk, \rowA\rangle$ acting on $\cala(\sfa^{n})$, there is a nice product formula for its number of fixed points.

\begin{remark} \label{rem:pl_fixed_points}
For any $m \in \zz_{>0}$, the set $\frac{1}{m}\zz^{\sfp}\cap \calc(\sfp)$ of rational points in the chain polytope with denominator dividing $m$ is a finite set. In fact, a result of Proctor~\cite{proctor1988odd, proctor1990new} (namely, the enumeration of ``plane partitions of staircase shape'') implies that when $\sfp=\sfa^{n}$ the cardinality of this set is
\[\#\frac{1}{m}\zz^{\sfa^{n}}\cap \calc(\sfa^{n})=\prod_{1\leq i\leq j \leq n} \frac{i+j+2m}{i+j}.\]
Observe how the case $m=1$ recaptures the product formula for the Catalan number.

Since $\frac{1}{m}\zz^{\sfa^{n}}\cap \calc(\sfa^{n})$ is preserved by the PL antichain toggles $\tau_p$ (when $\kappa=1$, which we will assume from now on), it caries an action of $\lkpl$ and $\rowApl$. We could therefore ask for formulas counting fixed points of elements of $\langle \lkpl, \rowApl\rangle$ acting on this set of points.

Extending the CSP results of Armstrong--Stump--Thomas, it is conjectured (see \cite[Conj.~4.28]{hopkins2019minuscule}~\cite[Conj.~5.2]{hopkins2020cyclic}) that $\langle \rowApl \rangle$ acting on $\frac{1}{m}\zz^{\sfa^n}\cap \calc(\sfa^n)$ exhibits cyclic sieving with the sieving polynomial being the so-called ``$q$-multi-Catalan number.'' But this remains unproven.

\Cref{lem:conj to flip} extends directly to the PL level (and in fact to the birational level as well). Thus, the number of fixed points of $\rvacApl \circ \lkpl$ acting on $\frac{1}{m}\zz^{\sfa^n}\cap \calc(\sfa^n)$ is the same as the number of points in this set fixed by $\flip$. By another result of Proctor~\cite{proctor1983trapezoid} (enumerating ``plane partitions of shifted trapezoidal shape'') this number is known to be
\[ \#\left\{\pi \in\frac{1}{m}\zz^{\sfa^n}\cap \calc(\sfa^n)\colon \flip(\pi)=\pi\right\} = \hspace{-8pt} \prod_{1\leq i \leq j \leq \lceil n/2 \rceil} \hspace{-8pt} \frac{i+j-1+m}{i+j-1} \hspace{-8pt} \prod_{1\leq i \leq j \leq \lfloor n/2 \rfloor} \hspace{-8pt} \frac{i+j+m}{i+j}.\]

As for counting fixed points of $\lkpl$ acting on $\frac{1}{m}\zz^{\sfa^n}\cap \calc(\sfa^n)$, we have no idea what the answer should be.
\end{remark}

\section{Future directions}  \label{sec:future}

In this final section we discuss some possible future directions for research. 

\subsection{Rowvacuation for other posets} \label{subsec:other_rvac}

Our work here suggests that it may be interesting to study rowvacuation for other graded posets. It especially makes sense to study rowvacuation on those posets which are known to have good behavior of rowmotion. Prominent examples of such posets include the \dfn{minuscule posets} and \dfn{root posets}. As explained in~\cite{hopkins2020order}, work of Grinberg--Roby~\cite{grinberg2015birational2} and Okada~\cite{okada2020birational} implies that rowvacuation of a minuscule poset $\sfp$ has a very simple description in terms of a canonical anti-automorphism of $\sfp$. Meanwhile, in~\cite{defant2021symmetry}, Defant and the second author study rowvacuation for the classical type root posets. Additionally, as discussed in~\cite{hopkins2020order}, there are a handful of other families of posets which have good rowmotion behavior, and it might be worth looking at rowvacuation for these.

\subsection{Combinatorics of the \texorpdfstring{$\sfa^{n}$}{An} into \texorpdfstring{$\rect{n+1}{n+1}$}{[n+1]x[n+1]} embedding}

By tropicalizing the embedding $\iab$, we obtain an embedding 
\[\iapl\colon \calfpl_{\kappa}(\sfa^n)\to\calfpl_{2\kappa}(\rect{n+1}{n+1}).\] 
In particular (with $\kappa=1$) we have $\iapl(\calc(\sfa^n)) \subseteq 2 \cdot \calc(\rect{n+1}{n+1})$. We can ask where the vertices of $\calc(\sfa^n)$ are sent under $\iapl$. In fact, the image of the vertices of $\calc(\sfa^n)$ under $\iapl$ is a set that can be naturally identified with the $321$-avoiding permutations in the symmetric group $\ss_{n+1}$. In this way, $\iapl$ provides a bijection between Dyck paths (i.e., $\cala(\sfa^{n})$) and $321$-avoiding permutations. There are many known such bijections; it is not hard to see that $\iapl$ is precisely the \dfn{Billey--Jockusch--Stanley bijection}~\cite{billey1993combinatorial, callan2007bijections,elizalde2011fixed}. Furthermore, since $\cala(\sfa^{n})$ caries an action of rowmotion, we can use this Billey--Jockusch--Stanley bijection to define an action of rowmotion on the set of $321$-avoiding permutations. Rowmotion on $321$-avoiding permutations is studied in upcoming work of Adenbaum and Elizalde~\cite{adenbaum2021rowmotion}.

\subsection{Invariants}

A natural thing to do when studying any operator is to try to find functions that are invariant under the operator. For example, these invariant functions can separate orbits. However, for the operators studied in dynamical algebraic combinatorics, it is quite hard in practice to find nontrivial invariant functions. This is because, loosely speaking, these operators ``move things around a lot.'' Indeed, this is a major reason there has been so much focus on finding homomesies for these operators: there is a precise sense in which invariant functions and $0$-mesies are dual to one another (see~\cite[\S2.4]{propp2015homomesy} and~\cite{propp2021spectral}).

Nonetheless, there actually \emph{is} a very interesting function on antichains in $\cala(\sfa^n)$ which is invariant under both rowmotion and the Lalanne--Kreweras involution. This invariant appears in a paper of Panyushev~\cite{panyushev2009orbits}, but he attributes it to Oksana Yakimova and calls it the OY-invariant.

\begin{definition}
Let $A$ be an antichain in $\cala(\sfa^n)$ and $F\coloneqq \down^{-1}(A)$ be the order filter generated by $A$. Then we define the \dfn{OY-invariant} $\oy(A)$ of $A$ to be
\[\oy(A) \coloneqq \sum\limits_{e \in A} \big( \#\down(F \setminus  \{e\}) - \#A +1 \big).\]
\end{definition}

\begin{thm}[{\cite[Theorem 3.2, Proposition 3.6]{panyushev2009orbits}}]
The function $\oy\colon \cala(\sfa^n)\to \zz$ is invariant under both $\rowA$ and $\lk$.
\end{thm}

Can $\oy$ be generalized to the birational realm to a function that is invariant under $\rowAb$ and $\lkb$? We will now explain how we think it can.

First, for each $[i,j]\in \sfa^n$ and $A \in \cala(\sfa^n)$, define
\[\oy_{[i,j]}(A)=\begin{cases} \#\down(\down^{-1}(A)\setminus\{[i,j]\}) - \#A + 1 &\textrm{if $[i,j]\in A$},\\
0 &\textrm{if $[i,j]\not\in A$}.\end{cases}\]
Then note that
\[ \oy = \sum_{[i,j]\in \sfa^n} \oy_{[i,j]}.\]
In the next definition we give the detropicalized version of these statistics. For conciseness we omit the proof, but one can show that the map $\oyb_{[i,j]}$ defined below is equivalent to $\oy_{[i,j]}$ when tropicalized and restricted to the combinatorial realm.

\begin{definition}
Let $\pi\in \calab(\sfa^n)$. Define the \dfn{birational OY-invariant}  $\oyb$ of $\pi$ as
\[\oyb(\pi) \coloneqq \prod\limits_{[i,j] \in \sfa^n} \oyb_{[i,j]}(\pi).\]
Here for $[i,j]\in \sfa^n$ we define $\oyb_{[i,j]}(\pi)\coloneqq LR$, with $L$ and $R$ described below.
\begin{itemize}
\item If $i=1$, set $L\coloneqq1$.
If $i\geq 2$, then consider the subposet $\sfpl$ of $\sfa^n$ consisting of elements $e$ such that $e$ is greater than or equal to either $[i-1,i-1]$ or $[i,i]$ AND $e$ is less than or equal to either $[i-1,j-1]$ or $[i,j]$. Also consider the subposet $\sfppl = \sfpl \setminus  \{[i,j]\}$.
Then set
\[ L \coloneqq \frac{\sum_{ \substack{v_1  \lessdot \cdots \lessdot v_{j-i+1} \\ \textrm{ a maximal chain in $\sfpl$} }} \pi(v_1)\pi(v_2)\cdots \pi(v_{j-i+1})} {\sum_{ \substack{v_1  \lessdot \cdots \lessdot v_{j-i+1} \\ \textrm{ a maximal chain in $\sfppl$} }} \pi(v_1)\pi(v_2)\cdots \pi(v_{j-i+1})}.\]
\item If $j=n$, set $R\coloneqq 1$.
If $j\leq n-1$, then consider the subposet $\sfpr$ of $\sfa^n$ consisting of elements $e$ such that $e$ is greater than or equal to either $[j, j]$ or $[j+1, j+1]$ AND $e$ is less than or equal to either $[i,j]$ or $[i+1,j+1]$. Also consider the subposet $\sfppr = \sfpr \setminus  \{[i,j]\}.$
Then set
\[ R \coloneqq \frac{\sum_{ \substack{v_1  \lessdot \cdots \lessdot v_{j-i+1} \\ \textrm{ a maximal chain in $\sfpr$} }} \pi(v_1)\pi(v_2)\cdots \pi(v_{j-i+1})} {\sum_{ \substack{v_1  \lessdot \cdots \lessdot v_{j-i+1} \\ \textrm{ a maximal chain in $\sfppr$} }} \pi(v_1)\pi(v_2)\cdots \pi(v_{j-i+1})}.\]
\end{itemize}
\end{definition}

\begin{example}
Consider the poset $\sfa^3$ with the same generic labeling $\pi \in\calab(\sfa^3)$ as in \cref{fig:lk_example}. Then
\begin{align*}
\oyb(\pi) &= {\color{red}\oyb_{[1,1]}(\pi)} {\color{blue}\oyb_{[2,2]}(\pi)} {\color{green}\oyb_{[3,3]}(\pi)} {\color{purple}\oyb_{[1,2]}(\pi)} {\color{cyan}\oyb_{[2,3]}(\pi)}{\color{brown}\oyb_{[1,3]}(\pi)}\\
&={\color{red}\frac{u+v}{v}} \cdot{\color{blue}\frac{u+v}{u}\cdot\frac{v+w}{w}}\cdot{\color{green}\frac{v+w}{v}}\cdot{\color{purple}\frac{vx+vy+wy}{(v+w)y}}\cdot{\color{cyan}\frac{ux+vx+vy}{(u+v)x}}\cdot{\color{brown} 1}\\
&= \frac{(ux+vx+vy)(vx+vy+wy)(u+v)(v+w)}{uv^2 wxy}.
\end{align*}
\end{example}

\begin{conj} \label{conj:oy}
The function $\oyb\colon \calab(\sfa^n)\to\rr_{>0}$ is invariant under both $\rowAb$ and $\lkb$.
\end{conj}

We have verified \cref{conj:oy} for $n\leq 6$.

\begin{remark}
At the combinatorial level, it appears that $\oy$ is invariant not just under $\rowA$ and $\lk$, but in fact under every antichain rank toggle $\bftau_i$. However, the birational function $\oyb$ is {\bf not} invariant under the birational antichain rank toggles.
\end{remark}

\bibliography{birational_lalanne-kreweras}{}
\bibliographystyle{abbrv}

\end{document}